\documentclass[11pt]{article}
\usepackage[round]{natbib}
\usepackage{amssymb}
\usepackage{amsmath}
\usepackage{subfigure}
\usepackage{fullpage}
\usepackage{amsthm}
\usepackage{graphicx}
\usepackage{bbm}
\usepackage{xcolor}
\usepackage{algorithm} 
\usepackage{bbm}
\usepackage{algpseudocode}
\algnewcommand\algorithmicinput{\textbf{INPUT:}}
\algnewcommand\INPUT{\item[\algorithmicinput]}
\algnewcommand\algorithmicoutput{\textbf{OUTPUT:}}
\algnewcommand\OUTPUT{\item[\algorithmicoutput]}

\usepackage{hyperref}[]
\hypersetup{
    colorlinks=true,
    linkcolor=blue,
    filecolor=magenta,      
    urlcolor=cyan,
      citecolor=blue,
    }

\usepackage{cleveref}
\newsavebox\newcaptionbox\newdimen\newcaptionboxwid

\long\def\@makecaption#1#2{
 \vskip 10pt
        \baselineskip 11pt
        \setbox\@tempboxa\hbox{#1. #2}
        \ifdim \wd\@tempboxa >\hsize
        \sbox{\newcaptionbox}{\small\sl #1.~}
        \newcaptionboxwid=\wd\newcaptionbox
        \usebox\newcaptionbox {\footnotesize #2}
        \else
          \centerline{{\small\sl #1.} {\small #2}}
        \fi}

\def\fnum@figure{Figure \thefigure}
\def\fnum@table{Table \thetable}

\newcommand{\lc}{ \lceil} 
\newcommand{\rc}{ \rceil } 
\newcommand{\kse}{\kappa_{\max}^{s,e}}

\DeclareMathOperator*{\argmax}{arg\,max}

\newtheorem{assumption}{Assumption}
\newtheorem{lemma}{Lemma}
\newtheorem{proposition}{Proposition}

\newtheorem{definition}{Definition}
\newtheorem{remark}{Remark}
\newtheorem{theorem}{Theorem}

\title{Change point detection and inference in multivariable nonparametric models under  mixing conditions}

\author{Carlos Misael Madrid Padilla$^1$ \and Haotian Xu$^2$ \and  Daren Wang$^3$ \and Oscar Hernan Madrid Padilla$^4$ \and  Yi Yu$^5$}
\date{%
    $^1$Department of Mathematics, University of Notre Dame\\%
    $^2$Department of Statistics, Pennsylvania State University\\%
    $^3$Department of Statistics, University of Notre Dame\\%
    $^4$Department of Statistics, University of California\\%
    $^5$Department of Statistics, University of Warwick\\[2ex]%
    \today
}

\begin{document}

\maketitle

\begin{abstract}
This paper studies multivariate nonparametric change point localization and inference problems. The data consists of a multivariate time series with potentially short range dependence. The distribution of this data is assumed to be piecewise constant with densities in a H\"{o}lder class. The change points, or times at which the distribution changes, are unknown. We derive the limiting distributions of the change point estimators when the minimal jump size vanishes or remains constant, a first in the literature on change point settings. We are introducing two new features: a consistent estimator that can detect when a change is happening in data with short-term dependence, and a consistent block-type long-run variance estimator. Numerical evidence is provided to back up our theoretical results.
\end{abstract}

\section{Introduction}
In this paper, we study the problem of change point detection in nonparametric settings. Our model assumes that a vector of measurements is collected at every time point following a distribution that has a probability density function belonging to a  H\"{o}lder function class. The change point assumption implies that the probability density functions remain the same across time except for abrupt changes at the change points. Furthermore, our theory permits temporal dependence of the measurements, a feature not explored in prior work.

To be more specific, the observations 
$\{X_t\}_{t=1}^{T}\subset \mathbb{R}^p$ are assumed to be an $\alpha$-mixing sequence of random vectors with unknown distributions $\{P_t\}_{t=1}^{T}$. The  $\alpha$-mixing coefficients, $\{\alpha_k\}_{k\in\mathbb{Z}}$, have an exponential-decay,
 \begin{equation}
 \label{mix-cond}
     \alpha_k\le e^{-2ck}, \quad k \in \mathbb{Z},
 \end{equation}
for a certain $c>0.$ 
The decay rate of $\alpha_k$ imposes a temporal dependence between events that are separated by $k$ time points, as is stated in \eqref{mix-cond}. This is a standard requirement in the literature \citep[e.g][]{abadi2004sharp,merlevede2009bernstein}. To model the nonstationarity of sequentially observed multivariate data, we assume that there exists $K \in \mathbb{N}$ change points, namely $\{\eta_k\}_{k=1}^{K}\subset \{2,...,T\}$ with  $1=\eta_0<\eta_1< \ldots <\eta_k\le T<\eta_{K+1}=T+1$, such that 
\begin{equation} \label{model1}
    P_t\neq P_{t-1} \ \text{if and only if } t \in \{\eta_1, \ldots ,\eta_K\}.
\end{equation}  
 Our primary interest is to accurately estimate $\{\eta_k\}_{k=1}^K$ and perform inference. We
 refer to \Cref{assume: model assumption} below for detailed technical conditions on the model described by \eqref{mix-cond} and \eqref{model1}.

Nonstationary multivariate data is frequently encountered in real-world applications, including
biology \citep[e.g.][]{molenaar2009analyzing,wolkovich2021phenological}, epidemiology \citep[e.g.][]{azhar2021study,nguyen2021forecasting}, social science \citep[e.g.][]{kunitomo2021robust,cai2022state}, climatology \citep[e.g.][]{corbella2012predicting,heo2022greedy}, finance \citep[e.g.][]{herzel2002non,schmitt2013non}, neuroscience \citep[e.g.][]{frolov2020revealing,gorrostieta2019time}, among others.

Due to the importance of modeling nonstationary data in various scientific fields, this problem has received extensive attention in the statistical change point literature, \citep[e.g.][]{aue2009break,fryzlewicz2014wild,cho2015multiple,cho2016change,wang2020univariate}. 
However, there are a few limitations in the existing works in multivariate nonparametric settings. Firstly, to the best of our knowledge, temporal dependence has not been considered. 
Secondly, there is no consistent result for data with the underlying densities being as general as H\"{o}lder smooth. Lastly, but most importantly, the statistical task of deriving limiting distributions of the change point estimators has reportedly not been treated in the multivariate nonparametric change point 
literature.

Taking into account the aforestated limitations, this paper examines change point problems in a fully nonparametric framework, wherein the underlying distributions are only assumed to have piecewise and H\"{o}lder smooth continuous densities and the magnitudes of the distributional changes are measured by the $L_2$-norm of the differences between the corresponding densities. 

The rest of the paper is organized as follows. 
In \Cref{sec-model}, we explain the model assumptions for multivariate time series with change points in a nonparametric setting. \Cref{sec-methods} details the two-step change point estimation procedure, as well as the estimators at each step. Theoretical results, including the consistency of the preliminary estimator and the limiting distribution of the final estimator, are presented in  \Cref{Theory}. \Cref{sec-numeric} evaluates the practical performance of the proposed procedure via various simulations and a real data analysis. Finally, \Cref{sec-conclusion} concludes with a discussion.
\subsection{Notation}  For any function $f:\ \mathbb{R}^p \to \mathbb R$ and for $1 \leq q < \infty$, define $\|f\|_{L_q} = (\int_{\mathbb{R}^p} |f(x)|^q dx)^{1/q}$ and for $q = \infty$, define $\|f\|_{L_\infty} = \sup_{x\in{\mathbb{R}^p}}\vert f(x)\vert$.  Define     $L_q = \{f: \, \mathbb{R}^p \to \mathbb{R}, \, \|f\|_q<\infty \}$. Moreover, for $q=2$,  define $\langle f,g\rangle_{L_2}=\int_{\mathbb{R}^p}f(x)g(x)dx$  where $f,g:\ \mathbb{R}^p \to \mathbb R$.  For any vector $s = (s_1, \ldots, s_p)^{\top} \in{\mathbb{N}^p}$, define $\vert s\vert = \sum_{i = 1}^p s_i$, $s!=s_1!\cdots s_p!$ and the associated partial differential operator $D^s= \frac{\partial^{\vert s\vert}}{\partial x_1^{s_1}\cdots  \partial x_{p}^{s_p}}$.
For $\alpha>0$, denote $\lfloor\alpha\rfloor $ to be the largest integer   smaller than $\alpha$. For any function $f:\, \mathbb{R}^p \to \mathbb R $ that is 
$ \lfloor\alpha\rfloor $-times continuously differentiable at point $x_0$, denote by $f_{x_0}^\alpha $ its Taylor polynomial of degree $ \lfloor\alpha\rfloor $ at $x_0$, which is defined as $$f_{x_0}^\alpha(x) = \sum_{|s| \le \lfloor\alpha\rfloor  } \frac{(x-x_0)^s  }{s!} D^s f(x_0).$$ For a constant $L>0$, let $\mathcal{H}^{\alpha}(L,\mathbb{R}^p)$ be the set  of functions $f:\, \mathbb{R}^p \to \mathbb R $ such that $f$ is $\lfloor\alpha\rfloor$-times differentiable for all  $  x  \in \mathbb{R}^p$  and satisfy $ |f(x) - f_{x_0}^\alpha(x)  | \le L| x-x_0|^ \alpha$, for all $x, x_0\in \mathbb{R}^p$. Here $|x-x_0| $ is  the Euclidean distance between $x, x_0 \in \mathbb R^p$. In nonparametric statistics literature, $\mathcal{H}^{\alpha}(L,\mathbb{R}^p)$ is often referred to as the  class of H\"{o}lder functions. We refer readers to \citet{rigollet2009optimal} for detailed discussions on H\"{o}lder functions.

A process $\{X_t\}_{t\in{\mathbb{Z}}}$ is said to be $\alpha$-mixing if 
$$\alpha_k=\sup_{t\in{\mathbb{Z}}}\alpha(\sigma(X_s,s\le t),\sigma(X_s,s\ge t+k))\longrightarrow_{k\rightarrow \infty} 0,$$ where
$
\alpha(\mathcal{A}, \mathcal{B})=\sup _{A \in \mathcal{A}, B \in \mathcal{B}}|\mathbb{P}(A \cap B)-\mathbb{P}(A) \mathbb{P}(B)| .
$ for any two $\sigma$-fields $\mathcal{A}$ and $\mathcal{B}$. For two positive sequences $\{a_n\}_{n\in \mathbb N^+ }$ and $\{b_n\}_{n\in \mathbb N ^+ }$, we write $a_n = O(b_n)$ or $a_n\lesssim b_n$, if $a_n\le Cb_n$ with some constant $C > 0$ that does not depend on $n$, and $a_n = \Theta(b_n)$ or $a_n\asymp b_n$, if $a_n = O(b_n)$ and $b_n = O(a_n)$.  For a deterministic or random $\mathbb{R}$-valued sequence $a_n$, write that a sequence of random variable $X_n=O_p(a_n)$, if $\lim _{M \rightarrow \infty} \lim \sup _{n \to \infty} \mathbb{P}(|X_n| \geq M a_n)=0$. Write $X_n=o_p(a_n)$ if $\limsup _{n \to \infty} \mathbb{P}(|X_n| \geq M a_n)=0$ for all $M>0$. The convergences in distribution and probability are respectively denoted by $\stackrel{\mathcal{D}}{\to}$ and $\stackrel{P.}{\to}$.
\section{Model setup}\label{sec-model}

Detailed  assumptions imposed on the model \eqref{model1}  are collected in \Cref{assume: model assumption}.

  \begin{assumption}\label{assume: model assumption}
The data $\{X_t\}_{t=1}^{T}\subset \mathbb{R}^p$ is generated based on model \eqref{model1}, satisfying \eqref{mix-cond}, and
\\
{\bf a.}
For $t = 1, \ldots, T$, the distribution $P_t$ has a  Lebesgue density function
$f_t
:\mathbb{R}^p \rightarrow \mathbb{R}$. With $r, L >0$, we assume that, $f_t \in \mathcal{H}^r(L,\mathcal{X})$, where $\mathcal{X}$ is the union of the supports of all the density functions $f_t$, with bounded Lebesgue measure. 
\\
 {\bf b.} Let $g_t$ be the joint density between $X_1$ and $X_{t+1}.$
It satisfies that 
\begin{equation}
\label{join-d-cond}
    \vert\vert g_{t}\vert\vert_{L_\infty}<\infty.
\end{equation} 
{\bf c.}  
The minimal spacing between two consecutive change points $\Delta = \min_{k = 1}^{ K+1}(\eta_{k} - \eta_{k-1})$ satisfies that $\Delta=\Theta(T)$. 
\\
{\bf d.} For $k\in\{1,...,K\}$, let
\begin{equation}
    \kappa_k=\vert\vert  f_{\eta_k}-f_{\eta_{k+1}}\vert\vert_{L_2}
\end{equation}
be the jump size at the $k$th change point and let
$$
    \kappa=\min_{k=1,...,K}\kappa_k>0.
$$
\end{assumption} 

The minimal spacing $\Delta$ and the minimal jump size $\kappa$ are two key parameters characterizing the change point phenomenon.  
\Cref{assume: model assumption}\textbf{c.} requires that $\Delta  = \Theta(T)$, which is necessary only for our inference results in \Cref{theorem:Final-est} and \Cref{Long-Run-V-Theorem}. Indeed, this condition may appear strong compared to the existing literature on localization, such as \citep{padilla2019optimal,wang2020univariate,padilla2021optimal}. 
For our localization results, we can easily relax this condition to $ \Delta \ll T$, as stated in \cite{padilla2022change}. To achieve this, consider increasing $\mathfrak{K}$ in \Cref{definition:seeded} to broaden the coverage of the seeded intervals in MNSBS, and apply the narrowest over-threshold selection method, as described in Theorem 3 of \citep{kovacs2020seeded}.

\Cref{assume: model assumption}\textbf{d.} characterizes the changes in density functions through the function's $L_2$-norm. A reason to use the $L_2$-norm is that the $L_2$ space has an inner product structure.

Revolving the change point estimators, we are to conduct the estimation and inference tasks.  For a sequence of estimators $\widehat{\eta}_1 < \ldots < \widehat{\eta}_{\widehat{K}} \subset \{1, \ldots, T\}$, we are to show their consistency, i.e.~with probability tending to one as the sample size $T$ grows unbounded, it holds that 
\begin{equation}
\label{cons-1}
    \widehat{K} = K \ \text{and} \ \underset{k=1, \ldots,  \widehat{K}}{\max}\vert  \widehat{\eta}_k - \eta_{k}  \vert \leq \epsilon, \ \text{with} \ \underset{T \rightarrow \infty }{\lim} \frac{\epsilon}{\Delta} = 0.
\end{equation}
 We refer to $\epsilon$ as the localization error in the rest of this paper.  

 With a consistent estimation result, we further refine $\{\widehat{\eta}_k\}_{k = 1}^{\widehat{K}}$ and obtain $\{\widetilde{\eta}_k\}_{k = 1}^{\widehat{K}}$ satisfying that $\left|\widetilde{\eta}_k-\eta_k\right|=O_p(1)$.  We are to derive the limiting distribution of 
 \[
 	( \widetilde{\eta}_k- \eta_k)\kappa_k^{\frac{p}{r}+2}, \quad T \to \infty. 
 \]
\subsection{Summary of the results}
The contributions of this paper are as follows.
\begin{itemize}
	\item We develop a multivariate nonparametric seeded change point detection algorithm detailed \Cref{algorithm:WBS}, which is based on the seeded binary segmentation method (SBS), proposed in  \cite{kovacs2020seeded}, for the univariate Gaussian change in mean setup. As suggested in \cite{kovacs2020seeded}, SBS may be adaptable to a wide range of change point detection problems, such as that found in \cite{padilla2022change} for Functional data. We have innovatively adapted SBS to the multivariate nonparametric setup.
    \item 
Under the model assumptions outlined in \Cref{assume: model assumption} and the signal-to-noise ratio condition in \Cref{assume-snr} that $\kappa^2 \Delta  \gtrsim\log(T)T^{\frac{p}{2r+p}}$, we demonstrate that the output of \Cref{algorithm:WBS} is consistent, with a localization error of $\kappa_k^{-2} T^{\frac{p}{2r+p}} \log (T)$, for $k \in \{1, \ldots, K\}$. We note that this localization error was obtained under the temporal dependence stated in \eqref{mix-cond} and with a more general smoothness assumption outlined in \Cref{assume: model assumption}\textbf{a.}, which is a novel contribution to the literature.
	\item Based on the consistent estimators $\{\widehat{\eta}\}_{k=1}^{\widehat{K}}$, we construct refined estimators $\{\widetilde{\eta}_k\}_{k=1}^{\widehat{K}}$ and derive their limiting distributions in different regimes, as detailed in \Cref{theorem:Final-est}. This result is novel in the literature of nonparametric temporal dependence models, and such two-regime limiting distributions are rarely seen in the literature, with the exception of mean change under fixed-dimensional time series \citep[e.g.][]{yao1987approximating, yao1989least, bai1994least}, high-dimensional vector time series \citep[e.g.][]{kaul2021inference}, functional time series setting \citep[e.g.][]{aue2009estimation}, and high-dimensional linear regression \citep[e.g.][]{xu2022change}.
	\item Extensive numerical results are presented in \Cref{sec-numeric} to corroborate the theoretical findings. 
    The code used for numerical experiments is available upon request prior to publication. If the paper is accepted, we will include the code and instructions on how to reproduce the numerical results in Section \ref{sec-numeric}.
\end{itemize}
\section{Multivariate nonparametric seeded change point estimators and their refinement} \label{sec-methods}
In this section, we present the initial and refined change point estimators, both of which share the same building block, namely CUSUM statistics, defined in \Cref{Cus-Sta}.

\begin{definition}[CUSUM statistics]\label{Cus-Sta}
 For any integer triplet $0 \le s < t < e \le T$, let the CUSUM statistic be 
\begin{align*}
 \widetilde F_{t, h }^{(s,e] } (x) &= \sqrt { \frac{e-t}{ (e-s )(t-s)}}\sum_{i=s+1}^{t}F_{i,h}(x)
  -\sqrt { \frac{t-s}{ (e-s )(e-t)}}\sum_{i=t+1}^{e}F_{i,h}(x), \ x\in{\mathbb{R}^p} 
\end{align*} 
where $F_{t, h}(\cdot)$ is a kernel estimator, $F_{t, h} (x) = \mathcal{K}_h(x-X_{t}), \ x\in{\mathbb{R}^p}$
with the kernel function 
\[
     \mathcal{K}_h (x ) = \frac{1}{h  ^p  } \mathcal{K} \bigg  (\frac{x}{h} \bigg ),\quad x\in{\mathbb{R}^p},
\]
accompanied with the bandwidth $h>0.$
\end{definition}
The CUSUM statistic is a key ingredient of our algorithm and is based on the kernel estimators $F_{t, h}(\cdot)$.   We would like to highlight that kernel-based change-point estimation techniques have been employed in the detection of change-points in nonparametric models in existing literature, as demonstrated in \citep[e.g.][]{ arlot2019kernel, li2019scan, padilla2021optimal}.

Our preliminary estimator is based on SBS. Such an estimator is obtained by combining the CUSUM statistic in \Cref{Cus-Sta} with a modified version of SBS, which is based on a collection of deterministic intervals defined in \Cref{definition:seeded}.


\begin{definition} [Seeded intervals]\label{definition:seeded}
Let $\mathfrak{K} = \lc C_{\mathfrak{K}}\log(T) \rc$, with some sufficiently large absolute constant $C_{\mathfrak{K}} > 0$.  For $k \in \{1, \ldots, \mathfrak{K}\}$, let $\mathcal{J}_k$ be the collection of $2^k - 1$ intervals of length $l_k = T2^{-k+1}$ that are evenly shifted by $l_k/2 = T2^{-k}$, i.e. 
\begin{align*}
    \mathcal{J}_k =& \{(\lfloor (i -1) T2^{-k} \rfloor, \, \lc (i-1) T2^{-k} + T2^{-k+1}\rc ],
    \\
    &\quad i = 1, \ldots,  2^k - 1 \}.
\end{align*}
The overall collection of seeded intervals is denoted as $\mathcal{J} = \cup_{k = 1}^{\mathfrak{K}} \mathcal{J}_k$.
\end{definition}

With the CUSUM statistics and the seeded intervals as building blocks, we are now ready to present our multivariate nonparametric seeded change point detection algorithm.
 \begin{algorithm}[ht]
	\begin{algorithmic}
		\INPUT Sample $\{X_t\}_{t=s}^{e} \subset \mathbb{R}^p$, collection of seeded intervals $\mathcal{J}$, tuning parameter $\tau > 0$ and bandwidth $h > 0$.
  \State \textbf{initialization}: If $(s,e]=(0,n]$, set $\textbf{S} \rightarrow \varnothing$ and set $\rho \rightarrow \log(T)h^{-p}$.  
    
		\For{$\mathcal I=(\alpha, \beta ] \in \mathcal{J}$}  
		\If{$ \mathcal I= (\alpha, \beta ] \subseteq  (s,e]  $ and $\beta-\alpha >2\rho $}
		\State $b_{\mathcal I} \leftarrow \argmax_{\alpha +\rho \leq t \leq \beta - \rho}   \vert\vert \widetilde F^{(\alpha, \beta]}_{t,h}\vert\vert_{L_2}$ 
		\State $a_{\mathcal I} \leftarrow \vert\vert \widetilde F^{(\alpha, \beta]}_{b_{\mathcal I},h}\vert\vert_{L_2}$
		\Else
		\State $a_{\mathcal I} \leftarrow -1$	
		\EndIf
		\EndFor
		\State $\mathcal{I}^* \leftarrow \argmax_{\mathcal I\in{\mathcal{J}}} a_{\mathcal I}$
		\If{$a_{\mathcal I^*} > \tau$}
		\State ${\bf S} \leftarrow {\bf S} \cup \{b_{r^*}\} $
		\State MNSBS$((s, b_{\mathcal{I}*}), \mathcal{J}, \tau,h)$
		\State MNSBS$((b_{\mathcal {I}*}+1,e), \mathcal{J}, \tau,h)$
		\EndIf  
		\OUTPUT The set of estimated change points $\textbf{S}$.
		\caption{Multivariate Nonparametric Seeded Binary Segmentation. MNSBS $((s, e), \mathcal{J}, \tau, h)$}
		\label{algorithm:WBS}
	\end{algorithmic}
\end{algorithm} 

\Cref{algorithm:WBS} is proposed as a preliminary estimator for multiple change points in sequentially observed multivariate time series data. It takes advantage of seeded intervals to provide a multi-scale search system and recursively uses CUSUM statistics to identify potential change points. Inputs required are observed data $\{X_t\}_{t=1}^T$, seeded intervals $\mathcal{J}$, bandwidth $h$ for constructing the CUSUM statistics, and threshold $\tau$ for detecting change points. Theoretical and numerical guidance for tuning parameters is presented in \Cref{Theory,sec-numeric}.

Denote by $\{\widehat{\eta}_k\}_{k = 1}^{\widehat{K}}$ our preliminary estimators provided by \Cref{algorithm:WBS}. It has been demonstrated in various studies, such as \citep{rinaldo2021localizing,xu2022change,yu2022localising}, that a refinement procedure can likely reduce the localization error of preliminary estimates of change points. Thus, a refinement step is proposed. First, let 
\begin{equation}
	\label{final-int}
	s_k=\frac{9}{10}\widehat{\eta}_{k-1}+\frac{1}{10}\widehat{\eta}_{k} \ \text{and} \ e_k=\frac{9}{10}\widehat{\eta}_{k+1}+\frac{1}{10}\widehat{\eta}_{k}.
\end{equation}
Then, $\{\widehat{\eta}_k\}_{k = 1}^{\widehat{K}}$ and $\widetilde{h}\asymp h$ produce an estimator of $\kappa_k$ as:
\begin{equation}
\label{kappa-hat}
    \widehat{\kappa}_k=\Big|\Big|\frac{ \sqrt{\frac{\widehat{\eta}_{k+1}-\widehat{\eta}_k}{(\widehat{\eta}_{k+1}-\widehat{\eta}_{k-1})(\widehat{\eta}_k-\widehat{\eta}_{k-1})}}\sum_{i=\widehat{\eta}_{k-1}+1}^{\widehat{\eta}_k}F_{i,\widetilde{h}}-\sqrt{\frac{(\widehat{\eta}_k-\widehat{\eta}_{k-1})}{(\widehat{\eta}_{k+1}-\widehat{\eta}_{k-1})(\widehat{\eta}_{k+1}-\widehat{\eta}_k)}}\sum_{i=\widehat{\eta}_k+1}^{\widehat{\eta}_{k+1}}F_{i,\widetilde{h}}}{\sqrt{\frac{(\widehat{\eta}_k-\widehat{\eta}_{k-1})(\widehat{\eta}_{k+1}-\widehat{\eta}_k)}{\widehat{\eta}_{k+1}-\widehat{\eta}_{k-1}}}}
    \Big|\Big|_{L_2}.
\end{equation}
 We then propose the final change points estimators,
\begin{equation}
\label{Final-est}
    \widetilde{\eta}_k=\underset{s_k<\eta<e_k}{\arg \min } \ \widehat{Q}_k(\eta)
    =\underset{s_k<\eta<e_k}{\arg \min }\Big\{\sum_{t=s_k+1}^{\eta}\vert\vert F_{t,h_1}-F_{(s_k,\eta_k],h_1}\vert\vert_{L_2}^2 
    +\sum_{t=\eta+1}^{e_k}\vert\vert F_{t,h_1}-F_{(\eta_k,e_k],h_1}\vert\vert_{L_2}^2\Big\},
\end{equation}
where, $h_1=c_{\widehat{\kappa}_k}\widehat{\kappa}_k^{\frac{1}{r}},$ $$F_{(s_k,\eta_k],h_1}=\frac{1}{\eta_k-s_k}\sum_{i=s_k+1}^{\eta_k}F_{i,h_1} \ \text{and} F_{(\eta_k,e_k],h_1}=\frac{1}{e_k-\eta_k}\sum_{i=\eta_k+1}^{e_k}F_{i,h_1}.$$

If the initial change point estimators are consistent, i.e.~\eqref{cons-1}, 
then with arbitrary choice of the constants $9/10$ and $1/10$ in \eqref{final-int}, $(s_k,e_k)$ should contain only one true change point, $\eta_k$. Motivated by \cite{padilla2021optimal}, where it was demonstrated that $\kappa_k$ is a near minimax optimal bandwidth, we use $\widehat{\kappa}_k$ as bandwidth for the kernel density estimator in \eqref{Final-est} and search for a change point in the smaller interval $(s_k,e_k),$ aiming for a better estimate of $\eta_k$. 

We observe in practice and theory that the local refinement step improves estimation. Our experiments in \Cref{sec-numeric} support this. In addition, \Cref{theorem:Final-est} shows an improved error rate for $\vert  \widetilde{\eta}_k  - \eta_k\vert$ over the original estimator, and also studies the limiting distribution of $\widetilde{\eta}_k$.
 
The intuition behind the construction of $\widehat{\kappa}_k$ is that  the numerator in \eqref{kappa-hat} is, up to a normalizing factor,  the difference of kernel density estimators of $ f_{\eta_k}$  and $f_{\eta_{k+1}}$. The denominator in \eqref{kappa-hat} is just a normalizing term.

With regards to the choices of tuning parameters in practice, see our discussion in \Cref{sec-comparison} and \Cref{sec-numeric}. 

The computational complexity of \Cref{algorithm:WBS}, i.e.~the preliminary estimators, is of order $O(T\log(T) \cdot \mathrm{Kernel})$, where $O(T\log(T))$ is due to the computational cost of the Seeded Binary Segmentation and ``kernel'' stands for the computational cost of numerical computation of the $L_2$-norm of the CUSUM statistics based on the kernel function evaluated at each time point. The dependence on the dimension $p$ is only through the evaluation of the kernel function. The computational complexity of the final estimators (including estimating $\widehat{\kappa}_k$'s) is of order $O(T \cdot \mathrm{Kernel})$.
Therefore, the overall cost for finding $\{\widetilde{\eta}_k\}_{k=1}^{ \widehat{K} }$ is of order $O(T\log T \cdot \mathrm{Kernel})$. 
\section{Consistent estimation and limiting distributions}\label{Theory}
Recalling that the core of our estimators defined in \Cref{sec-methods} is a kernel estimator, 
we first state conditions needed for the kernel function $\mathcal{K}(\cdot)$.
 
 \begin{assumption}[The kernel function] \label{kernel-as}
Assume that the kernel function $\mathcal{K}(\cdot)$ has compact support and satisfies the following.

\
\noindent {\bf a.}  For the H\"{o}lder smooth parameter $r$ in \Cref{assume: model assumption}{\bf a}, assume that $\mathcal{K}(\cdot)$ is adaptive to $\mathcal{H}^{r}(L,\mathbb{R}^p)$, i.e.~for any $f \in \mathcal{H}^{r}(L,\mathbb{R}^p)$, it holds that  
\[
    \sup_{x \in{\mathbb{R}^p}} \Big|\int_{\mathbb{R}^p} h^{-p} \mathcal{K}\Big(\frac{x-z}{h}\Big) f(z) \,\mathrm{d}z - f(x)\Big| \leq \widetilde{C} h^r,
\]
for some absolute constant $\widetilde{C} > 0$.

\
\noindent{\bf b.} The class of functions 
$\mathcal{F}_\mathcal{K} = \{\mathcal{K}(x-\cdot)/h:\,   \mathbb{R}^p \to $ $ \mathbb R^+, h > 0\}$ is separable  in $L_{\infty}(\mathbb{R}^d)$ and is a  uniformly bounded VC-class; i.e.~there exist constants $A, \nu > 0$  such that for any probability measure $Q$ on $\mathbb{R}^p$ and any $u \in (0, \|\mathcal{K}\|_{{L_\infty}})$, it holds that $\mathcal{N}(\mathcal{F}_\mathcal{K},L_2(Q),u) \leq ({A\|\mathcal{K}\|_{L_\infty}}/{u})^{v}$, where $\mathcal{N}(\mathcal{F}_\mathcal{K},  L_2(Q) ,u)$ denotes the $u$-covering number of the metric space $(\mathcal{F}_\mathcal{K},  L_2(Q))$. 
 
\
\noindent{\bf c.} For fixed $m>0$, it holds that
  $$\int_{0}^{\infty}t^{m-1}\sup_{\vert\vert x\vert\vert\ge t}\vert \mathcal{K}(x)\vert^m \, \mathrm{d}t<\infty, \
 \int_{\mathbb{R}^d}\mathcal{K}(z)\vert\vert z\vert\vert \,\mathrm{d}z\le C_K,$$
  where $C_K > 0$ is an absolute constant.
\end{assumption} 

\Cref{kernel-as} is a standard result in the nonparametric literature \citep[e.g.][]{gine1999laws, gine2001consistency,  sriperumbudur2012consistency, kim2019uniform,padilla2021optimal}, holding for various kernels, such as uniform, Epanechnikov and Gaussian.
 
\subsection{Consistency of preliminary estimators}

The consistency of the preliminary estimators outputted by \Cref{algorithm:WBS} holds provided the signal strength is large enough.  This is detailed in the following assumption.

\begin{assumption}[Signal-to-noise ratio]\label{assume-snr}
Assume there exists an arbitrarily slow diverging sequence $\gamma_T>0$ such that
\[
    \kappa^2 \Delta >\gamma_T \log(T)T^{\frac{p}{2r+p}}.
\]
\end{assumption}
Signal-to-noise ratio condition of such a form is ubiquitous in the change point detection literature, but dealing with H\"{o}lder smooth function, nonparametric, and short range temporal dependence is novel. Combined with \Cref{assume: model assumption}\textbf{c.}, the SNR in \Cref{assume-snr} is reduced to $\kappa^2  \gtrsim \gamma_T \log(T)T^{-\frac{2r}{2r+p}}$.

We are now ready to present the main theorem concerning \Cref{algorithm:WBS}, showing the consistency of the proposed MNSBS.

\begin{theorem}\label{theorem:FSBS}
Let the data be $\{X_t\}_{t = 1}^T \subset \mathbb{R}^p$ satisfying \Cref{assume: model assumption}.  Let $\{\widehat \eta_k\}_{k=1}^{\widehat K}$ be the estimated change points by MNSBS detailed in \Cref{algorithm:WBS}, with inputs $\{X_t\}$, tuning parameter $\tau=c_\tau T^{\frac{p}{4r+2p}} \log^{\frac{1}{2}} (T)$ and bandwidth $
h=c_h T^{\frac{-1}{2r+p}}$, with $c_h, c_{\tau} > 0$ being absolute constants.  Under \Cref{kernel-as} and \Cref{assume-snr}, it holds that,
\begin{align*}
 \mathbb{P}\Big\{\widehat{K}=K, \ \Big|\widehat{\eta}_k-\eta_k\Big| \leq C_\epsilon \kappa_k^{-2} T^{\frac{p}{2r+p}} \log (T),
\forall k=1, \ldots, K\Big\} 
\geq 1 -  3C_{p,Ker}T^{-1}, \nonumber
\end{align*}
where, $C_\epsilon > 0$, and $C_{p,Ker}>0$ depending on the kernel and the dimension $p$, are absolute constant. 
\end{theorem}

\subsection{Limiting distributions based on refined estimators}

As for the refined estimators $\{\widetilde{\eta}_k\}_{k = 1}^{\widehat{K}}$ defined in \eqref{Final-est}, to derive limiting distributions thereof, we require the stronger signal-to-noise ratio condition below.



\begin{assumption}[Stronger signal-to-noise ratio]\label{assume-snr2}
 Assume that there exists an arbitrarily slow diverging sequence $\gamma_T>0$ such
\[
    \kappa^{\frac{3p}{r}+3} \Delta > \gamma_T \log(T)T^{\frac{p}{2r+p}}.
\]
\end{assumption}

\Cref{assume-snr2} is strictly stronger than \Cref{assume-snr}, since deriving limiting distributions is usually a more challenging task than providing a high-probability estimation error upper bound, see for example \cite{xu2022change}.

Next, we state our main result of this subsection concerning the estimator  \eqref{Final-est}.

\begin{theorem}
\label{theorem:Final-est}
 Given data $\{X_t\}_{t=1}^T$, suppose that \Cref{assume: model assumption}, \Cref{kernel-as}, and \Cref{assume-snr} hold. Let $\{\widetilde{\eta}_k\}_{k=1}^{\widehat{K}}$ be the change point estimators defined in \eqref{Final-est}, with
 \begin{itemize}
     \item the intervals $\{(s_k, e_k)\}_{k=1}^{\widehat{K}}$ defined in \eqref{final-int};
     \item  the preliminary estimators $\{\widehat{\eta}_k\}_{k=1}^{\widehat{K}}$ from $\operatorname{MNSBS}\Big( \mathcal{J}, \tau, h\Big)$ detailed in \Cref{algorithm:WBS};
     \item the MNSBS tuning parameters $\tau=c_\tau T^{\frac{p}{4r+2p}} \log^{\frac{1}{2}} (T)$ and $h=c_h T^{\frac{-1}{2r+p}}$;
    \item and $\widehat{\kappa}_k$ as in \eqref{kappa-hat}.
 \end{itemize}
{\bf{a.}} (Non-vanishing regime) For $k \in\{1, \ldots, K\}$, if $\kappa_k \rightarrow \varrho_k$, as $T \rightarrow \infty$, with $\varrho_k>0$ being an absolute constant, then the following results hold.
\\
{\bf{a.1.}} The estimation error satisfies that $\Big|\widetilde{\eta}_k-\eta_k\Big|=O_p(1)$, as $T \rightarrow \infty$.
\\
{\bf{a.2.}} When $T \rightarrow \infty$,
\begin{equation}
    ( \widetilde{\eta}_k- \eta_k)\kappa_k^{\frac{p}{r}+2} \ \underrightarrow{\mathcal{D}} \ 
    \underset{\widetilde{r}\in{\mathbb{Z}}}{\arg \min } P_k(\widetilde{r}),
\end{equation}
where
\begin{equation}
    P(\widetilde{r})= \begin{cases}
    \sum_{t=\widetilde{r}+1}^0 2\Big\langle F_{t,h_2}-f_t*\mathcal{K}_{h_2}, 
    (f_{\eta_{k}}-f_{\eta_{k+1}})*\mathcal{K}_{h_2}\Big\rangle_{L_2}
    +\widetilde{r}\vert\vert  (f_{\eta_{k+1}}-f_{\eta_{k}})*\mathcal{K}_{h_2}\vert\vert_{L_2}^2 & \text { if } \widetilde{r}<0 \\ 0 & \text { if } \widetilde{r}=0 \\  \sum_{t=1}^{\widetilde{r}} 2\Big\langle F_{t,h_2}-f_t*\mathcal{K}_{h_2}, 
    (f_{\eta_{k+1}}-f_{\eta_{k}})*\mathcal{K}_{h_2}\Big\rangle_{L_2}
    +\widetilde{r}\vert\vert  (f_{\eta_{k+1}}-f_{\eta_{k}})*\mathcal{K}_{h_2}\vert\vert_{L_2}^2 & \text { if } \widetilde{r}>0\end{cases}
\end{equation}
with $*$ denoting convolution and $h_2=c_{\kappa_k} \kappa_k^{\frac{1}{r}}.$
\\
{\bf{b.}} (Vanishing regime) For $k \in\{1, \ldots, K\}$, if $\kappa_k \rightarrow 0$, as $n \rightarrow \infty$, then the following results hold.
\\
{\bf{b.1.}} The estimation error satisfies that $\Big|\widetilde{\eta}_k-\eta_k\Big|=O_p(\kappa_k^{-2-\frac{p}{r}})$, as $T \rightarrow \infty$.
\\
{\bf{b.2.}} When $T \rightarrow \infty$,
\begin{equation}
\label{limit-dis}
    ( \widetilde{\eta}_k- \eta_k)\kappa_k^{\frac{p}{r}+2} \ \underrightarrow{\mathcal{D}} \ 
    \underset{\widetilde{r}\in{\mathbb{Z}}}{\arg \min } \ \widetilde{\sigma}_{\infty}(k)B(\widetilde{r})+\vert \widetilde{r}\vert,
\end{equation}
with $*$ denoting convolution and $h_2=c_{\kappa_k}\kappa_k^{\frac{1}{r}}$. Here
\begin{equation}
    B(\widetilde{r})= \begin{cases}B_1(-\widetilde{r}) & \text { if } \widetilde{r}<0 \\ 0 & \text { if } \widetilde{r}=0 \\  B_2(\widetilde{r}) & \text { if } \widetilde{r}>0\end{cases}
\end{equation}
and
\begin{align}
\label{long-run-var}
    \widetilde{\sigma}_{\infty}^2(k)=\lim_{T\rightarrow \infty} \frac{\kappa_k^{\frac{p}{r}-2}}{T}Var&\Big(\sum
_{t=1}^{T} \Big\langle F_{{t},h_2}-f_t*\mathcal{K}_{h_2}, 
(f_{\eta_{k}}-f_{\eta_{k+1}})*\mathcal{K}_{h_2}\Big\rangle_{L_2}\Big)
\end{align}
with $B_1(r)$ and $B_2(r)$ being two independent standard Brownian motions.

\end{theorem}

\Cref{theorem:Final-est} considers two regimes of the jump sizes: vanishing and  non-vanishing. Notably, the upper bounds in these regimes on the localization error 
can be written as $$\max_{1 \leq k\leq K} \vert \widetilde{\eta}_k- \eta_k\vert\kappa_k^{\frac{p}{r}+2} = O_{p}(1).$$  Therefore, for $r =1$ our final estimator $\{\widetilde{\eta}_k\}$ attains a minimax optimal rate of convergence, see Lemma 3 in \cite{padilla2021optimal}. Furthermore, in the  setting $r=1$ and $\Delta = \Theta(T)$, our resulting rate is shaper than that in Theorem 1 in \cite{padilla2021optimal}, as we are able to remove the logarithmic factors from the upper bound. Additionally, our method can achieve optimal rates with choices of tuning parameters that do not depend on $\kappa$.

Comparing \Cref{theorem:Final-est} with  \Cref{theorem:FSBS}, we observe an improvement in the localization error, as \Cref{theorem:FSBS} showed $\max_{k=1,...,\widehat{K}}\kappa_k^{2}\Big|\widehat{\eta}_k-\eta_k\Big| \leq C_\epsilon  T^{\frac{p}{2r+p}} \log (T).$

Finally, we highlight that \Cref{theorem:Final-est} summarizes our derivations of the  limiting distributions associated with the $\{\widetilde{\eta}_k\}_{k=1}^{\widetilde{K}}$ estimators.  In the non-vanishing case, the resulting limiting distribution can be approximated
by a two-sided random walk distribution. In contrast, in the vanishing case, the mixing central limit theorem
leads to the two-sided Brownian motion distribution in the limit.

The limiting distributions in \Cref{theorem:Final-est} quantify the asymptotic uncertainty of $\{\widetilde{\eta}_k\}_{k=1}^{\widetilde{K}}$, enabling inference on change point locations, such as constructing confidence intervals. This is especially interesting in the vanishing regime, where the estimating error $|\widetilde{\eta}_k - \eta_k|$ diverges, as shown in \Cref{theorem:Final-est}\textbf{b.1.}. Thus, in the vanishing regime, the limiting distribution can be used to quantify the uncertainty of our change point estimator. Our result \Cref{theorem:Final-est}\textbf{a}. also shows that, in the non-vanishing regime, change points can be accurately estimated within a constant error rate. 

\subsubsection{Consistent long-run variance estimation}
Next, we discuss aspects of practically performing inference on change point locations using $\{\widetilde{\eta}_k\}_{k=1}^{\widetilde{K}}$. It is crucial to access consistent estimators of the long-run variances $\{\widetilde{\sigma}_{\infty}(k)\}$ for the limiting distribution in \Cref{theorem:Final-est}\textbf{b.2.} In this subsection, we propose a block-type long-run variance estimator and derive its consistency (\Cref{algorithm:2}).

\begin{theorem}
\label{Long-Run-V-Theorem}
 Under \Cref{assume: model assumption}, \Cref{kernel-as}, \Cref{assume-snr} and with all the notation in \Cref{theorem:Final-est}, let $\Big\{\widetilde{\sigma}_{\infty}^2(k)\Big\}_{k=1}^{\widehat{K}}$ be as in \eqref{long-run-var} and $R=O(T^{\frac{p+r}{2r+p}}/\kappa_k^{\frac{p}{2r}+\frac{3}{2}})$. We have that
$$
\max _{k=1}^K\Big|\widehat{\sigma}_{\infty}^2(k)-\widetilde{\sigma}_{\infty}^2(k)\Big| \stackrel{P }{\longrightarrow} 0, \quad T \rightarrow \infty 
$$
with
$\widehat{\sigma}_{\infty}^2(k)$ the output of \Cref{algorithm:2}.
\end{theorem}

 \begin{algorithm}[ht]
	\begin{algorithmic}
		\INPUT $\{ X_t\}_{t=1}^T,\{\widehat{\eta}_k\}_{k=1}^{\widehat{K}},\{\widehat{\kappa}_k\}_{k=1}^{\widehat{K}},\{(s_k, e_k)\}_{k=1}^{\widehat{K}}$ and tuning parameter $R \in \mathbb{N}$    
		\For{$k = 1, \ldots, \widehat{K}$}  
    \State{Let $h_1=c_{\widehat{\kappa}}\widehat{\kappa}_k^{\frac{1}{r}}$}
            \For{$t\in\{ s_k, \ldots, e_k-1\}$} 
            \State {
        $Y_t=\widehat{\kappa}_k^{\frac{p}{2r}-1}\Big\langle F_{{t},h_1}-f_t*\mathcal{K}_{h_1}, (f_{\widehat{\eta}_{k}}-f_{\widehat{\eta}_{k+1}})*\mathcal{K}_{h_1}\Big\rangle_{L_2}$
            }
            \EndFor
            \State {$S=\lfloor \frac{e_k-s_k}{R}\rfloor$}
            \For{$r\in\{1,\ldots,R\}$}
            \State{$\mathcal{S}_r=\{s_k+(r-1)S,\ldots,s_k+rS-1\}$}
            \EndFor
            \State{$\widehat{\sigma}_{\infty}^2(k)=\frac{1}{R}\sum_{r=1}^{R}\Big(\frac{1}{\sqrt{S}}\sum_{i\in{\mathcal{S}_r}}Y_i\Big)^2$}
		\EndFor
            \OUTPUT $\{\widehat{\sigma}_{\infty}^2(k)\}_{k=1}^{\widehat{K}}$.
		\caption{Long-run variance estimators}
		\label{algorithm:2}
	\end{algorithmic}
\end{algorithm} 

\subsection{Discussions on multivariate nonparametric seeded change point\\ detection~(MNSBS)}\label{sec-comparison}

\textbf{Tuning parameters.} Our procedure comprises three steps: (1) preliminary estimation, (2) local refinement, and (3) confidence interval construction, with three key tuning parameters.

For step (1), we specify the density estimator of the sampling distribution to be a kernel estimator with bandwidth $h \asymp T^{-1/(2r+p)}$, which follows from the classical nonparametric literature \citep[e.g.][]{yu1993density,Tsybakov:1315296}. The threshold tuning parameter $\tau$ is set to a high-probability upper bound on the CUSUM statistics when there is no change point, of the form $\tau = C_\tau \log^{1/2}(T) \sqrt {h^{-p}}$, which reflects the requirement on the SNR detailed in \Cref{assume-snr}, $\kappa\sqrt{\Delta} \gtrsim \tau$.

The bandwidth $h_{1}$ satisfying $h_{1}\asymp \widehat{\kappa}_k^{\frac{1}{r}}$, inspired by the near minimax rate-optimal bandwidth choice in \cite{padilla2021optimal}, is chosen for refined estimation in step (2).

The rest of the tuning parameters are the bandwidth $\widetilde{h}\asymp h$ for estimating $\widehat{\kappa}_k$ in both steps (2) and (3), and $\widehat{\sigma}^2_{\infty}$ and the number of blocks $R$ for long-run variance estimation in (3), which are specified in \Cref{algorithm:2}.

\textbf{Related work.} 
A comparison with \cite{padilla2021optimal} is presented: the H\"{o}lder condition is milder than their Lipschitz assumption; \Cref{assume: model assumption}\textbf{d} specifies changes through $L_2$-norm of probability density functions, weaker than $L_{\infty}$ distance when $\mathcal{X}$ is compact used in \cite{padilla2021optimal}; our assumptions allow for some dependence captured by $\alpha$-mixing coefficients, unlike \cite{padilla2021optimal} who assume independent observations.

Let us compare \Cref{theorem:FSBS} with Theorem 1 in \cite{padilla2021optimal}: recall the consistency definition stated at \eqref{cons-1}, and in view of \Cref{assume-snr} and \Cref{theorem:FSBS}, we see that with properly chosen tuning parameters and probability tending to one as $T$ grows,
\[
    \max_{k = 1 }^K  {|\widehat{\eta}_k - \eta_k|}/{\Delta} \lesssim { T^{\frac{p}{2r+p}}  \log(T)}/{(\kappa^2 \Delta) } =  o(1),
\]
where the equality follows from \Cref{assume-snr}. This yields the localization consistency guarantee. 
Theorem 1 in \cite{padilla2021optimal} also establishes a consistency.


In terms of the conditions needed, in \Cref{assume-snr} we  allow $r$ to be arbitrary. To compare against \cite{padilla2021optimal} consider $r = 1$ and recall we impose $\Delta  = \Theta(T)$. In this case,  ignoring $\gamma_{T}$, \Cref{assume-snr} reduces to,
  \begin{equation}
      \label{eq:comp1}
       h \lesssim   T^{-\frac{1}{2+p}  }  \log^{1/2}( T) \,\lesssim \, \kappa,
  \end{equation}
where we have used our choice of $h$ in \Cref{theorem:FSBS}. In contrast, in the same setting, the signal-to-noise ration condition in \cite{padilla2021optimal} is 
\begin{equation}
\label{eq:snr_padilla}
\kappa^{2+p}\Delta    \gtrsim \gamma_{T}\log^{1+\epsilon}(T),
\end{equation}
and with bandwidth being the unknown parameter $\kappa$. However, if \eqref{eq:comp1} holds, then 
\[
    \begin{array}{lll}
     \kappa^{p+2}\Delta \gtrsim \kappa^2\kappa^{p}T \gtrsim   \kappa^2  \left(T^{ - \frac{1}{2+p}  } \right)^pT= \kappa^2 T^{ \frac{2}{2+p} }
    \end{array}
\]
which combined with \eqref{eq:comp1} implies \eqref{eq:snr_padilla}. Thus, 
in \cite{padilla2021optimal} the signal-to-noise ratio condition is weaker than \Cref{assume-snr}, and its localization error is faster. Nevertheless, \Cref{theorem:FSBS} allows for temporal dependence and has practical implications for the choice of the bandwidth.
 
\section{Numerical Experiments}\label{sec-numeric}
We refer to MNSBS to the final estimator.
\subsection{Simulated data analysis}\label{simu-data}
We compare our proposed MNSBS with four competitors -- MNP \citep{padilla2021optimal}, EMNCP \citep{matteson2014nonparametric}, SBS \citep{cho2015multiple} and DCBS \citep{cho2016change} -- across a wide range of simulation settings, using corresponding R functions in \texttt{changepoints} \citep{changepoints}, \texttt{ecp} \citep{ecp} and \texttt{hdbinseg} \citep{hdbinseg} packages.
 We evaluate $L_2$ based statistics in Change point estimation and Long-run variance estimation using the Subregion-Adaptive Vegas Algorithm\footnote{The Subregion-Adaptive Vegas Algorithm is available in R package \texttt{cubature} \citep{cubature}} with a maximum of $10^5$ function evaluations.

For MNSBS implementation we use the Gaussian kernel and the false discovery rate control-based procedure of \cite{padilla2021optimal} for $\tau$ selection. Preliminary estimators are set as $h = 2\times(1/T)^{1/(2r+p)}$, while the second stage estimator has bandwidths respectively set as $\widetilde{h} = 0.05$ and $h_{1} = 2\times \widehat{\kappa}_k^{1/r}$. Selection of $R = \Big\lfloor  \left(\max_{k = 1}^{\widehat{K}}\{e_k - s_k\}\right)^{3/5} \Big\rfloor$ with $\{(s_k, e_k)\}_{k = 1}^{\widehat{K}}$ is guided by \Cref{Long-Run-V-Theorem} using $\{(s_k, e_k)\}_{k =1}^{\widehat{K}}$ from \eqref{final-int}. For the confidence interval construction, we use $\{\widehat{\kappa}_k\}_{k = 1}^{\widehat{K}}$ and $\{\widehat{\sigma}^2_{\infty}(k)\}_{k = 1}^{\widehat{K}}$ to estimate the required unknown quantities for the confidence interval construction. 

\subsubsection{Localization}
We consider four different scenarios with two equally spaced change points. 
For each scenario, we set $r = 2$, and varies $T \in \{150, 300\}$ and $p \in \{3,5\}$. Moreover, we consider $$
\{Y_t = \mathbbm{1}\{\lfloor T/3 \rfloor < t \leq \lfloor 2T/3 \rfloor\}Z_t + X_t\}_{t=1}^T
\subset \mathbb{R}^p$$ with with
$$
X_{t} = 0.3X_{t-1} + \epsilon_{t}.
$$

$\bullet${\bf{Scenario 1} (S1)}
 For any $t$, $Z_t=\mu \in \mathbb{R}^p$, such that $\mu_j = 0$ for $j \in \{1, \ldots, \lceil p/2 \rceil\}$ and $\mu_j = 2$ otherwise. Moreover, $\{\epsilon_t\}$ are i.i.d. $\mathcal{N}(0_p,I_p)$.

$\bullet${\bf{Scenario 2} (S2)}  
For any $t$, $Z_t = 0.3Z_{t-1} + \epsilon_{t}^{(1)}.
$ Moreover,
 $\{\epsilon_t^{(1)}\},\{\epsilon_{t}\}\subset \mathbb{R}^p$ are i.i.d. with entries independently follow Unif$(-1, 1)$ and Unif$(-\sqrt{3}, \sqrt{3})$, respectively.
 
$\bullet${\bf{Scenario 3} (S3)} 
Similarly as in {\bf{S2}}. But now,  $
\{\epsilon_{t}^{(1)}\},\{\epsilon_{t}\}\subset \mathbb{R}^p$ are i.i.d. with entries independently follow standardized $\text{Pareto}(3, 1)$ and $\text{Log-Normal}(0, 1)$, respectively.

$\bullet${\bf{Scenario 4} (S4)} 
For any $t$, $Z_t|\{u_t = 1\} = 1.5\times 1_p, \quad  Z_t|\{u_t = 0\} = -1.5\times 1_p$ and $ X_{t} = 0.3X_{t-1} + \epsilon_{t}$, and $\{\epsilon_t\}\subset \mathbb{R}^p$ are i.i.d. $\mathcal{N}(O_p,I_p)$.

$\bullet${\bf{Scenario 5} (S5)} 
For any $t$,
$
Z_t = 0.3Z_{t-1} + \epsilon_{t}^{(1)} + 0.5\times 1_p
$
and  $\{\epsilon_t^{(1)}\},\{\epsilon_t\}\subset \mathbb{R}^p$ are i.i.d. with entries independently follow $\text{Unif}(-\sqrt{3}, \sqrt{3})$ and the standardised $\text{Pareto}(3, 1)$, respectively.

\textbf{S1-S5} encompass a variety of simulation settings including the same type of distributions, changed mean and constant covariance \textbf{S1}; the same type of distributions, constant mean, changed covariance \textbf{S2}; different types of distributions, constant mean and covariance \textbf{S3};  mixture of distributions \textbf{S4}; and change between light-tailed and heavy-tailed distributions \textbf{S5}. 

\subsubsection{Inference}
We consider the following process
$$
\{Y_t = \mathbbm{1}\{\lfloor T/2 \rfloor < t \leq T\}\mu + X_t\}_{t=1}^{T},
$$
with
$$
X_{t} = 0.3X_{t-1} + \epsilon_{t},
$$
Here, $\mu = 1_p$ and $\{\epsilon_t\}_{t = 1}^T\subset\mathbb{R}^p$ are i.i.d. $\mathcal{N}(0_p,I_p)$. We vary $T \in \{100, 200, 300\}$ and $p \in \{2,3\}$, and observe that our localization results are robust to the bandwidth parameters, yet sensitive to the smoothness parameter $r$. We thus set $r = 1000$ in our simulations, as the density function of a multivariate normal distribution belongs to the H\"{o}lder function class with $r = \infty$.


\subsubsection{Evaluation results}
For a given set of true change points $\mathcal{C}= \{\eta_k\}_{k = 0}^{K+1}$, to assess the accuracy of the estimator $\widehat{\mathcal{C}} = \{\widehat{\eta}_k\}_{k = 0}^{\widehat{K}+1}$ with $\widehat{\eta}_0 = 1$ and $\widehat{\eta}_{T+1} = T+1f$, we report (1) the proportion of misestimating $K$ and (2) the scaled Hausdorff distance $d_{\mathrm{H}}(\widehat{\mathcal{C}},\mathcal{C})$, defined by $$d_{\mathrm{H}}(\widehat{\mathcal{C}},\mathcal{C})=\frac{1}{T}\max\{\max_{x\in{\widehat{\mathcal{C}}}}\min_{y\in{\mathcal{C}}}\{\vert x-y\vert\},\max_{y\in{\widehat{\mathcal{C}}}}\min_{x\in{\mathcal{C}}}\{\vert x-y\vert\}\}.$$

The performance of our change point inference is measured by the coverage of $\eta_k$, defined as $cover_k(1-\alpha)$ for significance level $\alpha \in (0,1)$. For, $k = 1, \dots, K,$
\begin{equation}
    cover_k(1-\alpha)  = \mathbbm{1}\bigg\{  
    \eta_k \in \bigg[\widetilde{\eta}_k + \frac{\widehat{q}_u(\alpha/2)}{\widehat{\kappa}_k^{p/r+2}}, 
    \widetilde{\eta}_k + \frac{\widehat{q}_u(1-\alpha/2)}{\widehat{\kappa}_k^{p/r+2}} \bigg]\bigg\},
\end{equation}
with $\widehat{q}_u(\alpha/2)$ and $\widehat{q}_u(1-\alpha/2)$ are the $\alpha/2$ and $1-\alpha/2$ empirical quantiles of the simulated limiting distribution given in \eqref{limit-dis}, $\widehat{\kappa}_k$ is defined in \eqref{kappa-hat}, and $k = 1, \dots, K$.

We repeat the experiment $200$ times for each setting and report simulation results for localisation in \Cref{Scenario1}, \ref{Scenario2}, \ref{Scenario3}, \ref{Scenario4}, and \ref{Scenario5}. Inference performance is presented in \Cref{inference}. To the best of our knowledge, no competitor exists for change point inference in multivariate nonparametric change settings. 

MNSBS generally performs well in all scenarios considered, among the top two except for {\bf{S2}}. DCBS, designed to estimate change points in mean or second-order structure, performs best in {\bf{S2}}, while MNSBS is comparable to ECP, and significantly better in {\bf{S1}} and {\bf{S5}} for large $T$.
\subsection{Real data application}\label{Real-data}
We applied our proposed change point inference procedure to analyze stock price data\footnote{The stock price data are downloaded from \url{https://fred.stlouisfed.org/series}.}, which consisted of daily adjusted close price of the 3 major stock market indices (S\&P 500, Dow Jones and NASDAQ) from Jan-01-2021 to Jan-19-2023. After removing missing values and standardizing the raw data, the sample size was $n=515$ and the dimension $p=3$. 

We localized 6 estimated change points and performed inference based on them; results are summarized in \Cref{Rexample}. We also implemented the NMP and ECP methods on the same dataset, the estimated change points being presented in \Cref{RDataRes}. Except for the time point Aug-24-2022 estimated by ECP, all other estimated change points were located in the constructed $99\%$ confidence intervals by our proposed method.


\section{Conclusion}\label{sec-conclusion} 
We tackle the problem of change point detection for short range dependent multivariable nonparametric data, which has not been studied in the literature. Our two-stage algorithm MNSBS can consistently estimate the change points in stage one, a novelty in the literature. Then, we derived limiting distributions of change point estimators for inference in stage two, a first in the literature.

Our theoretical analysis reveals multiple challenging and interesting directions for future exploration. Relaxing the assumption $\Delta\asymp T$ may be of interest. In addition, in \Cref{theorem:Final-est}.\textbf{a}, we can see the limiting distribution is a function of the data-generating mechanisms, lacking universality, therefore deriving a practical method to derive the limiting distributions in the non-vanishing regime may be interesting.  

\newpage
\bibliography{references}


\newpage
\appendix
\section*{Appendices}
Additional numerical results and all technical details are included in the supplementary materials.

\section{Detailed simulation results}
In this section, we present the results of our numerical study. To obtain them, we first conduct simulation studies in various scenarios of multivariate nonparametric changes to show the superior performance of our method in localization. We then perform several other simulation studies to illustrate the effectiveness of our methods in change point inference under the vanishing regime. Finally, a real data example is presented. We refer to MNSBS to the final estimator.
\subsection{Simulated data}\label{sec-tables}
We present the tables containing the results of the simulation study in \Cref{sec-numeric} of the main text. On each table, the mean over $200$ repetitions is reported, and the numbers in parenthesis denote the standard errors. For the purpose of identifying misestimation, we compute the averaged coverage among all the repetitions whose $\widehat{K} = K$. In each setting, we highlight the best result in \textbf{bold} and the second best result in \textbf{\textit{bold and italic}}.

\begin{table}[H]
\caption{Localisation results of Scenario 1.}
\label{Scenario1}
\vskip 0.15in
\begin{center}
\begin{small}
\begin{sc}
\begin{tabular}{lcccr}
\hline
& \multicolumn{2}{c}{$T = 150$} & \multicolumn{2}{c}{$T = 300$} \\
Method   & $p = 3$ & $p = 5$ & $p = 3$ & $p = 5$ \\
\hline
\multicolumn{5}{c}{propotion of times $\widehat{K} \neq K$}                                                                             \\
MNSBS & \textbf{\textit{0.040}}            & \textbf{\textit{0.025}}  & \textbf{\textit{0.140 }}           & \textbf{\textit{0.105}}  \\
NMP & 0.170            & 0.325  & 0.255   & 0.420 \\
ECP & \textbf{0.010}            & \textbf{0}  & 0.570    & 0.630 \\
SBS & 0.705            & 0.540  & \textbf{0.115}    & \textbf{0.010} \\
DCBS & 0.210            & 0.120  & 0.145    & 0.115 \\
\multicolumn{5}{c}{average (standard deviation) of $d_{\mathrm{H}}$}                                                                             \\
MNSBS & \textbf{\textit{0.026}} (0.042)            & \textbf{\textit{0.012}} (0.024)  & \textbf{\textit{0.028}} (0.045)  & \textbf{\textit{0.016}} (0.038)  \\
NMP & 0.064 (0.060)            & 0.077 (0.051)  & 0.045 (0.050)  & 0.064 (0.059)  \\
ECP & \textbf{0.017} (0.021)          & \textbf{0.007} (0.009)  & 0.080 (0.065)  & 0.087 (0.066)  \\
SBS & 0.245 (0.138)            & 0.190 (0.157)  & 0.051 (0.096)  & \textbf{0.010} (0.028)  \\
DCBS & 0.061 (0.087)            & 0.024 (0.035)  & \textbf{0.025} (0.036)  & 0.017 (0.033)  \\
\hline
\end{tabular}
\end{sc}
\end{small}
\end{center}
\vskip -0.1in
\end{table}

\begin{table}[H]
\caption{Localisation results of Scenario 2.}
\label{Scenario2}
\vskip 0.15in
\begin{center}
\begin{small}
\begin{sc}
\begin{tabular}{lcccr}
\hline
& \multicolumn{2}{c}{$T = 150$} & \multicolumn{2}{c}{$T = 300$} \\
Method   & $p = 3$ & $p = 5$ & $p = 3$ & $p = 5$ \\
\hline
\multicolumn{5}{c}{propotion of times $\widehat{K} \neq K$}                                                                             \\
MNSBS & 0.815            & 0.645  & 0.670            & 0.590  \\
NMP & \textbf{\textit{0.800}}            & 0.705  & 0.785   & \textbf{\textit{0.515}} \\
ECP & \textbf{0.515}            & \textbf{0.270}  & \textbf{\textit{0.630}}    & 0.665 \\
SBS & 1.000            & 0.995  & 0.820    & 0.625 \\
DCBS & 0.830            & \textbf{\textit{0.490}}  & \textbf{0.160}    & \textbf{0.055} \\
\multicolumn{5}{c}{average (standard deviation) of $d_{\mathrm{H}}$}                                                                             \\
MNSBS & 0.265 (0.124)            & 0.212 (0.154)  & 0.185 (0.147)  & 0.161 (0.149)  \\
NMP & \textbf{\textit{0.261}} (0.123)            & 0.235 (0.142)  & 0.174 (0.151)  & 0.151 (0.153)  \\
ECP & \textbf{0.177} (0.132)          & \textbf{0.094} (0.099)  & \textbf{\textit{0.102}} (0.060)  & \textbf{\textit{0.099}} (0.060)  \\
SBS & 0.333 (0.003)            & 0.332 (0.019)  & 0.279 (0.115)  & 0.222 (0.146)  \\
DCBS & 0.280 (0.117)            & \textbf{\textit{0.174}} (0.156)  & \textbf{0.066} (0.107)  & \textbf{0.020} (0.034)    \\
\hline
\end{tabular}
\end{sc}
\end{small}
\end{center}
\vskip -0.1in
\end{table}

\begin{table}[H]
\caption{Localisation results of Scenario 3.}
\label{Scenario3}
\vskip 0.15in
\begin{center}
\begin{small}
\begin{sc}
\begin{tabular}{lcccr}
\hline
& \multicolumn{2}{c}{$T = 150$} & \multicolumn{2}{c}{$T = 300$} \\
Method   & $p = 3$ & $p = 5$ & $p = 3$ & $p = 5$ \\
\hline
\multicolumn{5}{c}{propotion of times $\widehat{K} \neq K$}                                                                             \\
MNSBS & 0.840            & 0.835  & \textbf{0.740}            & \textbf{\textit{0.755}}  \\
NMP & \textbf{\textit{0.835}}            & \textbf{\textit{0.815}}  & 0.785   & \textbf{0.750} \\
ECP & \textbf{0.750}            & \textbf{0.705}  & \textbf{\textit{0.745}}    & 0.800 \\
SBS & 1.000            & 1.000  & 1.000    & 1.000 \\
DCBS & 0.970            & 0.975  & 0.990    & 0.980 \\
\multicolumn{5}{c}{average (standard deviation) of $d_{\mathrm{H}}$}                                                                             \\
MNSBS & 0.274 (0.073)            & \textbf{\textit{0.274}} (0.078) & \textbf{\textit{0.257}} (0.084) & \textbf{\textit{0.251}} (0.087) \\
NMP & \textbf{\textit{0.269}} (0.077)            & 0.275 (0.076)  & 0.260 (0.082)  & \textbf{\textit{0.251}} (0.084)  \\
ECP & \textbf{0.255} (0.099)          & \textbf{0.232} (0.103)  & \textbf{0.247} (0.091)  & \textbf{0.223} (0.094)  \\
SBS & 0.333 (0)            & 0.333 (0.009)  & 0.333 (0)  & 0.333 (0)  \\
DCBS & 0.330 (0.024)            & 0.331 (0.018)  & 0.333 (0.010)  & 0.333 (0.002)  \\
\hline
\end{tabular}
\end{sc}
\end{small}
\end{center}
\vskip -0.1in
\end{table}

\begin{table}[H]
\caption{Localisation results of Scenario 4.}
\label{Scenario4}
\vskip 0.15in
\begin{center}
\begin{small}
\begin{sc}
\begin{tabular}{lcccr}
\hline
& \multicolumn{2}{c}{$T = 150$} & \multicolumn{2}{c}{$T = 300$} \\
Method   & $p = 3$ & $p = 5$ & $p = 3$ & $p = 5$ \\
\hline
\multicolumn{5}{c}{propotion of times $\widehat{K} \neq K$}                                                                             \\
MNSBS & \textbf{\textit{0.420}}            & \textbf{\textit{0.055}}  & \textbf{0.110}            & \textbf{\textit{0.095}}\\
NMP & 0.575            & 0.405  & 0.145    & 0.210 \\
ECP & \textbf{0.120}            & \textbf{0.050}  & \textbf{\textit{0.125}}    & \textbf{0.055} \\
SBS & 1.000            & 1.000  & 0.910    & 0.845 \\
DCBS & 0.885            & 0.915  & 0.150    & 0.155 \\
\multicolumn{5}{c}{average (standard deviation) of $d_{\mathrm{H}}$}                                                                             \\
MNSBS & \textbf{\textit{0.143}} (0.153)            & \textbf{0.020} (0.055)  & \textbf{0.019} (0.037)  & \textbf{\textit{0.015}} (0.037)  \\
NMP & 0.202 (0.149)            & 0.152 (0.141)  & 0.038 (0.054)  & 0.048 (0.050)  \\
ECP & \textbf{0.058} (0.082)          & \textbf{\textit{0.024}} (0.032)  & \textbf{\textit{0.027}} (0.045)  & \textbf{0.011} (0.048)  \\
SBS & 0.333 (0)            & 0.333 (0)  & 0.305 (0.090)  & 0.285 (0.112)  \\
DCBS & 0.295 (0.102)            & 0.306 (0.089)  & 0.060 (0.111)  & 0.059 (0.112)  \\
\hline
\end{tabular}
\end{sc}
\end{small}
\end{center}
\vskip -0.1in
\end{table}

\begin{table}[H]
\caption{Localisation results of Scenario 5.}
\label{Scenario5}
\vskip 0.15in
\begin{center}
\begin{small}
\begin{sc}
\begin{tabular}{lcccr}
\hline
& \multicolumn{2}{c}{$T = 150$} & \multicolumn{2}{c}{$T = 300$} \\
Method   & $p = 3$ & $p = 5$ & $p = 3$ & $p = 5$ \\
\hline
\multicolumn{5}{c}{propotion of times $\widehat{K} \neq K$}                                                                             \\
MNSBS & \textbf{\textit{0.110}}            & \textbf{\textit{0.080}}  & \textbf{0.215}            & \textbf{\textit{0.230}}  \\
NMP & 0.155           & 0.090  & \textbf{\textit{0.230}}   & \textbf{0.215} \\
ECP & \textbf{0.100}            & \textbf{0.055}  & 0.545    & 0.655 \\
SBS & 0.995         & 0.990  & 0.995    & 0.960 \\
DCBS & 0.960            & 0.965  & 0.975    & 0.980 \\
\multicolumn{5}{c}{average (standard deviation) of $d_{\mathrm{H}}$}                                                                             \\
MNSBS & \textbf{\textit{0.057}} (0.089)            & \textbf{\textit{0.036}} (0.054)  & \textbf{0.039} (0.058)  & \textbf{0.035} (0.051)  \\
NMP & 0.070 (0.107)            & 0.041 (0.060)  & \textbf{\textit{0.040}} (0.057)  & \textbf{\textit{0.038}} (0.054)  \\
ECP & \textbf{0.044} (0.051)          & \textbf{0.032} (0.041)  & 0.083 (0.063)  & 0.093 (0.061)  \\
SBS & 0.332 (0.021)            & 0.330 (0.031)  & 0.331 (0.023)  & 0.321 (0.059)  \\
DCBS & 0.316 (0.061)            & 0.305 (0.075)  & 0.317 (0.062)  & 0.313 (0.068)  \\
\hline
\end{tabular}
\end{sc}
\end{small}
\end{center}
\vskip -0.1in
\end{table}

\begin{table}[H]
\caption{Localisation results of inference.}
\label{inference}
\vskip 0.15in
\begin{center}
\begin{small}
\begin{sc}
\begin{tabular}{lcccr}
\hline
& \multicolumn{2}{c}{$\alpha = 0.01$} & \multicolumn{2}{c}{$\alpha = 0.05$} \\
$n$   & $\mathrm{cover}(1-\alpha)$ & $\mathrm{width}(1-\alpha)$ & $\mathrm{cover}(1-\alpha)$ & $\mathrm{width}(1-\alpha)$ \\
\hline
\multicolumn{5}{c}{$p = 2$}                                                                             \\
100 & 0.864            & 17.613 (6.712) & 0.812 & 14.005 (5.639) \\
200 & 0.904            & 22.940 (7.740)  & 0.838  & 18.407 (6.541)  \\
300 & 0.993            & 26.144 (9.027)  & 0.961  & 20.902 (5.936)  \\
\multicolumn{5}{c}{$p = 3$}                                                                             \\
100 & 0.903            & 15.439 (5.792)  & 0.847            & 11.153 (4.361) \\
200 & 0.966            & 20.108 (7.009)  & 0.949    & 13.920 (5.293)  \\
300 & 0.981            & 22.395 (6.904)  & 0.955    & 15.376 (4.763) \\
\hline
\end{tabular}
\end{sc}
\end{small}
\end{center}
\vskip -0.1in
\end{table}

\subsection{Real data example}\label{RDataRes}
The transformed real data used on \Cref{Real-data} is illustrated in the figure below. These data correspond to   the daily adjusted close price, from Jan-01-2021 to Jan-19-2023, of the 3 major stock market indices, S\&P 500, Dow Jones and NASDAQ. Moreover, in \Cref{Rexample}, we present the  estimated change point by our proposed method MNSBS on the data before mentioned, together with their respective inference. 

\begin{figure}[H]
\vskip 0.2in
\begin{center}
\centerline{\includegraphics[width=0.7\textwidth]{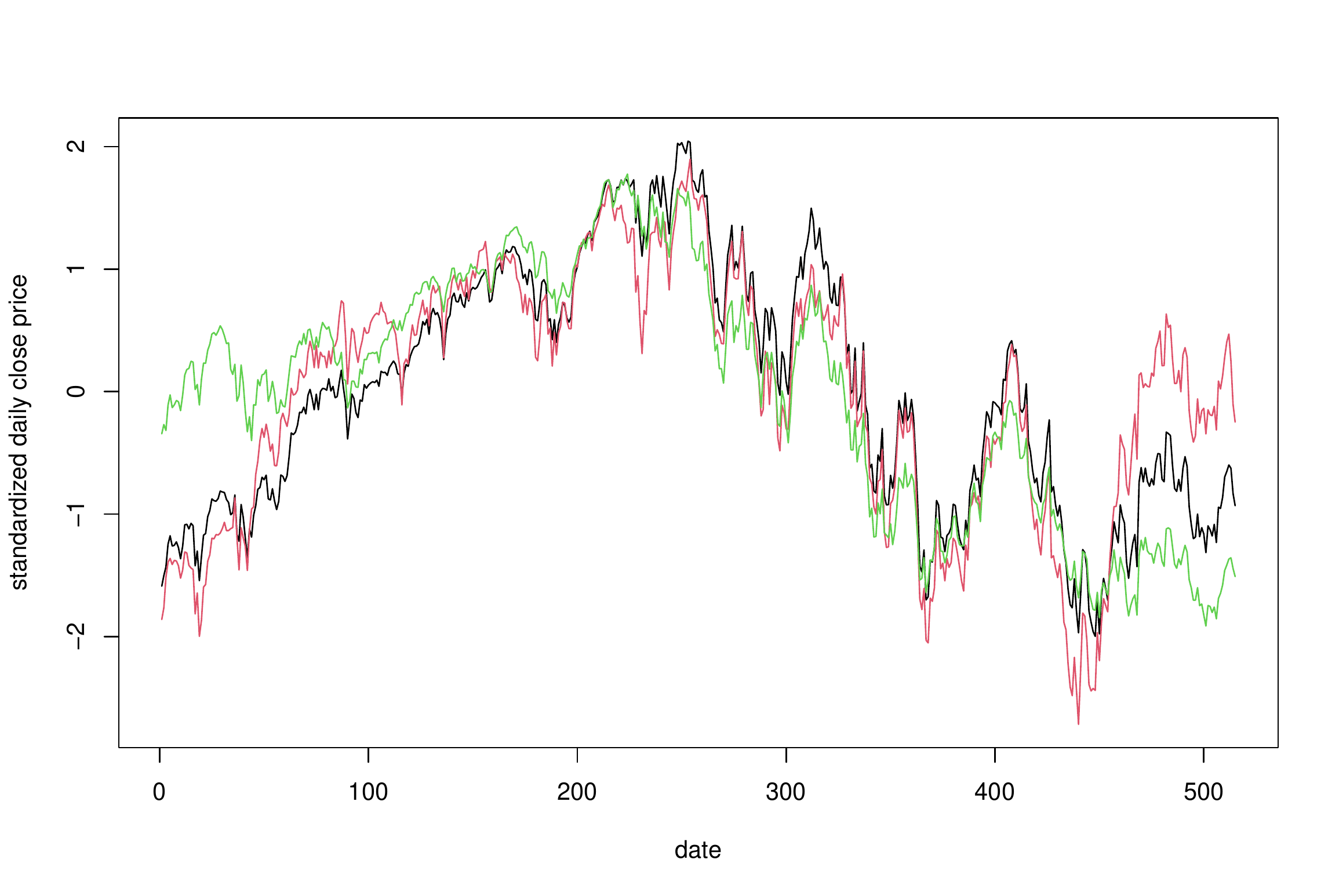}}
\caption{Plot of the standardized daily close price, from Jan-01-2021 to Jan-19-2023, of the 3 major stock market indices.} 
\label{fig:time_n}
\end{center}
\vskip -0.2in
\end{figure}

\begin{table}[H]
\caption{Confidence intervals constructed for change point locations in the Real data example.}
\label{Rexample}
\vskip 0.15in
\begin{center}
\begin{small}
\begin{sc}
\begin{tabular}{lcccr}
\hline
 & \multicolumn{2}{c}{$\alpha = 0.01$} & \multicolumn{2}{c}{$\alpha = 0.05$} \\
$\widehat{\eta}$   & Lower bound & Upper bound & Lower bound & Upper bound \\
\hline
April-07-2021   & April-01-2021            & April-12-2021 & April-05-2021            & April-09-2021\\
June-30-2021   & June-23-2021            & July-09-2021 & June-25-2021            & July-07-2021 \\
Oct-19-2021   & Oct-12-2021            & Oct-26-2021 & Oct-14-2021            & Oct-22-2021\\
Jan-18-2021   & Jan-12-2021            & Jan-21-2021 & Jan-13-2021            & Jan-20-2021\\
April-25-2022   & April-20-2022            & April-28-2022 & April-21-2022            & April-27-2022\\
Oct-27-2022   & Oct-24-2022            & Nov-01-2022 & Oct-25-2022            & Oct-31-2022\\
\hline
\end{tabular}
\end{sc}
\end{small}
\end{center}
\vskip -0.1in
\end{table}
The result of the implementation of NMP and ECP methods on the same dataset are 
\{\text{April-01-2021}, \text{July-01-2021}, \text{Oct-19-2021}, \text{Jan-14-2022}, \text{April-21-2022}, \text{Oct-26-2022}\}
and \{\text{April-08-2021}, \text{June-25-2021}, \text{Oct-18-2021}, \text{Jan-18-2022}, \text{April-28-2022}, \text{Aug-24-2022}, \text{Oct-27-2022}\} respectively.
\newpage
\section{Proof of \Cref{theorem:FSBS}}\label{sec-proof-thm1}
In this section, we present the proof of theorem \Cref{theorem:FSBS}.
\begin{proof}[Proof of \Cref{theorem:FSBS}]

For any $(s,e]\subseteq (0,T] $, let 
	$$ \widetilde f^{  (s, e]}_{t }    (x ) =  \sqrt { \frac{e-t}{ (e -s  )(t-s )}} \sum_{ l =s +1}^{ t}  f_l (x )  -
   \sqrt { \frac{t-s}{ (e-s  )(e-t)}} \sum_{l =t+1}^{ e}  f_l  (x ),\ x\in\mathcal{X}.$$
   For any $\widetilde{r}\in (\rho, T-\rho ] $, we consider
\begin{align*}
& 
\mathcal A(  (s,e], \rho , \lambda   )  =
\bigg\{ \max_{t=s+\rho+1}^{e-\rho }\sup_{x\in\mathbb{R}^p} |\widetilde F_{t, h } ^{s,e } (x)   - \widetilde f_{t} ^{s,e }  (x)  | \le \lambda \bigg\} ;
\\
 &   
\mathcal B (  \widetilde{r}  ,  \rho,  \lambda) = \bigg\{\max_{ N=\rho }^{T -\widetilde{r}   }\sup_{x\in\mathbb{R}^p} \bigg| \frac{1}{\sqrt  N} \sum_{t=\widetilde{r}+1}^{\widetilde{r}+ N}    F_{t,h   }(x)   - \frac{1}{\sqrt  N} \sum_{t=\widetilde{r}  +1}^{\widetilde{r}+N}   f_{t }  (x)  \bigg| \le \lambda \bigg\}  \bigcup 
\\
& 
\quad \quad  \quad \quad \quad \ \   \bigg\{\max_{ N=\rho }^{   \widetilde{r} } \sup_{x\in\mathbb{R}^p}\bigg| \frac{1}{\sqrt  N}  \sum_{t=\widetilde{r} -N +1}^{\widetilde{r}}  F_{t,h  }(x)   - \frac{1}{\sqrt  N}  \sum_{t=\widetilde{r} -N +1}^{\widetilde{r}}     f_{t}  (x)  \bigg| \le \lambda   \bigg\}.
\end{align*}
\\
\\
From \Cref{algorithm:WBS}, we have that
$$ \rho =\frac{\log(T)}{h^p}.$$
Therefore, \Cref{Concentration-Bound} imply that  with
\begin{align}\label{eq:size of lambda in functions} 
 \lambda 
=
C_\lambda   \bigg(  2C \sqrt{\frac{\log T}{h^p}}+\frac{2C_1\sqrt{p}}{\sqrt{h^p}}+2C_2\sqrt{T}h^{r}. \bigg)   ,
\end{align} 
for some diverging sequence $C_\lambda $, it holds that 
$$P \bigg\{ \mathcal A^{c} ( (s,e], \rho,  \lambda )  \bigg\} \lesssim \frac{1}{T^{2}},$$
and,
$$P \bigg\{ \mathcal B^{c} ( \widetilde{r}  , \rho,   \frac{\lambda}{2} )  \bigg\}\lesssim \frac{1}{T^{2}}.$$
Now, we notice that,
\begin{align*}
&\sum_{k=1}^ \mathfrak K \widetilde{n}_k 
=
\sum_{k=1}^ \mathfrak K (2^k -1)
\le
\sum_{k=1}^ \mathfrak K 2^k
\le
2(2^{\lc C_{\mathfrak{K}}(\log(T)) \rc}-1)
=O(T)
.
\end{align*}
In addition, there are $K =O(1) $ number of  change points. In consequence, it follows that 
\begin{align}
& 
P \bigg\{ \mathcal A  ( \mathcal I , \rho ,  \lambda )   \text{ for all } \mathcal I \in \mathcal J   \bigg\} \ge1- \frac{1}{T} , \label{eq:event B0} 
\\
&
P \bigg\{ \mathcal B  (  s , \rho,  \lambda )  \cup \mathcal B (e, \rho  ,  \lambda )   \text{ for all } \mathcal (s,e]=\mathcal  I \in \mathcal J  \bigg\}  \ge 1- \frac{1}{T},  \label{eq:event B1} 
\\
&
P \bigg\{ \mathcal B  (  \eta_k  , \rho,  \lambda )   \text{ for all }  1\le k \le K \bigg\}  \ge1- \frac{1}{T}  \label{eq:event B2} .
\end{align} 
The rest of the argument is made by assuming the events in equations \eqref{eq:event B0}, \eqref{eq:event B1} and  \eqref{eq:event B2} hold. 
By \Cref{remark-1}, we have that on these events,
it is satisfied that 
$$\max_{t=s+\rho+1}^{e-\rho }  \vert\vert  \widetilde F_{t, h }^{s,e} (x)   -   \widetilde f_{t}^{s,e} (x)\vert\vert_{L_2}\le \lambda.$$
Denote 
$$\Upsilon_k  =C      \log (T)  \bigg( T ^{ \frac{  p}{ 2r+p}  }           \bigg) \kappa_k^{-2}      \quad \text{and} 
\quad \Upsilon_{\max  } = C      \log (T)  \bigg(          T ^{ \frac{  p}{ 2r+p }  }           \bigg) \kappa ^{-2}    ,  $$
where $ \kappa = \min \{\kappa_1, \ldots, \kappa_K \} $.
Since $\Upsilon _k$ is  the desired   localisation rate, by induction, it suffices to consider any generic interval $(s, e]  \subseteq (0, T]$ that satisfies the following three conditions:
	\begin{align*}
		&
		\eta_{m-1} \le s\le \eta_m \le \ldots\le \eta_{m+q} \le e \le \eta_{m+q+1}, \quad q\ge -1; \\
		& 
		\text{ either }   \eta_m-s\le \Upsilon_m  \quad \text{or} \quad   s-\eta_{m-1} \le  \Upsilon_{m-1};
		\\
		& 
		\text{ either }  \eta_{m+q+1}-e \le \Upsilon_{m+q+1} \quad  \text{or}\quad   e-\eta_{m+q} \le \Upsilon _{m+q}  .
	\end{align*}
Here $q = -1$ indicates that there is no change point contained in $(s, e]$.
\\
\\
Denote 
$$\Delta_k =  \eta_{k-1}-\eta_{k }  \text{ for }  k =1, \ldots, K+1  \quad \text{and} \quad \Delta= \min\{ \Delta_1,\ldots, \Delta_{K+1}\} .$$
Observe  that  since 
$\kappa_k >0$ 
for all $1\le k\le K$
and  that 
$\Delta_k =\Theta(T) $,   
it holds that $\Upsilon_{\max  }   =o(\Delta)  $.  
Therefore, it has to be the case that for any true  change point $\eta_m\in (0, T ]$, 
either 
$|\eta_m  -s |\le \Upsilon_{m}  $ 
or 
$|\eta_m -s| \ge \Delta - \Upsilon_{\max  }     \geq  \Theta(T) $. This means that 
$ \min\{ |\eta_m-e|, |\eta_m-s| \}\le \Upsilon_{m}  $ 
indicates that 
$\eta_m$ 
is a detected change point in the previous induction step, even if $\eta_m\in (s, e]$.  
We refer to 
$\eta_m\in (s,e]$ 
as an undetected change point if 
$ \min\{ \eta_m -s, \eta_m-e\} =\Theta(T)  $.
To complete the induction step, it suffices to show that MNSBS$( (s,e], h    ,  \tau   )$     
\\
{\bf  (i)} will not detect any new change
point in $(s,e ]$ if
all the change points in that interval have been previously detected, and
\\
{\bf	(ii) }will find a point $D_{m*}^{\mathcal I^* }$ in $(s,e]$ such that $|\eta_m-D_{m*}^{\mathcal I^* } |\le \Upsilon_m$ if there exists at least one undetected change point in $(s, e]$.

In order to accomplish this, we need the following series of steps.

{\bf Step 1.}  We first observe that if $\eta_k \in \{ \eta_k\}_{k=1}^K$ is any change point in the functional time series, by \Cref{lemma:properties of seeded},
 there exists  a seeded interval  $ \mathcal I_k =(s_k, e_k] $ containing  exactly one change point $\eta_k $   such that   
\begin{align*}   
\min\{ \eta_k-s_k , e_k -\eta_k  \}\ge \frac{1}{16} \zeta_k ,  \quad \text{and}
 \quad \max\{ \eta_k-s_k , e_k -\eta_k  \}\le \zeta _k
 \end{align*}
where,
$$ \zeta _k = \frac{9}{10} \min \{ \eta_{k+1}-\eta_k, \eta_k -\eta_{k-1} \} . $$ 
Even more, we notice that if  
$ \eta_k\in (s,e]$ 
is any undetected change point in 
$(s,e] $. 
Then it must hold that 
 $$ s-\eta_{k-1}  \le \Upsilon_{\max}  .$$
Since 
$\Upsilon_{\max} =  O( \log (T)           T ^{ \frac{  p}{ 2r+p }  } )$ and $O(\log^a(T))=o(T^b)$ for any positive numbers $a$ and $b$, we have that $\Upsilon_{\max}=o(T)$. Moreover,  
$\eta_k -s _k \le \zeta_k  \le \frac{9}{10}(\eta_k -\eta_{k-1})  $, so that it holds that 
$$s_k -\eta_{k-1} \ge \frac{1}{10} (\eta_{k} -\eta_{k-1} ) >  \Upsilon_{\max } \ge s- \eta_{k-1} $$
and in consequence 
$  s_k\ge s $. 
Similarly $  e_k\le e $.   
Therefore 
$$ \mathcal I_k = (s_k, e_k] \subseteq (s,e]. $$
\
\\
{\bf Step 2.} Consider the collection of intervals  $\{ \mathcal I _k =(s_k, e_k] \}_{k=1}^K $ in {\bf Step 1.}
 In this step, it is shown that   for each $k \in \{ 1,\ldots,K\}$, it holds that
	\begin{align} \label{eq:properties of vectors}
	 \max_{ t=s_k  +\rho}^{t=e _k  - \rho} 	\vert\vert \widetilde F_{t, h }^{ ( s_k  , e_k  ]   }  \vert\vert_{L_2}  \ge   c_1 \sqrt {T}   \kappa_k   ,
	\end{align}
	for some sufficient small constant $c_1$. 
	 \\
	 \\
	 Let $k \in \{ 1,\ldots,K\}$.
	By  {\bf Step 1}, $ \mathcal I_k $
	contains exactly one change point $\eta_k$.  Since $ f _t $ is a    one-dimensional population  time series and there is only  one change point in $\mathcal I_k=(s_k, e_k]$, 
	it holds that 
$$f_{s_k+1}=...=f_{\eta_k}\neq f_{\eta_k+1}=...=f_{e_k}$$	
which implies, for $s_k<t< \eta_k$
\begin{align*}
\widetilde f^{  (s_k, e_k]}_{t }    =&  \sqrt { \frac{e_k-t}{ (e_k -s_k  )(t-s_k )}} \sum_{ l =s_k +1}^{ t}  f_{\eta_k}    -\sqrt { \frac{t-s_k}{ (e_k-s_k  )(e_k-t)}} \sum_{l =t+1}^{\eta_k}  f_{\eta_k}  \\
-&
   \sqrt { \frac{t-s_k}{ (e_k-s_k  )(e_k-t)}} \sum_{l= \eta_k+1}^{ e_k}  f_{\eta_k+1}  
 \\
 =&
 (t-s_k)\sqrt { \frac{e_k-t}{ (e_k -s_k  )(t-s_k )}}f_{\eta_k}  
 -(\eta_k-t)\sqrt { \frac{t-s_k}{ (e_k-s_k  )(e_k-t)}}f_{\eta_k}  \\
-&
 (e_k-\eta_k)\sqrt { \frac{t-s_k}{ (e_k-s_k  )(e_k-t)}}f_{\eta_k+1}  
 \\
 =&
 \sqrt { \frac{(t-s_k)(e_k-t)}{ (e_k -s_k  )}}f_{\eta_k}  
 -(\eta_k-t)\sqrt { \frac{t-s_k}{ (e_k-s_k  )(e_k-t)}}f_{\eta_k}  \\
-&
 (e_k-\eta_k)\sqrt { \frac{t-s_k}{ (e_k-s_k  )(e_k-t)}}f_{\eta_k+1}  
 \\
 =&
 (e_k-t)\sqrt { \frac{t-s_k}{ (e_k-t)(e_k -s_k  )}}f_{\eta_k}  
 -(\eta_k-t)\sqrt { \frac{t-s_k}{ (e_k-s_k  )(e_k-t)}}f_{\eta_k}  \\
-&
 (e_k-\eta_k)\sqrt { \frac{t-s_k}{ (e_k-s_k  )(e_k-t)}}f_{\eta_k+1}  
 \\
 =&
 (e_k-\eta_k)\sqrt { \frac{t-s_k}{ (e_k-t)(e_k -s_k  )}}f_{\eta_k} 
-
 (e_k-\eta_k)\sqrt { \frac{t-s_k}{ (e_k-s_k  )(e_k-t)}}f_{\eta_k+1}  
 \\
 =&
  (e_k-\eta_k)\sqrt { \frac{t-s_k}{ (e_k-t)(e_k -s_k  )}}(f_{\eta_k}  -f_{\eta_k+1}   ).
\end{align*}
Similarly, for $\eta_k\le t\le e_k$
\begin{align*}
 f^{  (s_k, e_k]}_{t }    =\sqrt {\frac { e  _k-t}{(e  _k- s  _k)(t-s   _k)} }(\eta_k-s  _k)    (   f  _{ \eta_{k}  }   -f _{ \eta_{k}+1    }).   
\end{align*}
Therefore,
\begin{align}\label{eq:one change point 1d cusum}
\widetilde f^{  (s_k, e_k]}_{t }        
=
\begin{cases}
\sqrt {\frac {t-s  _k}{(e  _k-s  _k)(e  _k-t)} }( e  _k-\eta_k)    (   f  _{ \eta_{k}  }   -f _{ \eta_{k  }+1    } ) ,  & s_k   <   t< \eta_k;  
\\
\sqrt {\frac { e  _k-t}{(e  _k- s  _k)(t-s   _k)} }(\eta_k-s  _k)    (   f  _{ \eta_{k} }   -f _{ \eta_{k  }+1    } )  , & \eta_k \le  t\le e_k .
\end{cases} 
\end{align}
Since $\Delta=\Theta(T)$, $\rho=O(\log(T)T^{\frac{p}{2r+p}})$ and $\log^a(T)=o(T^{b})$ for any positive numbers $a$ and $b,$ we have that
	 \begin{equation}
	 \label{eq-bound}
	     \min\{ \eta_k-s_k , e_k -\eta_k  \}\ge \frac{1}{16} \zeta_k  \ge\frac{3}{4} c_2 T  >  \rho  ,
	 \end{equation}
   so that  $\eta_k \in [s_k+\rho, e_k-\rho]$.
   Then, from \eqref{eq:one change point 1d cusum}, \eqref{eq-bound} and the fact that $\vert e_k-s_k\vert<T$ and $\vert \eta_k-s_k\vert<T$, 
   \begin{align}        \label{eq:functional population lower bound}  \vert\vert \widetilde f^{  (s_k, e_k]}_{ \eta_k  }  \vert\vert_{L_2}
   =   \sqrt {\frac { e  _k-\eta_k}{(e  _k- s  _k)(\eta_k-s   _k)} }(\eta_k-s  _k)    \vert\vert   f  _{ \eta_{k} }   -f _{ \eta_{k  }  +1  } \vert \vert_{L_
2}
   \ge  c_ 2 \sqrt {T } \frac{3}{4}\kappa_k.
   \end{align} 
	\
	\\
	Therefore, it holds that 
		\begin{align*}
	\max_{ t=s_k  +\rho}^{t=e _k  - \rho} \vert\vert\widetilde F^{  (s_k, e_k]}_{t,h  }      \vert\vert_{L_2}
	\ge & 	 \vert\vert \widetilde F^{  (s_k, e_k]}_{\eta_k,h }    \vert\vert_{L_2}
	\\
	\ge & \vert\vert \widetilde f^{  (s_k, e_k]}_{ \eta_k  }  \vert\vert_{L_2}  
	 -\lambda 
	\\
	\ge &  c_ 2 \frac{3}{4} \sqrt {T} \kappa_k  -\lambda  ,
	   \end{align*} 
where the first inequality follows from the fact that $\eta_k \in [s_k+\rho, e_k-\rho]$, the second inequality follows from the good  event in \eqref{eq:event B0} and \Cref{remark2}, and the last inequality follows from \eqref{eq:functional population lower bound}. 
	\\
	Next, we observe that $\log^{\frac{1}{2}}(T)\sqrt{\frac{1}{h^p}}=o(\sqrt{T^{\frac{2r+p}{p}}})O(\sqrt{T^{\frac{p}{2r+p}}})=o(\sqrt{T})$, $\rho<c_2T$, and $h^r=o(1)$.
	In consequence, since $\kappa_k  $ is a positive constant, by the upper bound of $\lambda $ on  \Cref{eq:size of lambda in functions}, for sufficiently large $T$, it holds that 
	$$ \frac{c_2}{4}\sqrt{T}\kappa_k\ge \lambda.$$
	Therefore,
	$$\max_{ t=s_k  +\rho}^{t=e _k  - \rho}	\vert\vert \widetilde F^{  (s_k, e_k]}_{t,h  }    \vert\vert_{L_2}  \ge \frac{c_2}{2}\sqrt {T} \kappa_k. $$
	Therefore \Cref{eq:properties of vectors} holds with $ c_1 =\frac{c_2}{2}.$
\\	
 \\
 {\bf Step 3.} 
In this step, it is  shown that   SBS$( (s, e] , h ,\tau)   $  can 
 consistently detect or reject the existence of undetected
change points within $(s, e]$.
 \\
 \\
Suppose $\eta_k\in (s, e ]$ is any undetected change point. Then by the second half of {\bf Step 1}, $\mathcal I_k \subseteq (s,e] $. Therefore 
  \begin{align*}  
		A_{\mathcal I^* } \ge \max_{ t=s_k  +\rho}^{t=e _k  - \rho}	\vert\vert \widetilde F_{t, h }^{ ( s_k  , e_k  ]   } \vert\vert_{L_2}  \ge   c_1 \sqrt {T}   \kappa_k   >  \tau,
	\end{align*} 
where the second inequality follows from    \Cref{eq:properties of vectors}, and the last inequality follows from the fact that, $\log^a(T)=o(T^b)$ for any positive numbers $a$ and $b$ implies $\tau = C_\tau      \bigg(   \log  (T)        \sqrt {   \frac{1}{h^p }  
 }      \bigg)  =o(\sqrt{T})$.
\\
\\
Suppose there does not exist any undetected change point in $(s, e]$. Then for any $ \mathcal I =(\alpha ,\beta]  \subseteq (s, e]$, one of the following situations must hold,
	\begin{itemize}
	\item [(a)]	There is no change point within $ (\alpha ,  \beta   ]$;
	\item [(b)] there exists only one change point $\eta_{k} $ within $(\alpha , \beta ]$ and $\min \{ \eta_{k}- \alpha   ,  \beta-\eta_{k}\}\le \Upsilon_k $; 
	\item [(c)] there exist two change points $\eta_{k} ,\eta_{k+1}$ within $(\alpha ,  \beta ]$ and 
	$$  \eta_{k}-  \alpha  \le \Upsilon_k  \quad \text{and} \quad    \beta   -\eta_{k+1}  \le \Upsilon_ {k+1}  .$$
	\end{itemize}
 Observe that if (a) holds, then we have
$$
\max _{\alpha+\rho<t<\beta-\rho} \vert\vert \widetilde{F}_{t,h}^{(\alpha, \beta]}\vert\vert_{L_2} \leq \max _{\alpha+\rho<t<\beta-\rho} \vert\vert\widetilde{f}_{t}^{(\alpha, \beta]}\vert\vert_{L_2}+\lambda=\lambda \text {. }
$$
Cases (b) and (c) can be dealt with using similar arguments. We will only work on (c) here. It follows that, in the good event in \Cref{eq:event B0}, 
\begin{align}
\max _{\alpha+\rho<t<\beta-\rho} \vert\vert\widetilde{F}_{t,h}^{(\alpha, \beta]}\vert\vert_{L_2} &\leq \max _{\alpha<t<\beta} \vert\vert\widetilde{f}_{t}^{(\alpha, \beta]}\vert\vert_{L_2}+\lambda\\
&\leq \sqrt{e-\eta_k} \kappa_{k+1}+ \sqrt{\eta_k-s} \kappa_k+\lambda
\\
&\le 2 \sqrt{C} \log ^{\frac{1}{2}}(T) \sqrt{T^{\frac{p}{2 r+p}} }+\lambda
\end{align}
where the second inequality is followed by \Cref{lemma:cusum boundary bound}.
 Therefore in the  good  event  in \Cref{eq:event B0}, for any $\mathcal I  =(\alpha, \beta] \subseteq(s,e] $, it holds that 
\begin{equation*} 
A_{\mathcal I}  
= 
\max_{  t=\alpha+ \rho }^{  \beta-\rho      }  \vert\vert \widetilde F _{t,h} ^{( \alpha  ,  \beta]  }\vert\vert_{L_2} 
\le
2\sqrt{C} \log^{ \frac{1}{2} }  (T)  \sqrt {        T ^{ \frac{  p}{ 2r+p }  }}    + \lambda ,
\end{equation*}
Then,
$$
\begin{aligned}
& 2 \sqrt{C} \log ^{\frac{1}{2}}(T) \sqrt{1+T^{\frac{p}{2 r+p}}} +\lambda_{\mathcal{A}} \\
= & 2 \sqrt{C} \log ^{\frac{1}{2}}(T) \sqrt{\frac{1}{ h^p}+1} + 2C \sqrt{\frac{\log T}{h^p}}+\frac{2C_1\sqrt{p}}{\sqrt{h^p}}+2C_2\sqrt{T}h^{r} .
\end{aligned}
$$
We observe that $\sqrt{\frac{\log (T)}{ h^p}}=O\Big(\log (T)^{1 / 2} \sqrt{\frac{1}{h^p}}\Big)$. Moreover,
$$
\sqrt{T} h^r=\sqrt{T}\Big(\frac{1}{ T}\Big)^{\frac{r}{2 r+p}} \leq\Big(T^{\frac{1}{2}-\frac{r}{2 r+p}}\Big) ,
$$
and given that,
$$
\frac{1}{2}-\frac{r}{2 r+p}=\frac{p}{2(2 r+p)}
$$
we get,
$$
\sqrt{T} h^r=o\Big(\log ^{\frac{1}{2}}(T) \sqrt{\frac{1}{ h^p}}\Big) .
$$
Therefore, by the choice of $\tau$, we will always correctly reject the existence of undetected change points, since
$$2 \sqrt{C} \log ^{\frac{1}{2}}(T) \sqrt{T^{\frac{p}{2 r+p}}} +\lambda\le \tau.$$
 Thus, by the choice of $\tau$, it holds that with sufficiently large constant $C_\tau $, 
\begin{align} 
\label{help}
A_ m^{\mathcal I} \le \tau \quad \text{for all } \quad    \mathcal I  \subseteq(s,e] .
\end{align}
As a result, MNSBS$( (s,e], h    ,  \tau   )$     will correctly
reject  if $(s,e]$ contains no undetected change points. 
\\
\
\\
 {\bf Step 4.} 
Assume that there exists an undetected change point $\eta_{\widetilde{k}}\in (s, e]$ such that 
$$ \min\{ \eta_{\widetilde{k}} -s, \eta_{\widetilde{k}}-e\} =\Theta(T).$$  
Let  $\mathcal I^*$ be defined as in  MNSBS $( (s,e], h   ,  \tau   )$ with 
$$\mathcal I^* =( \alpha ^*, \beta ^*].  $$
\\
To complete the  induction, it suffices to show that,   there exists a change point $\eta_k\in  (s ,e ]$ such that 
$ \min\{ \eta_k -s, \eta_k-e\}= \Theta(T)$ and $|b_{\mathcal I^*} -\eta_k|\le \Upsilon_k$.
To this end, we consider the collection of change points of $\{f_t\}_{t\in (\alpha^*,\beta^*]}$
We are to ensure that the assumptions of \Cref{final-bound} are satisfied.
In the following, $\lambda$ is used in \Cref{final-bound}.
Then  \Cref{eq:wbs noise 1} and \Cref{eq:wbs noise 2}  are directly consequence of 
 \Cref{eq:event B0}, \Cref{eq:event B1}, \Cref{eq:event B2}.     
By  {\bf Step 1} with $ \mathcal I_k =(s_k,e_k]$, it holds that $$\min\{ \eta_k-s_k , e_k -\eta_k  \}\ge \frac{1}{16} \zeta_k  \ge c_2 T  ,$$ 
Therefore  for all $ k\in \{ \widetilde{k} : \min \{ \eta_{\widetilde{k}}-s, e-\eta_{\widetilde{k}}\} \ge c_2 T\} $,
$$ \max_{ t=\alpha^*  +\rho}^{t=\beta^*  - \rho}\vert\vert \widetilde F_{t, h }^{ ( \alpha^*  , \beta^*  ]   } \vert\vert_{L_2}  \ge \max_{ t=s_k  +\rho}^{t=e _k  - \rho}\vert\vert \widetilde F_{t, h }^{ ( s_k  , e_k  ]   }\vert\vert_{L_2} \ge   c_1 \sqrt {T}   \kappa_k, $$
where the last inequality follows from \Cref{eq:properties of vectors}.
Therefore
\eqref{con-1} holds in \Cref{final-bound}. 
Finally,
\Cref{con-2} is a direct consequence of  the choices that 
$$
h  =  C_h (T)^{\frac{-1}{2r+d }}    \quad  \text{and} \quad \rho =     \frac{ \log(T) }{nh^d}     .$$
Thus, all the conditions in \Cref{final-bound} are met. So that, there exists a change point $\eta_{k}$ of $\{f _t \}_{ t \in \mathcal I^* } $, satisfying
\begin{equation}
\min \{ \beta ^* -\eta_k,\eta_k- \alpha^* \}    > cT  , \label{eq:coro wbsrp 1d re1}
\end{equation}
and
\begin{align*}
		| b_{\mathcal I^* }-\eta_{k}|\le  \max \{ C_3\lambda  ^2 \kappa_k ^{-2}
	    ,\rho  \} \le&    C _4      \log (T)  \bigg(   \frac{1}{h^p }      + T h^{2r } 
	    \bigg) \kappa_k^{-2}    \\
	    \le&   C      \log (T)  \bigg(           T ^{ \frac{  p}{ 2r+p }  }          \bigg) \kappa_k^{-2}   
	   \end{align*} 
for sufficiently large constant $C $, where we have followed the same line of arguments as for the conclusion of \eqref{help}.
Observe that \\
{\bf i)} The change points
of $\{f_t\}_{ t \in \mathcal I^*}  $ belong to $(s, e]\cap \{ \eta_k\}_{k=1}^K$; and
\\
{\bf ii)} \Cref{eq:coro wbsrp 1d re1}  and  $( \alpha^* , \beta^*  ]  \subseteq (s, e]$ imply that
	\[
	\min \{e-\eta_k,\eta_k-s\}  >  cT     \ge  \Upsilon_{\max }.
	\]
	As discussed in the argument before {\bf Step 1}, this implies that
	$\eta_k $ must be an undetected change point of  $\{f_t\}_{ t \in \mathcal I^*}  $.
\end{proof}



\newpage
\section{Proof of \Cref{theorem:Final-est}}\label{sec-proof-thm2}
In this section, we present the proof of theorem \Cref{theorem:Final-est}.
\begin{proof}[Proof of \Cref{theorem:Final-est}]

{\bf{Uniform tightness}} of $\kappa_k^{2+\frac{p}{r}}\Big|\widetilde{\eta}_k-\eta_k\Big|$. Here we show {\bf{a.1}} and {\bf{b.1}}. For this purpose, we will follow a series of steps. On {\bf{step 1}}, we rewrite \eqref{Final-est} in order to derive a uniform bound. {\bf{Step 2}} analyses the lower bound while {\bf{Step 3}} the upper bound.
\\
{\bf{Step 1:}}
Denote $\widetilde{r}=\widetilde{\eta}_k-\eta_k$. Without loss of generality, suppose $\widetilde{r} \geq 0$. Since $\widetilde{\eta}_k=\eta_k+\widetilde{r}$, defined in \eqref{Final-est}, is the minimizer of $\widehat{Q}_k(\eta)$, it follows that
\begin{equation*}
    \widehat{Q}_k(\eta_k+\widetilde{r})-\widehat{Q}_k(\eta_k)\le 0.
\end{equation*}
Let
\begin{equation}
    Q^*(\eta)=\sum_{t=s_k+1}^{\eta}\vert\vert F_{t,h_2}-f_{(s_k,\eta_k]} \ast \mathcal{K}_{h_2}\vert\vert_{L_2}^2+\sum_{t=\eta+1}^{e_k}\vert\vert F_{t,h_2}-f_{(\eta_k,e_k]}\ast \mathcal{K}_{h_2}\vert\vert_{L_2}^2,
\end{equation}
where, 
\begin{equation}
    f_{(s_k,\eta_k]}=\frac{1}{\eta_k-s_k}\sum_{i=s_k+1}^{\eta_k}f_{i}, \ f_{(\eta_k,e_k]}=\frac{1}{e_k-\eta_k}\sum_{i=\eta_k+1}^{e_k}f_{i}.
\end{equation}
Observe that,
\begin{align}
\label{bounds-req}
    Q^*(\eta_k+\widetilde{r})-Q^*(\eta_k)\le& \widehat{Q}_k(\eta_k)-\widehat{Q}_k(\eta_k+\widetilde{r})-Q^*(\eta_k)+Q^*(\eta_k+\widetilde{r}).
\end{align}
If $\widetilde{r} \leq 1 / \kappa_k^{2+\frac{p}{r}}$, then there is nothing to show. So for the rest of the argument, for contradiction, assume that
\begin{align}
\label{tighness-req}
    \widetilde{r} \geq \frac{1}{\kappa_k^{2+\frac{p}{r}}}.
\end{align}
{\bf{Step 2: Finding a lower bound.}}
In this step, we will find a lower bound of the inequality \eqref{bounds-req}. To this end, we observe that,
\begin{align*}
    Q^*(\eta_k+\widetilde{r})-Q^*(\eta_k)=&\sum_{t=\eta_k+1}^{\eta_k+\widetilde{r}}\vert\vert F_{t,{h_2}}-f_{(s_k,\eta_k]} \ast \mathcal{K}_{{h_2}}\vert\vert_{L_2}^2-\sum_{t=\eta_k+1}^{\eta_k+\widetilde{r}}\vert\vert F_{t,{h_2}}-f_{(\eta_k,e_k]}\ast \mathcal{K}_{h_2}\vert\vert_{L_2}^2
    \\
    =&\sum_{t=\eta_k+1}^{\eta_k+\widetilde{r}}\vert\vert f_{(s_k,\eta_k]} \ast \mathcal{K}_{{h_2}}-f_{(\eta_k,e_k]}\ast \mathcal{K}_{h_2}\vert\vert_{L_2}^2
    \\
    &-2\sum_{t=\eta_k+1}^{\eta_k+\widetilde{r}}\langle f_{(s_k,\eta_k]} \ast \mathcal{K}_{{h_2}}-f_{(\eta_k,e_k]}\ast \mathcal{K}_{h_2},F_{t,{h_2}}-f_{(\eta_k,e_k]}\ast \mathcal{K}_{h_2}\rangle _{L_2}
    \\
    =&\sum_{t=\eta_k+1}^{\eta_k+\widetilde{r}}\frac{1}{2}\vert\vert f_{(s_k,\eta_k]}-f_{(\eta_k,e_k]}\vert\vert_{L_2}^2-2\vert\vert f_{(s_k,\eta_k]} \ast \mathcal{K}_{{h_2}}-f_{(s_k,\eta_k]}+f_{(\eta_k,e_k]}\ast \mathcal{K}_{h_2}-f_{(\eta_k,e_k]}\vert\vert_{L_2}^2
\end{align*}
\begin{align*}
    &-2\sum_{t=\eta_k+1}^{\eta_k+\widetilde{r}}\langle f_{(s_k,\eta_k]} \ast \mathcal{K}_{{h_2}}-f_{(\eta_k,e_k]}\ast \mathcal{K}_{h_2},F_{t,{h_2}}-f_{(\eta_k,e_k]}\ast \mathcal{K}_{h_2}\rangle _{L_2}
    \\
    \geq&
    \frac{1}{2}\widetilde{r}\kappa_k^2-2\sum_{t=\eta_k+1}^{\eta_k+\widetilde{r}}\vert\vert f_{(s_k,\eta_k]} \ast \mathcal{K}_{{h_2}}-f_{(s_k,\eta_k]}+f_{(\eta_k,e_k]}\ast \mathcal{K}_{h_2}-f_{(\eta_k,e_k]}\vert\vert_{L_2}^2
    \\
    &-2\sum_{t=\eta_k+1}^{\eta_k+\widetilde{r}}\langle f_{(s_k,\eta_k]} \ast \mathcal{K}_{{h_2}}-f_{(\eta_k,e_k]}\ast \mathcal{K}_{h_2},F_{t,{h_2}}-f_{(\eta_k,e_k]}\ast \mathcal{K}_{h_2}\rangle _{L_2}
\end{align*}
We consider,
\begin{align*}
&I_1:=2\sum_{t=\eta_k+1}^{\eta_k+\widetilde{r}}\vert\vert f_{(s_k,\eta_k]} \ast \mathcal{K}_{{h_2}}-f_{(s_k,\eta_k]}+f_{(\eta_k,e_k]}\ast \mathcal{K}_{h_2}-f_{(\eta_k,e_k]}\vert\vert_{L_2}^2,\ \text{and,}
\\
&I_2:=2\sum_{t=\eta_k+1}^{\eta_k+\widetilde{r}}\langle f_{(s_k,\eta_k]} \ast \mathcal{K}_{{h_2}}-f_{(\eta_k,e_k]}\ast \mathcal{K}_{h_2},F_{t,{h_2}}-f_{(\eta_k,e_k]}\ast \mathcal{K}_{h_2}\rangle _{L_2}.
\end{align*}
From above, we have that,
\begin{align*}
    Q^*(\eta_k+\widetilde{r})-Q^*(\eta_k)\geq&\frac{1}{2}\widetilde{r}\kappa_k^2-I_1-I_2.
\end{align*}
We now analyze the order of magnitude of term $I_1$. Then, we get a lower bound for the term $-I_1.$ In fact $I_1$, has an upper bound of the form $o_p(\widetilde{r}k^{\frac{p}{r}+2})$, where we use that $\vert\vert f_{\eta_k} \ast \mathcal{K}_{{h_2}}-f_{\eta_k}\vert\vert_{L_2}=o(1)$ and $\vert\vert f_{\eta_{k+1}} \ast \mathcal{K}_{{h_2}}-f_{\eta_{k+1}}\vert\vert_{L_2}=o(1)$. For the term $I_2$, we consider the random variable,
$$Y_i=\frac{\langle f_{[s_k+1,\eta_k]} \ast \mathcal{K}_{{h_2}}-f_{[\eta_k+1,e_k]}\ast \mathcal{K}_{h_2},F_{t,{h_2}}-f_{[\eta_k+1,e_k]}\ast \mathcal{K}_{h_2}\rangle _{L_2}}{\kappa_k\mathbb{E}(\vert\vert F_{t,{h_2}}-f_{\eta_{k+1}}\ast \mathcal{K}_{h_2}\vert\vert_{L^2}^{3})^{1/3}}.$$
In order to use \Cref{lemma3}, we need to bound $\mathbb{E}(|Y_i|^3).$ For this, first we use Cauchy Schwartz inequality,
\begin{align*}
    \mathbb{E}(|Y_i|^3)\le& \frac{(\vert\vert (f_{\eta_k+1}-f_{\eta_k})\ast \mathcal{K}_{h_2}
    \vert\vert_{L^2})^3
    \mathbb{E}(\vert\vert F_{t,{h_2}}-f_{\eta_{k+1}}\ast \mathcal{K}_{h_2}\vert\vert_{L^2}^3)}{\kappa_k^3\mathbb{E}(\vert\vert F_{t,{h_2}}-f_{\eta_{k+1}}\ast \mathcal{K}_{h_2}\vert\vert_{L^2}^{3})}
\end{align*}
then, by Minkowski’s inequality, 
\begin{align*}
    \vert\vert (f_{\eta_k+1}-f_{\eta_k})\ast \mathcal{K}_{h_2}
    \vert\vert_{L^2}=&\Big|\Big| \int_{\mathbb{R}^p}(f_{\eta_k+1}-f_{\eta_k})(\cdot-y)\mathcal{K}_{h_2}(y)dy\Big|\Big|_{L_2}\\
    \le& \int_{\mathbb{R}^p}\Big|\Big|(f_{\eta_k+1}-f_{\eta_k})(\cdot-y)\mathcal{K}_{h_2}(y)\Big|\Big|_{L_2}dy
    \\
    =& \Big(\int_{\mathbb{R}^p}\vert\mathcal{K}_{h_2}(y)\vert dy\Big)\Big|\Big|(f_{\eta_k+1}-f_{\eta_k})(\cdot-y)\Big|\Big|_{L_2}
    \\
    =&\vert\vert f_{\eta_k+1}-f_{\eta_k}\vert\vert_{L^2}\vert\vert \mathcal{K}_{h_2}
    \vert\vert_{L^1}.
\end{align*}
Therefore, by \Cref{kernel-as}, we have
\begin{align*}
    & \frac{(\vert\vert (f_{\eta_k+1}-f_{\eta_k})\ast \mathcal{K}_{h_2}
    \vert\vert_{L^2})^3
    \mathbb{E}(\vert\vert F_{t,{h_2}}-f_{\eta_{k+1}}\ast \mathcal{K}_{h_2}\vert\vert_{L^2}^3)}{\kappa_k^3\mathbb{E}(\vert\vert F_{t,{h_2}}-f_{\eta_{k+1}}\ast \mathcal{K}_{h_2}\vert\vert_{L^2}^{3})}
    \\
    \le&
    \frac{(\vert\vert f_{\eta_k+1}-f_{\eta_k}\vert\vert_{L^2}\vert\vert \mathcal{K}_{h_2}
    \vert\vert_{L^1})^3
    \mathbb{E}(\vert\vert F_{t,{h_2}}-f_{\eta_{k+1}}\ast \mathcal{K}_{h_2}\vert\vert_{L^2}^3)}{\kappa_k^3\mathbb{E}(\vert\vert F_{t,{h_2}}-f_{\eta_{k+1}}\ast \mathcal{K}_{h_2}\vert\vert_{L^2}^3)}
    \\
    \le& C_K.
\end{align*}
 for any $t \in{(\eta_k,e_k]}$.
Moreover, we have that 
   \begin{equation}      
          \begin{array}{lll}
   \label{eqn:aux}
    \mathbb{E}(\vert\vert F_{t,h_2}-f_{\eta_{k+1}}\ast \mathcal{K}_{h_2}\vert\vert_{L^2}^3)^{\frac{1}{3}}
   & =&\Big(\int \Big(\int (\mathcal{K}_{h_2}(x-z)-\mathbb{E}(\mathcal{K}_{h_2}(x-X_t)))^{2} dx\Big) ^{\frac{3}{2}}f_t(z)dz\Big)^{1/3}
    \\
    &\le& \Big(\int \Big(\int (\mathcal{K}_{h_2}(x-z))^2 dx\Big) ^{\frac{3}{2}}f_t(z)dz\Big)^{\frac{1}{3}}\\
&    =&\frac{1}{\kappa^{p/2r}}.
  \end{array} 
   \end{equation}
Therefore, by \Cref{lemma3}, we have that $I_2=o_p\Big(\sqrt{\widetilde{r}}\kappa_{k}\kappa_k^{-\frac{p}{2r}} (\log(\widetilde{r}\kappa_k^{\frac{p}{r}+2})+1)\Big).$
Thus,
\begin{align}
\label{low-bound}
    Q^*(\eta_k+\widetilde{r})-Q^*(\eta_k)\ge& \frac{1}{2}\widetilde{r}\kappa_k^2-O_p\Big(\sqrt{\widetilde{r}}\kappa_{k}\kappa_k^{-\frac{p}{2r}} (\log(\widetilde{r}\kappa_k^{\frac{p}{r}+2})+1)\Big)-o_p(\widetilde{r}\kappa_k^{\frac{p}{r}+2}).
\end{align}
{\bf{Step 3: Finding an upper bound.}}
Now, we proceeded to get an upper bound of \eqref{bounds-req}. This is, an upper bound of the following expression,
\begin{equation}
  \widehat{Q}_k(\eta_k)-\widehat{Q}_k(\eta_k+\widetilde{r})-Q^*(\eta_k)+Q^*(\eta_k+\widetilde{r}).  
\end{equation}
Observe that, this expression can be written as,
\begin{align*}
    &\widehat{Q}_k(\eta_k)-\widehat{Q}_k(\eta_k+\widetilde{r})-Q^*(\eta_k)+Q^*(\eta_k+\widetilde{r})
    \\
    =&-\sum_{t=\eta_k+1}^{\eta_k+\widetilde{r}}\vert\vert F_{t,{h_1}}-F_{(s_k,\eta_k],{h_1}}\vert\vert_{L_2}^2+\sum_{t=\eta_k+1}^{\eta_k+\widetilde{r}}\vert\vert F_{t,{h_1}}-F_{(\eta_k,e_k],{h_1}}\vert\vert_{L_2}^2
    \\
    &+\sum_{t=\eta_k+1}^{\eta_k+\widetilde{r}}\vert\vert F_{t,{h_2}}-f_{(s_k,\eta_k]} \ast \mathcal{K}_{{h_2}}\vert\vert_{L_2}^2-\sum_{t=\eta_k+1}^{\eta_k+\widetilde{r}}\vert\vert F_{t,{h_2}}-f_{(\eta_k,e_k]}\ast \mathcal{K}_{h_2}\vert\vert_{L_2}^2    
\end{align*}
So that, $$\widehat{Q}_k(\eta_k)-\widehat{Q}_k(\eta_k+\widetilde{r})-Q^*(\eta_k)+Q^*(\eta_k+\widetilde{r})=U_1+U_2,$$
where,
\begin{align*}  &U_1=\sum_{t=\eta_k+1}^{\eta_k+\widetilde{r}}\vert\vert F_{t,{h_2}}-f_{(s_k,\eta_k]} \ast \mathcal{K}_{{h_2}}\vert\vert_{L_2}^2-\sum_{t=\eta_k+1}^{\eta_k+\widetilde{r}}\vert\vert F_{t,{h_1}}-F_{(s_k,\eta_k],{h_1}}\vert\vert_{L_2}^2,\ \text{and,}
\\
&U_2=\sum_{t=\eta_k+1}^{\eta_k+\widetilde{r}}\vert\vert F_{t,{h_1}}-F_{(\eta_k,e_k],{h_1}}\vert\vert_{L_2}^2
    -\sum_{t=\eta_k+1}^{\eta_k+\widetilde{r}}\vert\vert F_{t,{h_2}}-f_{(\eta_k,e_k]}\ast \mathcal{K}_{h_2}\vert\vert_{L_2}^2.
\end{align*}
Now, we analyze each of the terms above. For $U_1,$ observe that
\begin{align*}
    &\sum_{t=\eta_k+1}^{\eta_k+\widetilde{r}}\vert\vert F_{t,{h_2}}-f_{(s_k,\eta_k]} \ast \mathcal{K}_{{h_2}}\vert\vert_{L_2}^2-\sum_{t=\eta_k+1}^{\eta_k+\widetilde{r}}\vert\vert F_{t,{h_1}}-F_{(s_k,\eta_k],{h_1}}\vert\vert_{L_2}^2
    \\
    =&\sum_{t=\eta_k+1}^{\eta_k+\widetilde{r}}\vert\vert F_{t,{h_2}}-f_{(s_k,\eta_k]} \ast \mathcal{K}_{{h_2}}\vert\vert_{L_2}^2-\sum_{t=\eta_k+1}^{\eta_k+\widetilde{r}}\vert\vert F_{t,{h_2}}-F_{(s_k,\eta_k],{h_2}}\vert\vert_{L_2}^2
    \\
    +&\sum_{t=\eta_k+1}^{\eta_k+\widetilde{r}}\vert\vert F_{t,{h_2}}-F_{(s_k,\eta_k],{h_2}}\vert\vert_{L_2}^2-\sum_{t=\eta_k+1}^{\eta_k+\widetilde{r}}\vert\vert F_{t,{h_1}}-F_{(s_k,\eta_k],{h_1}}\vert\vert_{L_2}^2
    \\
    =&I_3+I_4,
\end{align*}
where,
\begin{align*}
&I_3=\sum_{t=\eta_k+1}^{\eta_k+\widetilde{r}}\vert\vert F_{t,{h_2}}-f_{(s_k,\eta_k]} \ast \mathcal{K}_{{h_2}}\vert\vert_{L_2}^2-\sum_{t=\eta_k+1}^{\eta_k+\widetilde{r}}\vert\vert F_{t,{h_2}}-F_{(s_k,\eta_k],{h_2}}\vert\vert_{L_2}^2, \ \text{and,}
\\
&I_4=\sum_{t=\eta_k+1}^{\eta_k+\widetilde{r}}\vert\vert F_{t,{h_2}}-F_{(s_k,\eta_k],{h_2}}\vert\vert_{L_2}^2-\sum_{t=\eta_k+1}^{\eta_k+\widetilde{r}}\vert\vert F_{t,{h_1}}-F_{(s_k,\eta_k],{h_1}}\vert\vert_{L_2}^2.
\end{align*}
 To analyze $I_3,$ we rewrite it as follow,
\begin{align*}
I_3=&\sum_{t=\eta_k+1}^{\eta_k+\widetilde{r}}\vert\vert f_{(s_k,\eta_k]} \ast \mathcal{K}_{{h_2}}-F_{(s_k,\eta_k],{h_2}}\vert\vert_{L_2}^2-2\sum_{t=\eta_k+1}^{\eta_k+\widetilde{r}}\langle f_{(s_k,\eta_k]} \ast \mathcal{K}_{{h_2}}-F_{(s_k,\eta_k],{h_2}} ,F_{t,{h_2}}-f_{(s_k,\eta_k]} \ast \mathcal{K}_{{h_2}}\rangle_{L_2}
    \\
    =&I_{3,1}+I_{3,2},
\end{align*}
where,
\begin{align*}
&I_{3,1}=\sum_{t=\eta_k+1}^{\eta_k+\widetilde{r}}\vert\vert f_{(s_k,\eta_k]} \ast \mathcal{K}_{{h_2}}-F_{(s_k,\eta_k],{h_2}}\vert\vert_{L_2}^2, \ \text{and,}
\\
&I_{3,2}=-2\sum_{t=\eta_k+1}^{\eta_k+\widetilde{r}}\langle f_{(s_k,\eta_k]} \ast \mathcal{K}_{{h_2}}-F_{(s_k,\eta_k],{h_2}} ,F_{t,{h_2}}-f_{(s_k,\eta_k]} \ast \mathcal{K}_{{h_2}}\rangle_{L_2}.
\end{align*}
Now, we will get an upper bound for each of the terms above. The term $I_{3,1}=O_p\Big(\widetilde{r}\frac{1}{T}\frac{\log(T)}{\kappa_k^{\frac{p}{r}}}\Big)$, which is followed by the use of \Cref{remark-1} and $\Delta=\Theta(T)$. Even more, by \Cref{assume-snr2}, we get 
\begin{equation}
\label{tig-1-1}
    I_{3,1}=o_p(\widetilde{r}\kappa_k^{\frac{p}{r}+2}).
\end{equation}
For the term $I_{3,2}$, by Cauchy Schwartz inequality and triangle inequality,
\begin{align*}
    &\langle f_{(s_k,\eta_k]} \ast \mathcal{K}_{\kappa}-F_{(s_k,\eta_k],\kappa} ,F_{t,{h_2}}-f_{(s_k,\eta_k]} \ast \mathcal{K}_{{h_2}}\rangle_{L_2}
    \\
    &\le
    \vert\vert f_{(s_k,\eta_k]} \ast \mathcal{K}_{{h_2}}-F_{(s_k,\eta_k],{h_2}}\vert\vert_{L_2} \vert\vert F_{t,{h_2}}-f_{(s_k,\eta_k]} \ast \mathcal{K}_{{h_2}}\vert\vert_{L_2}
    \\
    &\le
    \vert\vert f_{(s_k,\eta_k]} \ast \mathcal{K}_{{h_2}}-F_{(s_k,\eta_k],{h_2}}\vert\vert_{L_2} \Big(\vert\vert F_{t,{h_2}}-f_{[\eta_{k}+1,e_k]} \ast \mathcal{K}_{{h_2}}\vert\vert_{L_2}+\vert\vert f_{[\eta_k+1,e_k]} \ast \mathcal{K}_{{h_2}}-f_{[s_k+1,\eta_k]} \ast \mathcal{K}_{{h_2}}\vert\vert_{L_2}\Big)
\end{align*}
for any $t \in{(\eta_k,\eta_k+\widetilde{r}]} $. By the \Cref{remark-1}, and the fact that $\Delta=\Theta(T)$, we have that $$\vert\vert f_{(s_k,\eta_k]} \ast \mathcal{K}_{{h_2}}-F_{(s_k,\eta_k],{h_2}}\vert\vert_{L_2}=
O_p\Big(\frac{1}{\sqrt{T}}\sqrt{\frac{\log(T)}{\kappa_k^{\frac{p}{r}}}}\Big)$$ and using basic properties of integrals $\vert\vert f_{[\eta_k+1,e_k]} \ast \mathcal{K}_{h_2}-f_{[s_k+1,\eta_k]} \ast \mathcal{K}_{h_2}\vert\vert_{L_2}=O(\kappa_k).$ Therefore,
\begin{align*}
    I_{3,2}&\le O_p\Big(\frac{1}{\sqrt{T}}\sqrt{\frac{\log(T)}{\kappa_k^{\frac{p}{r}}}}\Big)\Big( O(\widetilde{r}\kappa_k)+\sum_{t=\eta_k+1}^{\eta_k+\widetilde{r}}\vert\vert F_{t,{h_2}}-f_{[\eta_{k}+1,e_k]} \ast \mathcal{K}_{{h_2}}\vert\vert_{L_2}\Big)
\end{align*}
Now, we need to get a bound of the magnitude of 
$$\sum_{t=\eta_k+1}^{\eta_k+\widetilde{r}}\vert\vert F_{t,{h_2}}-f_{[\eta_{k}+1,e_k]} \ast \mathcal{K}_{{h_2}}\vert\vert_{L_2},$$
in order to get an upper for $I_{3,2}.$ This is done 
similarly to $I_2.$ We consider the random variable 
$$\widetilde{Y}_i=\frac{\langle F_{t,{h_2}}-f_{(\eta_k,e_k]}\ast \mathcal{K}_{h_2},F_{t,{h_2}}-f_{(\eta_k,e_k]}\ast \mathcal{K}_{h_2}\rangle^{\frac{1}{2}} _{L_2}-\mathbb{E}(\vert\vert F_{t,{h_2}}-f_{\eta_{k+1}}\ast \mathcal{K}_{h_2}\vert\vert_{L^2})}{\mathbb{E}(\vert\vert F_{t,{h_2}}-f_{\eta_{k+1}}\ast \mathcal{K}_{h_2}\vert\vert_{L^2}^3)^{\frac{1}{3}}}.$$
In order to use \Cref{lemma3}, we observe that since $\vert\vert F_{t,{h_2}}-f_{\eta_{k+1}}\ast \mathcal{K}_{h_2}\vert\vert_{L^2}\ge0,$
\begin{align*}
    \mathbb{E}(|\widetilde{Y}_i|^3)\le\frac{
    \mathbb{E}(\vert\vert F_{t,{h_2}}-f_{\eta_{k+1}}\ast \mathcal{K}_{h_2}\vert\vert_{L^2}^3)}{\mathbb{E}(\vert\vert F_{t,{h_2}}-f_{\eta_{k+1}}\ast \mathcal{K}_{h_2}\vert\vert_{L^2}^3)}
    = 1.
\end{align*}
Therefore, using \Cref{lemma3} and that $\mathbb{E}(\vert\vert F_{t,{h_2}}-f_{\eta_{k+1}}\ast \mathcal{K}_{h_2}\vert\vert_{L^2})=O(\kappa_k^{\frac{-p}{2r}})$ by \eqref{eqn:aux}, we get
that
\begin{align*}
    \sum_{t=\eta_k+1}^{\eta_k+\widetilde{r}}\vert\vert F_{t,{h_2}}-f_{[\eta_{k}+1,e_k]} \ast \mathcal{K}_{{h_2}}\vert\vert_{L_2}^2=O_p(\sqrt{\widetilde{r}\kappa_k^{-\frac{p}{r}}}(\log(\widetilde{r}\kappa_k^{\frac{p}{r}+2})+1))+O_p(\widetilde{r}\kappa_k^{\frac{-p}{2r}}).
\end{align*}
Thus, by \Cref{assume-snr2} and above,
\begin{align}
\label{tig-1-2}
    I_{3,2}&\le O_p\Big(\frac{1}{\sqrt{T}}\sqrt{\frac{\log(T)}{\kappa_k^{\frac{p}{r}}}}\Big)\Big( O(\widetilde{r}\kappa_k)+O_p(\sqrt{\widetilde{r}\kappa_k^{-\frac{p}{r}}}(\log(\widetilde{r}\kappa_k^{\frac{p}{r}+2})+1))+O_p(\widetilde{r}\kappa_k^{\frac{-p}{2r}})\Big)=o_p(\widetilde{r}\kappa_k^{\frac{p}{r}+2}).
\end{align}
Consequently, $I_3$ has been bounded, and we only need to go over the term $I_4$, to finalize the analysis for $U_1.$
To analyze $I_4,$ we observe that
\begin{align*}
    I_4=& \sum_{t=\eta_k+1}^{\eta_k+\widetilde{r}}\vert\vert F_{t,{h_2}}-F_{(s_k,\eta_k],{h_2}}\vert\vert_{L_2}^2-\sum_{t=\eta_k+1}^{\eta_k+\widetilde{r}}\vert\vert F_{t,{h_1}}-F_{(s_k,\eta_k],{h_1}}\vert\vert_{L_2}^2
    \\
    =&\sum_{t=\eta_k+1}^{\eta_k+\widetilde{r}}\Big[\langle F_{t,h_2},F_{t,h_2}\rangle_{L_2}-2\langle F_{t,h_2}, F_{(s_k,\eta_k],{h_2}}\rangle_{L_2}+\langle F_{(s_k,\eta_k],{h_2}},F_{(s_k,\eta_k],{h_2}} 
    \rangle_{L_2}\Big]
    \\
    &+\sum_{t=\eta_k+1}^{\eta_k+\widetilde{r}}\Big[-\langle F_{t,h_1},F_{t,h_1}\rangle_{L_2}+2\langle F_{t,h_1}, F_{(s_k,\eta_k],{h_1}}\rangle_{L_2}-\langle F_{(s_k,\eta_k],{h_1}},F_{(s_k,\eta_k],{h_1}} 
    \rangle_{L_2}\Big]
    \\
    =&I_{4,1}+I_{4,2}+I_{4,3},
\end{align*}
where,
\begin{align*}
    &I_{4,1}=\sum_{t=\eta_k+1}^{\eta_k+\widetilde{r}}\langle F_{t,h_2},F_{t,h_2}\rangle_{L_2}-\langle F_{t,h_1},F_{t,h_1}\rangle_{L_2}\, 
    \\
    &I_{4,2}=\sum_{t=\eta_k+1}^{\eta_k+\widetilde{r}}2\langle F_{t,h_1}, F_{(s_k,\eta_k],{h_1}}\rangle_{L_2}-2\langle F_{t,h_2}, F_{(s_k,\eta_k],{h_2}}\rangle_{L_2}, \ \text{and},
    \\
    &I_{4,3}=\sum_{t=\eta_k+1}^{\eta_k+\widetilde{r}}\langle F_{(s_k,\eta_k],{h_2}},F_{(s_k,\eta_k],{h_2}} 
    \rangle_{L_2}-\langle F_{(s_k,\eta_k],{h_1}},F_{(s_k,\eta_k],{h_1}} 
    \rangle_{L_2}.
\end{align*}
Now, we explore each of the terms $I_{4,1},I_{4,2},$ and $I_{4,3}.$ First, $I_{4,1}$
can be bounded as follows, we add and subtract $\langle F_{t,h_1},F_{t,h_2}\rangle_{L_2},$ to get
\begin{align*}
    &\sum_{t=\eta_k+1}^{\eta_k+\widetilde{r}}\langle F_{t,h_2},F_{t,h_2}\rangle_{L_2}-\langle F_{t,h_1},F_{t,h_1}\rangle_{L_2}
    \\
    =&\sum_{t=\eta_k+1}^{\eta_k+\widetilde{r}}\langle F_{t,h_2},F_{t,h_2}\rangle_{L_2}-\langle F_{t,h_1},F_{t,h_1}\rangle_{L_2}+
    \langle F_{t,h_1},F_{t,h_2}\rangle_{L_2}-\langle F_{t,h_1},F_{t,h_2}\rangle_{L_2}
    \\
    =&
    \sum_{t=\eta_k+1}^{\eta_k+\widetilde{r}}\langle F_{t,h_2}-F_{t,h_1},F_{t,h_2}\rangle_{L_2}+\langle F_{t,h_1},F_{t,h_2}-F_{t,h_1}\rangle_{L_2}
\end{align*}
which, by H\"{o}lder’s inequality, is bounded by
\begin{align*}
     \sum_{t=\eta_k+1}^{\eta_k+\widetilde{r}}\vert\vert F_{t,h_2}\vert\vert_{L_2}\vert\vert F_{t,h_2}-F_{t,h_1}\vert\vert_{L_2}+\vert\vert F_{t,h_1}\vert\vert_{L_2}\vert\vert F_{t,h_2}-F_{t,h_1}\vert\vert_{L_2}
     =\widetilde{r}O_p(\frac{T^{-\frac{r}{2r+p}}}{\kappa_k^{\frac{p}{2r}+\frac{1}{2}+\frac{p}{2r}}}\log^{\frac{r}{2r+p}}(T)))
\end{align*}
since $\vert\vert F_{t,h_1}-F_{t,h_2}\vert\vert_{L_2}=O(\frac{|\kappa-\widehat{\kappa}|^{\frac{1}{2}}}{\kappa_k^{\frac{p}{2r}+\frac{1}{2}}})=O_p(\frac{T^{-\frac{r}{2r+p}}}{\kappa_k^{\frac{p}{2r}+\frac{1}{2}}}\log^{\frac{r}{2r+p}}(T)),$ for any $t$, see \Cref{remark2} for more detail. 
Similarly, for $I_{4,2}$, we have that adding and subtracting $2\langle F_{t,h_1},F_{(s_k,\eta_k],{h_2}}\rangle_{L_2},$
\begin{align*}
    &\sum_{t=\eta_k+1}^{\eta_k+\widetilde{r}}2\langle F_{t,h_1}, F_{(s_k,\eta_k],{h_1}}\rangle_{L_2}-2\langle F_{t,h_2}, F_{(s_k,\eta_k],{h_2}}\rangle_{L_2}
    \\
    =&\sum_{t=\eta_k+1}^{\eta_k+\widetilde{r}}2\langle F_{t,h_1}, F_{(s_k,\eta_k],{h_1}}\rangle_{L_2}-2\langle F_{t,h_2}, F_{(s_k,\eta_k],{h_2}}\rangle_{L_2}+
    2\langle F_{t,h_1}, F_{(s_k,\eta_k],{h_2}}\rangle_{L_2}-2\langle F_{t,h_1}, F_{(s_k,\eta_k],{h_2}}\rangle_{L_2}
    \\
    =&
    \sum_{t=\eta_k+1}^{\eta_k+\widetilde{r}}2\langle F_{t,h_1}-F_{t,h_2}, F_{(s_k,\eta_k],{h_2}}\rangle_{L_2}+2\langle F_{t,h_1},F_{(s_k,\eta_k],{h_2}}-F_{(s_k,\eta_k],{h_1}}\rangle_{L_2},
\end{align*}
and by H\"{o}lder’s inequality and \Cref{remark2}, it is bounded by
\begin{align*}
     \sum_{t=\eta_k+1}^{\eta_k+\widetilde{r}}\vert\vert F_{t,h_1}-F_{t,h_2}\vert\vert_{L_2}\vert\vert F_{(s_k,\eta_k],{h_2}}\vert\vert_{L_2}+\vert\vert F_{t,h_1}\vert\vert_{L_2}\vert\vert F_{(s_k,\eta_k],{h_2}}-F_{(s_k,\eta_k],{h_1}}\vert\vert_{L_2}
     =\widetilde{r}O_p(\frac{T^{-\frac{r}{2r+p}}}{\kappa_k^{\frac{p}{2r}+\frac{1}{2}+\frac{p}{2r}}}\log^{\frac{r}{2r+p}}(T))).
\end{align*}
Finally, for $I_{4,3},$ we notice that, adding and subtracting $\langle F_{(s_k,\eta_k],{h_1}},F_{(s_k,\eta_k],{h_2}} 
    \rangle_{L_2}$, it is written as,
\begin{align*}
    &\sum_{t=\eta_k+1}^{\eta_k+\widetilde{r}}\langle F_{(s_k,\eta_k],{h_2}},F_{(s_k,\eta_k],{h_2}} 
    \rangle_{L_2}-\langle F_{(s_k,\eta_k],{h_1}},F_{(s_k,\eta_k],{h_1}} 
    \rangle_{L_2}
    \\
    =&\sum_{t=\eta_k+1}^{\eta_k+\widetilde{r}}\langle F_{(s_k,\eta_k],{h_2}},F_{(s_k,\eta_k],{h_2}} 
    \rangle_{L_2}-\langle F_{(s_k,\eta_k],{h_1}},F_{(s_k,\eta_k],{h_1}} 
    \rangle_{L_2}+
    \langle F_{(s_k,\eta_k],{h_1}},F_{(s_k,\eta_k],{h_2}} 
    \rangle_{L_2}-\langle F_{(s_k,\eta_k],{h_1}},F_{(s_k,\eta_k],{h_2}} 
    \rangle_{L_2}
    \\
    =&
    \sum_{t=\eta_k+1}^{\eta_k+\widetilde{r}}\langle F_{(s_k,\eta_k],{h_2}}-F_{(s_k,\eta_k],{h_1}},F_{(s_k,\eta_k],{h_2}}\rangle_{L_2}+\langle F_{(s_k,\eta_k],{h_1}},F_{(s_k,\eta_k],{h_2}}-F_{(s_k,\eta_k],{h_1}}\rangle_{L_2}
\end{align*}
which, by H\"{o}lder’s inequality and \Cref{remark2}, is bounded by
\begin{align*}
     &\sum_{t=\eta_k+1}^{\eta_k+\widetilde{r}}\vert\vert F_{(s_k,\eta_k],{h_2}}\vert\vert_{L_2}\vert\vert F_{(s_k,\eta_k],{h_2}}-F_{(s_k,\eta_k],{h_2}}\vert\vert_{L_2}+\vert\vert F_{(s_k,\eta_k],{h_1}}\vert\vert_{L_2}\vert\vert F_{(s_k,\eta_k],{h_2}}-F_{(s_k,\eta_k],{h_1}}\vert\vert_{L_2}
     \\
     =&
     \widetilde{r}O_p(\frac{T^{-\frac{r}{2r+p}}}{\kappa_k^{\frac{p}{2r}+\frac{1}{2}+\frac{p}{2r}}}\log^{\frac{r}{2r+p}}(T)))
\end{align*}
Then, by above and \Cref{assume-snr2}, we conclude 
\begin{equation}
\label{tig-1-3}
I_4=o_p(\widetilde{r}\kappa_k^{\frac{p}{r}+2}).
\end{equation}
From \eqref{tig-1-1}, \eqref{tig-1-2} and \eqref{tig-1-3}, we find that $U_1$ has the following upper bound,
\begin{align}
\label{tightness-1}
    \sum_{t=\eta_k+1}^{\eta_k+\widetilde{r}}\vert\vert F_{t,{h_2}}-f_{(s_k,\eta_k]} \ast \mathcal{K}_{{h_2}}\vert\vert_{L_2}^2-\sum_{t=\eta_k+1}^{\eta_k+\widetilde{r}}\vert\vert F_{t,{h_1}}-F_{(s_k,\eta_k],{h_1}}\vert\vert_{L_2}^2
    =o_p(\widetilde{r}\kappa_k^{\frac{p}{r}+2}).
\end{align}
Now, making an analogous analysis, we have that $U_2$ is upper bounded by,
\begin{align}
\label{tightness-2}
    \sum_{t=\eta_k+1}^{\eta_k+\widetilde{r}}\vert\vert F_{t,{h_1}}-F_{(\eta_k,e_k],{h_1}}\vert\vert_{L_2}^2
    -\sum_{t=\eta_k+1}^{\eta_k+\widetilde{r}}\vert\vert F_{t,{h_2}}-f_{(\eta_k,e_k]}\ast \mathcal{K}_{h_2}\vert\vert_{L_2}^2=o_p(\widetilde{r}\kappa_k^{\frac{p}{r}+2}).    
\end{align}
In fact, we observe that
\begin{align*}
    &\sum_{t=\eta_k+1}^{\eta_k+\widetilde{r}}\vert\vert F_{t,{h_1}}-F_{(\eta_k,e_k],{h_1}}\vert\vert_{L_2}^2
    -\sum_{t=\eta_k+1}^{\eta_k+\widetilde{r}}\vert\vert F_{t,{h_2}}-f_{(\eta_k,e_k]}\ast \mathcal{K}_{h_2}\vert\vert_{L_2}^2
    \\
    =&\sum_{t=\eta_k+1}^{\eta_k+\widetilde{r}}\vert\vert F_{t,{h_2}}-F_{(\eta_k,e_k],{h_2}}\vert\vert_{L_2}^2-\sum_{t=\eta_k+1}^{\eta_k+\widetilde{r}}\vert\vert F_{t,{h_2}}-f_{(\eta_k,e_k]}\ast \mathcal{K}_{h_2}\vert\vert_{L_2}^2
    \\
    +&\sum_{t=\eta_k+1}^{\eta_k+\widetilde{r}}\vert\vert F_{t,{h_1}}-F_{(\eta_k,e_k],{h_1}}\vert\vert_{L_2}^2
    -\sum_{t=\eta_k+1}^{\eta_k+\widetilde{r}}\vert\vert F_{t,{h_2}}-F_{(\eta_k,e_k],{h_2}}\vert\vert_{L_2}^2
    \\
    =&I_5+I_6,
\end{align*}
where,
\begin{align*}
    &I_5=\sum_{t=\eta_k+1}^{\eta_k+\widetilde{r}}\vert\vert F_{t,{h_2}}-F_{(\eta_k,e_k],{h_2}}\vert\vert_{L_2}^2-\sum_{t=\eta_k+1}^{\eta_k+\widetilde{r}}\vert\vert F_{t,{h_2}}-f_{(\eta_k,e_k]}\ast \mathcal{K}_{h_2}\vert\vert_{L_2}^2,\ \text{and},\\ 
&I_6=\sum_{t=\eta_k+1}^{\eta_k+\widetilde{r}}\vert\vert F_{t,{h_1}}-F_{(\eta_k,e_k],{h_1}}\vert\vert_{L_2}^2
    -\sum_{t=\eta_k+1}^{\eta_k+\widetilde{r}}\vert\vert F_{t,{h_2}}-F_{(\eta_k,e_k],{h_2}}\vert\vert_{L_2}^2.
\end{align*}
Then, $I_5$ is bounded as follows
\begin{align*}
    I_5=&\sum_{t=\eta_k+1}^{\eta_k+\widetilde{r}}\vert\vert f_{(\eta_k,e_k]} \ast \mathcal{K}_{{h_2}}-F_{(\eta_k,e_k],{h_2}}\vert\vert_{L_2}^2+2\sum_{t=\eta_k+1}^{\eta_k+\widetilde{r}}\langle f_{(\eta_k,e_k]} \ast \mathcal{K}_{{h_2}}-F_{(\eta_k,e_k],{h_2}} ,F_{t,{h_2}}-f_{(\eta_k,e_k]} \ast \mathcal{K}_{{h_2}}\rangle_{L_2}
\end{align*}
where,
\begin{align*}
 & I_{5,1}=  \sum_{t=\eta_k+1}^{\eta_k+\widetilde{r}}\vert\vert f_{(\eta_k,e_k]} \ast \mathcal{K}_{{h_2}}-F_{(\eta_k,e_k],{h_2}}\vert\vert_{L_2}^2,\ \text{and},
  \\
&I_{5,2}=2\sum_{t=\eta_k+1}^{\eta_k+\widetilde{r}}\langle f_{(\eta_k,e_k]} \ast \mathcal{K}_{{h_2}}-F_{(\eta_k,e_k],{h_2}} ,F_{t,{h_2}}-f_{(\eta_k,e_k]} \ast \mathcal{K}_{{h_2}}\rangle_{L_2}.
\end{align*}
The term $I_{5,1}=O_p\Big(\widetilde{r}\frac{1}{T}\frac{\log(T)}{\kappa^{\frac{p}{r}}}\Big)$, using \Cref{remark-1}. Even more, by \Cref{assume-snr2}, we get
\begin{equation}
\label{tig-2-1}
    I_{5,1}=o_p(\widetilde{r}\kappa_k^{\frac{p}{r}+2}).
\end{equation}
For the term $I_{5,2}$, by Cauchy Schwartz inequality,
\begin{align*}
    &\langle f_{(\eta_k,e_k]} \ast \mathcal{K}_{{h_2}}-F_{(\eta_k,e_k],{h_2}} ,F_{t,{h_2}}-f_{(\eta_k,e_k]} \ast \mathcal{K}_{{h_2}}\rangle_{L_2}
    \\
    &\le
    \vert\vert f_{(\eta_k,e_k]} \ast \mathcal{K}_{{h_2}}-F_{(\eta_k,e_k],{h_2}}\vert\vert_{L_2} \vert\vert F_{t,{h_2}}-f_{(\eta_k,e_k]} \ast \mathcal{K}_{{h_2}}\vert\vert_{L_2}
\end{align*}
for any $t \in{(\eta_k,\eta_k+\widetilde{r}]} $. By \Cref{remark-1}, we have that $\vert\vert f_{(
\eta_k,e_k]} \ast \mathcal{K}_{{h_2}}-F_{(\eta_k,e_k],{h_2}}\vert\vert_{L_2}=
O_p\Big(\frac{1}{\sqrt{T}}\sqrt{\frac{\log(T)}{\kappa_k^{\frac{p}{r}}}}\Big)$. Therefore,
\begin{align*}
    I_{5,2}&\le O_p\Big(\frac{1}{\sqrt{T}}\sqrt{\frac{\log(T)}{\kappa_k^{\frac{p}{r}}}}\Big)\Big( \sum_{t=\eta_k+1}^{\eta_k+\widetilde{r}}\vert\vert F_{t,{h_2}}-f_{[\eta_{k}+1,e_k]} \ast \mathcal{K}_{{h_2}}\vert\vert_{L_2}\Big)
\end{align*}
Now, similarly to the bound for $I_2,$ we consider the random variable
$$\bar{Y}_i=\frac{\langle F_{t,{h_2}}-f_{(\eta_k,e_k]}\ast \mathcal{K}_\kappa,F_{t,{h_2}}-f_{(\eta_k,e_k]}\ast \mathcal{K}_{h_2}\rangle _{L_2}^{\frac{1}{2}}-\mathbb{E}(\vert\vert F_{t,{h_2}}-f_{[\eta_{k}+1,e_k]} \ast \mathcal{K}_{{h_2}}\vert\vert_{L_2})}{\mathbb{E}(\vert\vert F_{t,{h_2}}-f_{\eta_{k+1}}\ast \mathcal{K}_{h_2}\vert\vert_{L^2}^3)^{\frac{1}{3}}}.$$
In order to use \Cref{lemma3}, we observe
\begin{align*}
    \mathbb{E}(|\bar{Y}_i|^3)=\frac{
    \mathbb{E}(\vert\vert F_{t,{h_2}}-f_{\eta_{k+1}}\ast \mathcal{K}_{h_2}\vert\vert_{L^2}^3)}{\mathbb{E}(\vert\vert F_{t,{h_2}}-f_{\eta_{k+1}}\ast \mathcal{K}_{h_2}\vert\vert_{L^2}^3)}
    = 1.
\end{align*}
so that, by \Cref{lemma3}, 
\begin{align}
\label{tig-2-2}
    I_{5,2}&\le O_p\Big(\frac{1}{\sqrt{T}}\sqrt{\frac{\log(T)}{\kappa^{\frac{p}{r}}}}\Big)\Big( O_p(\sqrt{\widetilde{r}\kappa_k^{-\frac{p}{r}}}(\log(\widetilde{r}\kappa_k^{\frac{p}{r}+2})+1))+O_p(\kappa_k^{\frac{-p}{2r}})\Big)=o_p(\widetilde{r}\kappa_k^{\frac{p}{r}+2}).
\end{align}
To analyze $I_6,$ we observe that
\begin{align*}
    I_6=& \sum_{t=\eta_k+1}^{\eta_k+\widetilde{r}}\vert\vert F_{t,{h_2}}-F_{(\eta_k,e_k],{h_2}}\vert\vert_{L_2}^2-\sum_{t=\eta_k+1}^{\eta_k+\widetilde{r}}\vert\vert F_{t,{h_1}}-F_{(\eta_k,e_k],{h_1}}\vert\vert_{L_2}^2
    \\
    =&\sum_{t=\eta_k+1}^{\eta_k+\widetilde{r}}\Big[\langle F_{t,h_2},F_{t,h_2}\rangle_{L_2}-2\langle F_{t,h_2}, F_{(\eta_k,e_k],{h_2}}\rangle_{L_2}+\langle F_{(\eta_k,e_k],{h_2}},F_{(\eta_k,e_k],{h_2}} 
    \rangle_{L_2}\Big]
    \\
    &\sum_{t=\eta_k+1}^{\eta_k+\widetilde{r}}\Big[-\langle F_{t,h_1},F_{t,h_1}\rangle_{L_2}+2\langle F_{t,h_1}, F_{(\eta_k,e_k],{h_1}}\rangle_{L_2}-\langle F_{(\eta_k,e_k],{h_1}},F_{(\eta_k,e_k],{h_1}} 
    \rangle_{L_2}\Big]\\
    =&I_{6,1}+I_{6,2}+I_{6,3},
\end{align*}
where,
\begin{align*}
    &I_{6,1}=\sum_{t=\eta_k+1}^{\eta_k+\widetilde{r}}\langle F_{t,h_2},F_{t,h_2}\rangle_{L_2}-\langle F_{t,h_1},F_{t,h_1}\rangle_{L_2},\\
    &I_{6,2}=\sum_{t=\eta_k+1}^{\eta_k+\widetilde{r}}2\langle F_{t,h_1}, F_{(\eta_k,e_k],{h_1}}\rangle_{L_2}-2\langle F_{t,h_2}, F_{(\eta_k,e_k],{h_2}}\rangle_{L_2}
    \\
    &I_{6,3}=\sum_{t=\eta_k+1}^{\eta_k+\widetilde{r}}\langle F_{(\eta_k,e_k],{h_2}},F_{(\eta_k,e_k],{h_2}} 
    \rangle_{L_2}-\langle F_{(\eta_k,e_k],{h_1}},F_{(\eta_k,e_k],{h_1}} 
    \rangle_{L_2}.
\end{align*}
Then we bound each of these terms. First, we rewrite $I_{6,1},$ as
\begin{align*}
    &\sum_{t=\eta_k+1}^{\eta_k+\widetilde{r}}\langle F_{t,h_2},F_{t,h_2}\rangle_{L_2}-\langle F_{t,h_1},F_{t,h_1}\rangle_{L_2}
    \\
    =&\sum_{t=\eta_k+1}^{\eta_k+\widetilde{r}}\langle F_{t,h_2},F_{t,h_2}\rangle_{L_2}-\langle F_{t,h_1},F_{t,h_1}\rangle_{L_2}+
    \langle F_{t,h_1},F_{t,h_2}\rangle_{L_2}-\langle F_{t,h_1},F_{t,h_2}\rangle_{L_2}
    \\
    =&
    \sum_{t=\eta_k+1}^{\eta_k+\widetilde{r}}\langle F_{t,h_2}-F_{t,h_1},F_{t,h_2}\rangle_{L_2}+\langle F_{t,h_1},F_{t,h_2}-F_{t,h_1}\rangle_{L_2}
\end{align*}
which, by H\"{o}lder’s inequality, is bounded by
\begin{align*}
     \sum_{t=\eta_k+1}^{\eta_k+\widetilde{r}}\vert\vert F_{t,h_2}\vert\vert_{L_2}\vert\vert F_{t,h_2}-F_{t,h_1}\vert\vert_{L_2}+\vert\vert F_{t,h_1}\vert\vert_{L_2}\vert\vert F_{t,h_2}-F_{t,h_1}\vert\vert_{L_2}
     =\widetilde{r}O_p(\frac{T^{-\frac{r}{2r+p}}}{\kappa_k^{\frac{p}{2r}+\frac{1}{2}+\frac{p}{2r}}}\log^{\frac{r}{2r+p}}(T)))
\end{align*}
since $\vert\vert F_{t,\kappa}-F_{t,\widehat{\kappa}}\vert\vert_{L_2}^2=O(\frac{|\kappa-\widehat{\kappa}|^{\frac{1}{2}}}{\kappa_k^{\frac{p}{2r}+\frac{1}{2}}})=O_p(\frac{T^{-\frac{r}{2r+p}}}{\kappa_k^{\frac{p}{2r}+\frac{1}{2}}}\log^{\frac{r}{2r+p}}(T))),$ for any $t$, see \Cref{remark2} for more detail. 
Similarly, for $I_{6,2}$ we have,
\begin{align*}
    &\sum_{t=\eta_k+1}^{\eta_k+\widetilde{r}}2\langle F_{t,h_1}, F_{(\eta_k,e_k],{h_1}}\rangle_{L_2}-2\langle F_{t,h_2}, F_{(\eta_k,e_k],{h_2}}\rangle_{L_2}
    \\
    =&\sum_{t=\eta_k+1}^{\eta_k+\widetilde{r}}2\langle F_{t,h_1}, F_{(\eta_k,e_k],{h_1}}\rangle_{L_2}-2\langle F_{t,h_2}, F_{(\eta_k,e_k],{h_2}}\rangle_{L_2}+
    2\langle F_{t,h_1}, F_{(\eta_k,e_k],{h_2}}\rangle_{L_2}-2\langle F_{t,h_1}, F_{(\eta_k,e_k],{h_2}}\rangle_{L_2}
    \\
    =&
    \sum_{t=\eta_k+1}^{\eta_k+\widetilde{r}}2\langle F_{t,h_1}-F_{t,h_2}, F_{(\eta_k,e_k],{h_2}}\rangle_{L_2}+2\langle F_{t,h_1},F_{(\eta_k,e_k],{h_2}}-F_{(\eta_k,e_k],{h_1}}\rangle_{L_2}
\end{align*}
and by H\"{o}lder’s inequality and \Cref{remark2}, it is bounded by
\begin{align*}
     &\sum_{t=\eta_k+1}^{\eta_k+\widetilde{r}}\vert\vert F_{t,h_1}-F_{t,h_2}\vert\vert_{L_2}\vert\vert F_{(\eta_k,e_k],{h_2}}\vert\vert_{L_2}+\vert\vert F_{t,h_1}\vert\vert_{L_2}\vert\vert F_{(\eta_k,e_k],{h_2}}-F_{(\eta_k,e_k],{h_1}}\vert\vert_{L_2}
     \\
     =&\widetilde{r}O_p(\frac{T^{-\frac{r}{2r+p}}}{\kappa_k^{\frac{p}{2r}+\frac{1}{2}+\frac{p}{2r}}}\log^{\frac{r}{2r+p}}(T))).
\end{align*}
Now for $I_{6,3}$, we write it as
\begin{align*}
&\sum_{t=\eta_k+1}^{\eta_k+\widetilde{r}}\langle F_{(\eta_k,e_k],{h_2}},F_{(\eta_k,e_k],{h_2}} 
    \rangle_{L_2}-\langle F_{(\eta_k,e_k],{h_1}},F_{(\eta_k,e_k],{h_1}} 
    \rangle_{L_2}
    \\   =&\sum_{t=\eta_k+1}^{\eta_k+\widetilde{r}}\langle F_{(\eta_k,e_k],{h_2}},F_{(\eta_k,e_k],{h_2}} 
    \rangle_{L_2}-\langle F_{(\eta_k,e_k],{h_1}},F_{(\eta_k,e_k],{h_1}} 
    \rangle_{L_2}+
    \langle F_{(\eta_k,e_k],{h_1}},F_{(\eta_k,e_k],{h_2}} 
    \rangle_{L_2}-\langle F_{(\eta_k,e_k],{h_1}},F_{(\eta_k,e_k],{h_2}} 
    \rangle_{L_2}
    \\
=&\sum_{t=\eta_k+1}^{\eta_k+\widetilde{r}}\langle F_{(\eta_k,e_k],{h_2}}-F_{(\eta_k,e_k],{h_1}},F_{(s_k,\eta_k],{h_2}}\rangle_{L_2}+\langle F_{(\eta_k,e_k],{h_1}},F_{(\eta_k,e_k],{h_2}}-F_{(\eta_k,e_k],{h_1}}\rangle_{L_2}
\end{align*}
which, by H\"{o}lder’s inequality and \Cref{remark2}, is bounded by
\begin{align*}
&\sum_{t=\eta_k+1}^{\eta_k+\widetilde{r}}\vert\vert F_{(\eta_k,e_k],{h_2}}\vert\vert_{L_2}\vert\vert F_{(\eta_k,e_k],{h_2}}-F_{(\eta_k,e_k],{h_2}}\vert\vert_{L_2}+\vert\vert F_{(\eta_k,e_k],{h_1}}\vert\vert_{L_2}\vert\vert F_{(\eta_k,e_k],{h_2}}-F_{(\eta_k,e_k],{h_1}}\vert\vert_{L_2}
\\
     =&\widetilde{r}O_p(\frac{T^{-\frac{r}{2r+p}}}{\kappa_k^{\frac{p}{2r}+\frac{1}{2}+\frac{p}{2r}}}\log^{\frac{r}{2r+p}}(T)))
\end{align*}
By above and \Cref{assume-snr2}, we conclude 
\begin{equation}
\label{tig-2-3}
    I_6=o_p(\widetilde{r}\kappa_k^{\frac{p}{r}+2}).
\end{equation}
From, \eqref{tig-2-1}, \eqref{tig-2-2} and \eqref{tig-2-3}, we get that $U_2$ is bounded by
\begin{align*}
    \sum_{t=\eta_k+1}^{\eta_k+\widetilde{r}}\vert\vert F_{t,\widehat{\kappa}}-F_{(\eta_k,e_k],\widehat{\kappa}}\vert\vert_{L_2}^2
    -\sum_{t=\eta_k+1}^{\eta_k+\widetilde{r}}\vert\vert F_{t,\kappa}-f_{(\eta_k,e_k]}\ast \mathcal{K}_\kappa\vert\vert_{L_2}^2=o_p(\widetilde{r}\kappa_k^{\frac{p}{r}+2})    
\end{align*}
Therefore, from \eqref{tightness-1} and \eqref{tightness-2}
\begin{equation}
\label{upp-bound}
\widehat{Q}_k(\eta_k)-\widehat{Q}_k(\eta_k+\widetilde{r})-Q^*(\eta_k)+Q^*(\eta_k+\widetilde{r})=o_p(\widetilde{r}\kappa_k^{\frac{p}{r}+2})
\end{equation}
{\bf{Step 4: Combination of all the steps above.}}
Finally, combining \eqref{bounds-req}, \eqref{low-bound} and \eqref{upp-bound}, uniformly for any $\widetilde{r}\geq \frac{1}{\kappa_k^{\frac{p}{r}+2}}$ we have that
\begin{align*}
    \frac{1}{2}\widetilde{r}\kappa_k^2-O_p\Big(\sqrt{\widetilde{r}}\kappa_{k}\kappa_k^{-\frac{p}{2r}} (\log(\widetilde{r}\kappa_k^{\frac{p}{r}+2})+1)\Big)-o_p(\widetilde{r}\kappa_k^{\frac{p}{r}+2}) \le& o_p(\widetilde{r}\kappa_k^{\frac{p}{r}+2})
\end{align*}
which implies,
\begin{equation}
\label{tightness}
    \widetilde{r}\kappa_k^{\frac{p}{r}+2}=O_p(1)
\end{equation}
and complete the proofs of ${\bf{a.1}}$ and ${\bf{b.1}}$.

{\bf{Limiting distributions.}} For any $k \in\{1, \ldots, K\}$, due to the uniform tightness of $\widetilde{r} \kappa_k^{\frac{p}{r}+2}$, \eqref{bounds-req} and \eqref{upp-bound}, as $T \rightarrow \infty$
$$
Q^*(\eta)=\sum_{t=s_k+1}^{\eta}\vert\vert F_{t,{h_2}}-f_{(s_k,\eta_k]} \ast \mathcal{K}_{{h_2}}\vert\vert_{L_2}^2+\sum_{t=\eta+1}^{e_k}\vert\vert F_{t,{h_2}}-f_{(\eta_k,e_k]}\ast \mathcal{K}_{h_2}\vert\vert_{L_2}^2,
$$
satisfies
$$
\Big|\widehat{Q}\Big(\eta_k+\widetilde{r}\Big)-\widehat{Q}\Big(\eta_k\Big)-\Big(Q^*\Big(\eta_k+\widetilde{r}\Big)-Q^*\Big(\eta_k\Big)\Big)\Big|  \stackrel{p}{\rightarrow} 0 .
$$
Therefore, it is sufficient to find the limiting distributions of $Q^*\Big(\eta_k+\widetilde{r}\Big)-Q^*\Big(\eta_k\Big)$ when $T \rightarrow \infty$.
{\bf{Non-vanishing regime.}}
Observe that for $\widetilde{r} > 0$, we have that when $T \rightarrow \infty$,
\begin{align*}
    Q^*(\eta_k+\widetilde{r})-Q^*(\eta_k)=&\sum_{t=\eta_k+1}^{\eta_k+\widetilde{r}}\vert\vert F_{t,{h_2}}-f_{(s_k,\eta_k]} \ast \mathcal{K}_{{h_2}}\vert\vert_{L_2}^2-\sum_{t=\eta_k+1}^{\eta_k+\widetilde{r}}\vert\vert F_{t,{h_2}}-f_{(\eta_k,e_k]}\ast \mathcal{K}_{h_2}\vert\vert_{L_2}^2
    \\
    =&\sum_{t=\eta_k+1}^{\eta_k+\widetilde{r}}\vert\vert f_{(s_k,\eta_k]} \ast \mathcal{K}_{{h_2}}-f_{(\eta_k,e_k]}\ast \mathcal{K}_{h_2}\vert\vert_{L_2}^2
    \\
    &-2\sum_{t=\eta_k+1}^{\eta_k+\widetilde{r}}\langle f_{(s_k,\eta_k]} \ast \mathcal{K}_{{h_2}}-f_{(\eta_k,e_k]}\ast \mathcal{K}_{h_2},F_{t,{h_2}}-f_{(\eta_k,e_k]}\ast \mathcal{K}_{h_2}\rangle _{L_2}
    \\
    \ \underrightarrow{\mathcal{D}} \ &
    \sum_{t=1}^{\widetilde{r}} 2\Big\langle F_{{h_2},t}-f_t*\mathcal{K}_{h_2}, (f_{\eta_{k+1}}-f_{\eta_{k}})*\mathcal{K}_{h_2}\Big\rangle_{L_2}+\widetilde{r}\vert\vert  (f_{\eta_{k+1}}-f_{\eta_{k}})*\mathcal{K}_{h_2}\vert\vert_{L_2}^2 .
\end{align*}
When $\widetilde{r}<0$ and $T\rightarrow \infty$, we have that
\begin{align*}
    Q^*(\eta_k+\widetilde{r})-Q^*(\eta_k)=&\sum_{t=\eta_k+\widetilde{r}}^{\eta_k-1}\vert\vert F_{t,{h_2}}-f_{(s_k,\eta_k]} \ast \mathcal{K}_{{h_2}}\vert\vert_{L_2}^2-\sum_{t=\eta_k+\widetilde{r}}^{\eta_k-1}\vert\vert F_{t,{h_2}}-f_{(\eta_k,e_k]}\ast \mathcal{K}_{h_2}\vert\vert_{L_2}^2
    \\
    =&\sum_{t=\eta_k+\widetilde{r}}^{\eta_k-1}\vert\vert f_{(s_k,\eta_k]} \ast \mathcal{K}_{{h_2}}-f_{(\eta_k,e_k]}\ast \mathcal{K}_{h_2}\vert\vert_{L_2}^2
    \\
    &-2\sum_{t=\eta_k+\widetilde{r}}^{\eta_k-1}\langle f_{(s_k,\eta_k]} \ast \mathcal{K}_{{h_2}}-f_{(\eta_k,e_k]}\ast \mathcal{K}_{h_2},F_{t,{h_2}}-f_{(\eta_k,e_k]}\ast \mathcal{K}_{h_2}\rangle _{L_2}
    \\
    \ \underrightarrow{\mathcal{D}} \ &
    \sum_{t=\widetilde{r}+1}^0 2\Big\langle F_{{h_2},t}-f_t*\mathcal{K}_{h_2}, (f_{\eta_{k}}-f_{\eta_{k+1}})*\mathcal{K}_{h_2}\Big\rangle_{L_2}+\widetilde{r}\vert\vert  (f_{\eta_{k+1}}-f_{\eta_{k}})*\mathcal{K}_{h_2}\vert\vert_{L_2}^2.
\end{align*}
Therefore, using Slutsky’s theorem and the Argmax (or Argmin) continuous mapping theorem (see 3.2.2 Theorem
van der Vaart and Wellner, 1996) we conclude 
\begin{equation}
    ( \widetilde{\eta}_k- \eta_k)\kappa_k^{\frac{p}{r}+2} \ \underrightarrow{\mathcal{D}} \ 
    \underset{\widetilde{r}\in{\mathbb{Z}}}{\arg \min } P_k(\widetilde{r})
\end{equation}
{\bf{Vanishing regime.}}
Vanishing regime. Let $m=\kappa_k^{-2-\frac{p}{r}}$, and we have that $m \rightarrow \infty$ as $T \rightarrow \infty$. Observe that for $\widetilde{r}>0$, we have that
$$
\begin{aligned}
 Q_k^*\Big(\eta_k+\widetilde{r} m\Big)-Q_k^*\Big(\eta_k\Big)=&\sum_{t=\eta_k}^{\eta_k+\widetilde{r} m-1}\vert\vert f_{(s_k,\eta_k]} \ast \mathcal{K}_{{h_2}}-f_{(\eta_k,e_k]}\ast \mathcal{K}_{h_2}\vert\vert_{L_2}^2\\
- & 2\sum_{t=\eta_k}^{\eta_k+\widetilde{r} m-1} \langle f_{(s_k,\eta_k]} \ast \mathcal{K}_{{h_2}}-f_{(\eta_k,e_k]}\ast \mathcal{K}_{h_2},F_{t,{h_2}}-f_{(\eta_k,e_k]}\ast \mathcal{K}_{h_2}\rangle _{L_2}
\end{aligned}
$$
Following the Central Limit Theorem for
$\alpha-$mixing, see \Cref{{CTL}}, we get
$$
\frac{1}{\sqrt{m}} \sum_{t=\eta_k}^{\eta_k+r m-1}\frac{\langle f_{(s_k,\eta_k]} \ast \mathcal{K}_{{h_2}}-f_{(\eta_k,e_k]}\ast \mathcal{K}_{h_2},F_{t,{h_2}}-f_{(\eta_k,e_k]}\ast \mathcal{K}_{h_2}\rangle _{L_2}}{\kappa_k^{\frac{p}{2r}+1}}\stackrel{\mathcal{D}}{\rightarrow} \kappa_k^{-\frac{p}{r}}\widetilde{\sigma}_{\infty}(k) \mathbb{B}(\widetilde{r}),
$$
where $\mathbb{B}(\widetilde{r})$ is a standard Brownian motion and $\widetilde{\sigma}(k)$ is the long-run variance given in \eqref{long-run-var}. 
Therefore, it holds that when $T \rightarrow \infty$
$$
Q_k^*\Big(\eta_k+\widetilde{r} m\Big)-Q_k^*\Big(\eta_k\Big) \stackrel{\mathcal{D}}{\rightarrow} \kappa_k^{-\frac{p}{r}}\widetilde{\sigma}_{\infty}(k)\mathbb{B}_1(r)+\widetilde{r}\kappa_k^{-\frac{p}{r}-2}\vert\vert f_{(s_k,\eta_k]} \ast \mathcal{K}_{{h_2}}-f_{(\eta_k,e_k]}\ast \mathcal{K}_{h_2}\vert\vert_{L_2}^2.
$$
Similarly, for $\widetilde{r}<0$, we have that when $n \rightarrow \infty$
$$
Q_k^*\Big(\eta_k+r m\Big)-Q_k^*\Big(\eta_k\Big) \stackrel{\mathcal{D}}{\rightarrow}\kappa_k^{-\frac{p}{r}}\widetilde{\sigma}_{\infty}(k) \mathbb{B}_1(-\widetilde{r})-\widetilde{r}\kappa_k^{-\frac{p}{r}-2}\vert\vert f_{(s_k,\eta_k]} \ast \mathcal{K}_{{h_2}}-f_{(\eta_k,e_k]}\ast \mathcal{K}_{h_2}\vert\vert_{L_2}^2. .
$$
Then, using Slutsky's theorem and the Argmax (or Argmin) continuous mapping theorem (see 3.2.2 Theorem in \cite{vanderVaart1996}), and the fact that, $\mathbb{E}(\vert\vert f_{(s_k,\eta_k]} \ast \mathcal{K}_{{h_2}}-f_{(\eta_k,e_k]}\ast \mathcal{K}_{h_2}\vert\vert_{L_2}^2)=O(\kappa_k^2)$, we conclude that
$$
\kappa_k^{2+\frac{p}{r}}\Big(\widetilde{\eta}_k-\eta_k\Big) \stackrel{\mathcal{D}}{\longrightarrow}    \underset{r\in{\mathbb{Z}}}{\arg \min } \ \widetilde{\sigma}_{\infty}(k)B(\widetilde{r})+\vert\widetilde{r}\vert,
$$
which completes the proof of $\mathbf{b} . \mathbf{2}$.
\end{proof}


\newpage
\section{Proof of \Cref{Long-Run-V-Theorem}}\label{sec-proof-thm3}
In this section, we present the proof of theorem \Cref{Long-Run-V-Theorem}.
\begin{proof}[Proof of \Cref{Long-Run-V-Theorem}]

First, letting $h_2=c_\kappa\kappa_k^{\frac{1}{r}}$ and $R=O(\frac{T^{\frac{p+r}{2r+p}}}{\kappa_k^{\frac{p}{2r}+\frac{3}{2}}})$, we consider
\begin{align}
\label{Long-run-var-Est-aux}
\breve{\sigma}_{\infty}^2(k)=\frac{1}{R}\sum_{r=1}^{R}\Big(\frac{1}{\sqrt{S}}\sum_{i\in{\mathcal{S}_r}}\breve{Y}_i\Big)^2, \ \text{where}, \ \breve{Y}_i=\kappa_k^{\frac{p}{2r}-1}\Big\langle F_{{h_2},i}-f_i*\mathcal{K}_{h_2}, (f_{\eta_{k}}-f_{\eta_{k+1}})*\mathcal{K}_{h_2}\Big\rangle_{L_2}.
\end{align}
We will show that 
\begin{itemize}
    \item[(i)] $\Big|\widehat{\sigma}_{\infty}^2(k)-\breve{\sigma}_{\infty}^2(k)\Big| \stackrel{P }{\longrightarrow} 0, \quad T \rightarrow \infty $, and
    \item[(ii)] $\Big|\breve{\sigma}_{\infty}^2(k)-\widetilde{\sigma}_{\infty}^2(k)\Big| \stackrel{P }{\longrightarrow} 0, \quad T \rightarrow \infty $
\end{itemize}
in order to conclude the result. For (i), we use $a^2-b^2=(a+b)(a-b)$, to write,
\begin{align*}
    \Big|\widehat{\sigma}_{\infty}^2(k)-\breve{\sigma}_{\infty}^2(k)\Big|=&\Big|\frac{1}{R}\sum_{r=1}^{R}\Big(\frac{1}{\sqrt{S}}\sum_{i\in{\mathcal{S}_r}}\breve{Y}_i\Big)^2-\frac{1}{R}\sum_{r=1}^{R}\Big(\frac{1}{\sqrt{S}}\sum_{i\in{\mathcal{S}_r}}Y_i\Big)^2\Big|
    \\
    =&
    \Big|\frac{1}{R}\sum_{r=1}^{R}\Big(\frac{1}{\sqrt{S}}\sum_{i\in{\mathcal{S}_r}}\breve{Y}_i-Y_i\Big)\Big(\frac{1}{\sqrt{S}}\sum_{i\in{\mathcal{S}_r}}\breve{Y}_i+Y_i\Big)\Big|
    \\
    =&\Big|\frac{1}{R}\sum_{r=1}^{R}I_1I_2\Big|
\end{align*}
Then, we bound each of the terms $I_1$ and $I_2.$ For $I_1$, we observe that,
\begin{align*}
    I_1=\Big|\frac{1}{\sqrt{S}}\sum_{i\in{\mathcal{S}_r}}\breve{Y}_i-Y_i\Big|\le \frac{1}{\sqrt{S}}\sum_{i\in{\mathcal{S}_r}}\Big|\breve{Y}_i-Y_i\Big|.
\end{align*}
Then, adding and subtracting, $\widehat{\kappa}_k^{\frac{p}{2r}-1}\Big\langle F_{{h_1},i}-f_i*\mathcal{K}_{h_1}, (f_{\eta_{k}}-f_{\eta_{k+1}})*\mathcal{K}_{h_2}\Big\rangle_{L_2}$ and $$\widehat{\kappa}_k^{\frac{p}{2r}-1}\Big\langle F_{{h_2},i}-f_i*\mathcal{K}_{h_2}-F_{{h_1},i}+f_i*\mathcal{K}_{h_1}, (f_{\eta_{k}}-f_{\eta_{k+1}})*\mathcal{K}_{h_1}\Big\rangle_{L_2},$$
we get that,
\begin{align*}
    &\Big|\breve{Y}_i-Y_i\Big|
    \\
    =&\Big|\kappa_k^{\frac{p}{2r}-1}\Big\langle F_{{h_2},i}-f_i*\mathcal{K}_{h_2}, (f_{\eta_{k}}-f_{\eta_{k+1}})*\mathcal{K}_{h_2}\Big\rangle_{L_2}-\widehat{\kappa}_k^{\frac{p}{2r}-1}\Big\langle F_{{h_1},i}-f_i*\mathcal{K}_{h_1}, (f_{\eta_{k}}-f_{\eta_{k+1}})*\mathcal{K}_{h_1}\Big\rangle_{L_2}\Big|
    \\
    =&
    \Big|\kappa_k^{\frac{p}{2r}-1}\Big\langle F_{{h_2},i}-f_i*\mathcal{K}_{h_2}, (f_{\eta_{k}}-f_{\eta_{k+1}})*\mathcal{K}_{h_2}\Big\rangle_{L_2}-\widehat{\kappa}_k^{\frac{p}{2r}-1}\Big\langle F_{{h_1},i}-f_i*\mathcal{K}_{h_1}, (f_{\eta_{k}}-f_{\eta_{k+1}})*\mathcal{K}_{h_2}\Big\rangle_{L_2}
    \\
    +&
    \widehat{\kappa}_k^{\frac{p}{2r}-1}\Big\langle F_{{h_1},i}-f_i*\mathcal{K}_{h_1}, (f_{\eta_{k}}-f_{\eta_{k+1}})*\mathcal{K}_{h_2}\Big\rangle_{L_2}-\widehat{\kappa}_k^{\frac{p}{2r}-1}\Big\langle F_{{h_1},i}-f_i*\mathcal{K}_{h_1}, (f_{\eta_{k}}-f_{\eta_{k+1}})*\mathcal{K}_{h_1}\Big\rangle_{L_2}\Big|
    \\
    +&
    \widehat{\kappa}_k^{\frac{p}{2r}-1}\Big\langle F_{{h_2},i}-f_i*\mathcal{K}_{h_2}-F_{{h_1},i}+f_i*\mathcal{K}_{h_1}, (f_{\eta_{k}}-f_{\eta_{k+1}})*\mathcal{K}_{h_1}\Big\rangle_{L_2}
    \\
    -&\widehat{\kappa}_k^{\frac{p}{2r}-1}\Big\langle F_{{h_2},i}-f_i*\mathcal{K}_{h_2}-F_{{h_1},i}+f_i*\mathcal{K}_{h_1}, (f_{\eta_{k}}-f_{\eta_{k+1}})*\mathcal{K}_{h_1}\Big\rangle_{L_2}
\end{align*}
which can be written as,
\begin{align*}
   &\Big|\Big\langle \kappa_k^{\frac{p}{2r}-1}(F_{{h_2},i}-f_i*\mathcal{K}_{h_2})-\widehat{\kappa}_k^{\frac{p}{2r}-1}(F_{{h_1},i}-f_i*\mathcal{K}_{h_1}), (f_{\eta_{k}}-f_{\eta_{k+1}})*\mathcal{K}_{h_2}-(f_{\eta_{k}}-f_{\eta_{k+1}})*\mathcal{K}_{h_1}\Big\rangle_{L_2}
   \\
   +&\Big\langle \widehat{\kappa}_k^{\frac{p}{2r}-1}(F_{{h_1},i}-f_i*\mathcal{K}_{h_1})-\kappa_k^{\frac{p}{2r}-1}(F_{{h_2},i}-f_i*\mathcal{K}_{h_2}), (f_{\eta_{k}}-f_{\eta_{k+1}})*\mathcal{K}_{h_1}-(f_{\eta_{k}}-f_{\eta_{k+1}})*\mathcal{K}_{h_2}\Big\rangle_{L_2}
   \\
   +&\Big\langle \kappa_k^{\frac{p}{2r}-1}(F_{{h_2},i}-f_i*\mathcal{K}_{h_2}), (f_{\eta_{k}}-f_{\eta_{k+1}})*\mathcal{K}_{h_2}-(f_{\eta_{k}}-f_{\eta_{k+1}})*\mathcal{K}_{h_1}\Big\rangle_{L_2}
   \\
   +&\Big\langle \widehat{\kappa}_k^{\frac{p}{2r}-1}(F_{{h_1},i}-f_i*\mathcal{K}_{h_1})-\kappa_k^{\frac{p}{2r}-1}(F_{{h_2},i}-f_i*\mathcal{K}_{h_2}), (f_{\eta_{k}}-f_{\eta_{k+1}})*\mathcal{K}_{h_1}\Big\rangle_{L_2}\Big|.
\end{align*}
Now, we bound the expression above. For this purpose, by triangle inequality, it is enough to bound each of the terms above. Then, we use H\"{o}lder's inequality. First,
\begin{align*}
    &\Big|\Big\langle \kappa_k^{\frac{p}{2r}-1}(F_{{h_2},i}-f_i*\mathcal{K}_{h_2})-\widehat{\kappa}_k^{\frac{p}{2r}-1}(F_{{h_1},i}-f_i*\mathcal{K}_{h_1}), (f_{\eta_{k}}-f_{\eta_{k+1}})*\mathcal{K}_{h_2}-(f_{\eta_{k}}-f_{\eta_{k+1}})*\mathcal{K}_{h_1}\Big\rangle_{L_2}\Big|
    \\
    &\le
    \vert \kappa_k^{\frac{p}{2r}-1}-\widehat{\kappa}_k^{\frac{p}{2r}-1}\vert\vert\vert F_{{h_2},i}-f_i*\mathcal{K}_{h_2}-F_{{h_1},i}+f_i*\mathcal{K}_{h_1}\vert\vert_{L_2}\vert\vert (f_{\eta_{k}}-f_{\eta_{k+1}})*\mathcal{K}_{h_2}-(f_{\eta_{k}}-f_{\eta_{k+1}})*\mathcal{K}_{h_1}\vert\vert_{L_2}.
\end{align*}
Then, using \eqref{kappa-b}, we have that $\vert \kappa_k^{\frac{p}{2r}-1}-\widehat{\kappa}_k^{\frac{p}{2r}-1}\vert=O_p(\frac{T^{-\frac{2r}{2r+p}}}{\kappa_k^{2-\frac{p}{2r}}}\log^{\frac{2r}{2r+p}}(T))$, and using \Cref{remark2}, it follows that
\begin{align*}
    \vert\vert F_{{h_2},i}-f_i*\mathcal{K}_{h_2}-F_{{h_1},i}+f_i*\mathcal{K}_{h_1}\vert\vert_{L_2}
    \le&
    \vert\vert F_{{h_2},i}-F_{{h_1},i}\vert\vert_{L_2}+\vert\vert f_i*\mathcal{K}_{h_1}-f_i*\mathcal{K}_{h_2}\vert\vert_{L_2}
    \\
    =&O_p(\frac{T^{-\frac{r}{2r+p}}}{\kappa_k^{\frac{p}{2r}+\frac{1}{2}}}\log^{\frac{r}{2r+p}}(T))
\end{align*}
and,
\begin{align*}
    &\vert\vert (f_{\eta_{k}}-f_{\eta_{k+1}})*\mathcal{K}_{h_2}-(f_{\eta_{k}}-f_{\eta_{k+1}})*\mathcal{K}_{h_1}\vert\vert_{L_2}
    \\
    \le&
    \vert\vert f_{\eta_{k}}*\mathcal{K}_{h_2}-f_{\eta_{k}}*\mathcal{K}_{h_1}\vert\vert_{L_2}+\vert\vert f_{\eta_{k+1}}*\mathcal{K}_{h_1}-f_{\eta_{k+1}}*\mathcal{K}_{h_2}\vert\vert_{L_2}
    =O_p(\frac{T^{-\frac{r}{2r+p}}}{\kappa_k^{\frac{p}{2r}+\frac{1}{2}}}\log^{\frac{r}{2r+p}}(T)).
\end{align*}
So that,
\begin{align*}
    &\Big|\Big\langle \kappa_k^{\frac{p}{2r}-1}(F_{{h_2},i}-f_i*\mathcal{K}_{h_2})-\widehat{\kappa}_k^{\frac{p}{2r}-1}(F_{{h_1},i}-f_i*\mathcal{K}_{h_1}), (f_{\eta_{k}}-f_{\eta_{k+1}})*\mathcal{K}_{h_2}-(f_{\eta_{k}}-f_{\eta_{k+1}})*\mathcal{K}_{h_1}\Big\rangle_{L_2}\Big|
    \\
    &=O_p(\frac{T^{-\frac{2r}{2r+p}}}{\kappa_k^{2-\frac{p}{2r}}}\log^{\frac{2r}{2r+p}}(T))O_p(\frac{T^{-\frac{r}{2r+p}}}{\kappa_k^{\frac{p}{2r}+\frac{1}{2}}}\log^{\frac{r}{2r+p}}(T))O_p(\frac{T^{-\frac{r}{2r+p}}}{\kappa_k^{\frac{p}{2r}+\frac{1}{2}}}\log^{\frac{r}{2r+p}}(T)).
\end{align*}
Now, in a similar way, we observe that 
\begin{align*}
    &\Big\langle \kappa_k^{\frac{p}{2r}-1}(F_{{h_2},i}-f_i*\mathcal{K}_{h_2}), (f_{\eta_{k}}-f_{\eta_{k+1}})*\mathcal{K}_{h_2}-(f_{\eta_{k}}-f_{\eta_{k+1}})*\mathcal{K}_{h_1}\Big\rangle_{L_2}
    \\
    \le& \vert\vert\kappa_k^{\frac{p}{2r}-1}(F_{{h_2},i}-f_i*\mathcal{K}_{h_2})\vert\vert_{L_2}\vert\vert(f_{\eta_{k}}-f_{\eta_{k+1}})*\mathcal{K}_{h_2}-(f_{\eta_{k}}-f_{\eta_{k+1}})*\mathcal{K}_{h_1}\vert\vert_{L_2}
    \\
    =&O_p(\kappa_k^{\frac{p}{2r}-1}\kappa_k^{-\frac{p}{2r}})O_p(\frac{T^{-\frac{r}{2r+p}}}{\kappa_k^{\frac{p}{2r}+\frac{1}{2}}}\log^{\frac{r}{2r+p}}(T))
\end{align*}
where equality is followed by noticing that 
\begin{align}
\label{bound-K}
    \vert\vert F_{{h_2},i}-f_i*\mathcal{K}_{h_2}\vert\vert_{L_2}\le& \vert\vert F_{{h_2},i}\vert\vert_{L_2}+\vert\vert f_i*\mathcal{K}_{h_2}\vert\vert_{L_2}
    \\
    =&O(\kappa_k^{-\frac{p}{2r}})+O(1),
\end{align}
and then using \Cref{remark2} and \Cref{assume: model assumption}.
Finally,
\begin{align*}
    &\Big\langle \widehat{\kappa}_k^{\frac{p}{2r}-1}(F_{{h_1},i}-f_i*\mathcal{K}_{h_1})-\kappa_k^{\frac{p}{2r}-1}(F_{{h_2},i}-f_i*\mathcal{K}_{h_2}), (f_{\eta_{k}}-f_{\eta_{k+1}})*\mathcal{K}_{h_1}\Big\rangle_{L_2}
    \\
    \le&
    \vert\widehat{\kappa}_k^{\frac{p}{2r}-1}-\kappa_k^{\frac{p}{2r}-1}\vert\vert\vert(F_{{h_1},i}-f_i*\mathcal{K}_{h_1})-(F_{{h_2},i}-f_i*\mathcal{K}_{h_2})\vert\vert_{L_2}\vert\vert(f_{\eta_{k}}-f_{\eta_{k+1}})*\mathcal{K}_{h_1}\vert\vert_{L_2}
    \\
    =&O_p(\frac{T^{-\frac{2r}{2r+p}}}{\kappa_k^{2-\frac{p}{2r}}}\log^{\frac{2r}{2r+p}}(T))O_p(\frac{T^{-\frac{r}{2r+p}}}{\kappa_k^{\frac{p}{2r}+\frac{1}{2}}}\log^{\frac{r}{2r+p}}(T))\kappa_k
\end{align*}
where equality is followed by \Cref{remark2}, \Cref{assume: model assumption} and Minkowski’s inequality.
Therefore,
{\small{
\begin{align*}
    &I_1\le
    \frac{1}{\sqrt{S}}\sum_{i\in{\mathcal{S}_r}}\Big|\breve{Y}_i-Y_i\Big|
    \\
    &\sqrt{S}\Big(O_p(\frac{T^{-\frac{2r}{2r+p}}}{\kappa_k^{2-\frac{p}{2r}}}\log(T)^{\frac{2r}{2r+p}})O_p(\frac{T^{-\frac{r}{2r+p}}}{\kappa_k^{\frac{p}{2r}+\frac{1}{2}}}\log(T)^{\frac{r}{2r+p}})O_p(\frac{T^{-\frac{r}{2r+p}}}{\kappa_k^{\frac{p}{2r}+\frac{1}{2}}}\log(T)^{\frac{r}{2r+p}})
    \\
    +&O_p(\kappa_k^{\frac{p}{2r}-1}\kappa_k^{-\frac{p}{2r}})O_p(\frac{T^{-\frac{r}{2r+p}}}{\kappa_k^{\frac{p}{2r}+\frac{1}{2}}}\log^{\frac{r}{2r+p}}(T))+O_p(\frac{T^{-\frac{r}{4r+2p}}}{\kappa_k^{2-\frac{p}{2r}}}\log^{\frac{r}{2r+p}}(T))O_p(\frac{T^{-\frac{r}{2r+p}}}{\kappa_k^{\frac{p}{2r}+\frac{1}{2}}}\log^{\frac{r}{2r+p}}(T))\kappa_k\Big)
    \\
    =&\sqrt{S}\Big(O_p(\frac{T^{-\frac{2r}{2r+p}}}{\kappa_k^{2-\frac{p}{2r}}}\log(T)^{\frac{2r}{2r+p}})O_p(\frac{T^{-\frac{2r}{2r+p}}}{\kappa_k^{\frac{p}{r}+1}}\log(T)^{\frac{2r}{2r+p}})+O_p(\kappa_k^{-1})O_p(\frac{T^{-\frac{r}{2r+p}}}{\kappa_k^{\frac{p}{2r}+\frac{1}{2}}}\log(T)^{\frac{r}{2r+p}})\Big).
\end{align*}}}
To bound the $I_2$ term, we add and subtract $\kappa_k^{\frac{p}{2r}-1}\Big\langle F_{{h_2},i}-f_i*\mathcal{K}_{h_2}, (f_{\eta_{k}}-f_{\eta_{k+1}})*\mathcal{K}_{h_1}\Big\rangle_{L_2}$, to get
\begin{align*}
    &\Big|\breve{Y}_i+Y_i\Big|
    \\
    =&\Big|\kappa_k^{\frac{p}{2r}-1}\Big\langle F_{{h_2},i}-f_i*\mathcal{K}_{h_2}, (f_{\eta_{k}}-f_{\eta_{k+1}})*\mathcal{K}_{h_2}\Big\rangle_{L_2}+\widehat{\kappa}_k^{\frac{p}{2r}-1}\Big\langle F_{{h_1},i}-f_i*\mathcal{K}_{h_1}, (f_{\eta_{k}}-f_{\eta_{k+1}})*\mathcal{K}_{h_1}\Big\rangle_{L_2}\Big|
    \\
    =&
    \Big|\kappa_k^{\frac{p}{2r}-1}\Big\langle F_{{h_2},i}-f_i*\mathcal{K}_{h_2}, (f_{\eta_{k}}-f_{\eta_{k+1}})*\mathcal{K}_{h_2}\Big\rangle_{L_2}-\kappa_k^{\frac{p}{2r}-1}\Big\langle F_{{h_2},i}-f_i*\mathcal{K}_{h_2}, (f_{\eta_{k}}-f_{\eta_{k+1}})*\mathcal{K}_{h_1}\Big\rangle_{L_2}
    \\
    +&
    \kappa_k^{\frac{p}{2r}-1}\Big\langle F_{{h_2},i}-f_i*\mathcal{K}_{h_2}, (f_{\eta_{k}}-f_{\eta_{k+1}})*\mathcal{K}_{h_1}\Big\rangle_{L_2}+\widehat{\kappa}_k^{\frac{p}{2r}-1}\Big\langle F_{{h_1},i}-f_i*\mathcal{K}_{h_1}, (f_{\eta_{k}}-f_{\eta_{k+1}})*\mathcal{K}_{h_1}\Big\rangle_{L_2}\Big|
    \\
    =&\Big|\kappa_k^{\frac{p}{2r}-1}\Big\langle F_{{h_2},i}-f_i*\mathcal{K}_{h_2}, (f_{\eta_{k}}-f_{\eta_{k+1}})*\mathcal{K}_{h_2}\Big\rangle_{L_2}+\kappa_k^{\frac{p}{2r}-1}\Big\langle F_{{h_2},i}-f_i*\mathcal{K}_{h_2}, (f_{\eta_{k}}-f_{\eta_{k+1}})*\mathcal{K}_{h_1}\Big\rangle_{L_2}
    \\
    +&\Big\langle \widehat{\kappa}_k^{\frac{p}{2r}-1}(F_{{h_1},i}-f_i*\mathcal{K}_{h_1})-\kappa_k^{\frac{p}{2r}-1}(F_{{h_2},i}-f_i*\mathcal{K}_{h_2}),(f_{\eta_{k}}-f_{\eta_{k+1}})*\mathcal{K}_{h_1}\Big\rangle_{L_2}\Big|.
\end{align*}
Then, as before, we bound each of the terms above using H\"{o}lder's inequality. We start with the term
\begin{align*}
    &\Big|\kappa_k^{\frac{p}{2r}-1}\Big\langle F_{{h_2},i}-f_i*\mathcal{K}_{h_2}, (f_{\eta_{k}}-f_{\eta_{k+1}})*\mathcal{K}_{h_2}\Big\rangle_{L_2}\Big|
    \\
    \le&\kappa_k^{\frac{p}{2r}-1}\vert\vert F_{{h_2},i}-f_i*\mathcal{K}_{h_2}\vert\vert_{L_2}\vert\vert(f_{\eta_{k}}-f_{\eta_{k+1}})*\mathcal{K}_{h_2}\vert\vert_{L_2}
    \\
    \le&
    \kappa_k^{\frac{p}{2r}-1}O_p(\kappa_k^{-\frac{p}{2r}})\kappa_k
    =O_p(1)
\end{align*}
where the second inequality is followed by \eqref{bound-K}.
Similarly, 
\begin{align*}
    &\Big|\kappa_k^{\frac{p}{2r}-1}\Big\langle F_{{h_2},i}-f_i*\mathcal{K}_{h_2}, (f_{\eta_{k}}-f_{\eta_{k+1}})*\mathcal{K}_{h_1}\Big\rangle_{L_2}\Big|
    \\
    \le&\kappa_k^{\frac{p}{2r}-1}\vert\vert F_{{h_2},i}-f_i*\mathcal{K}_{h_2}\vert\vert_{L_2}\vert\vert(f_{\eta_{k}}-f_{\eta_{k+1}})*\mathcal{K}_{h_1}\vert\vert_{L_2}
    \\
    \le&
    \kappa_k^{\frac{p}{2r}-1}O_p(\kappa_k^{-\frac{p}{2r}})\kappa_k
    =O_p(1)
\end{align*}
where the second inequality is followed by the \eqref{bound-K}.
Finally, the term 
$$\Big|\Big\langle \widehat{\kappa}_k^{\frac{p}{2r}-1}(F_{{h_1},i}-f_i*\mathcal{K}_{h_1})-\kappa_k^{\frac{p}{2r}-1}(F_{{h_2},i}-f_i*\mathcal{K}_{h_2}),(f_{\eta_{k}}-f_{\eta_{k+1}})*\mathcal{K}_{h_1}\Big\rangle_{L_2}\Big|$$
was previously bounded by,
$$O_p(\frac{T^{-\frac{2r}{2r+p}}}{\kappa_k^{2-\frac{p}{2r}}}\log^{\frac{2r}{2r+p}}(T))O_p(\frac{T^{-\frac{r}{2r+p}}}{\kappa_k^{\frac{p}{2r}+\frac{1}{2}}}\log^{\frac{r}{2r+p}}(T))\kappa_k.$$
Therefore,
\begin{align*}
    I_2\le
    \frac{1}{\sqrt{S}}\sum_{i\in{\mathcal{S}_r}}\Big|\breve{Y}_i+Y_i\Big|
    =\sqrt{S}\Big(O_p(1)+O_p(\frac{T^{-\frac{2r}{2r+p}}}{\kappa_k^{2-\frac{p}{2r}}}\log^{\frac{2r}{2r+p}}(T))O_p(\frac{T^{-\frac{r}{2r+p}}}{\kappa_k^{\frac{p}{2r}+\frac{1}{2}}}\log^{\frac{r}{2r+p}}(T))\kappa_k\Big).
\end{align*}
In consequences,
\begin{align*}
    \Big|\widehat{\sigma}_{\infty}^2(k)-\breve{\sigma}_{\infty}^2(k)\Big|=&\Big|\frac{1}{R}\sum_{r=1}^{R}\Big(\frac{1}{\sqrt{S}}\sum_{i\in{\mathcal{S}_r}}\breve{Y}_i\Big)^2-\frac{1}{R}\sum_{r=1}^{R}\Big(\frac{1}{\sqrt{S}}\sum_{i\in{\mathcal{S}_r}}Y_i\Big)^2\Big|
    \\
    =&
    \Big|\frac{1}{R}\sum_{r=1}^{R}I_1I_2\Big|
    \\
    =&S\Big(O_p(\frac{T^{-\frac{4r}{2r+p}}}{\kappa_k^{4-\frac{p}{r}}}\log^{\frac{4r}{2r+p}}(T))O_p(\frac{T^{-\frac{2r}{2r+p}}}{\kappa_k^{\frac{p}{r}+1}}\log^{\frac{2r}{2r+p}}(T))O_p(\frac{T^{-\frac{r}{2r+p}}}{\kappa_k^{\frac{p}{2r}+\frac{1}{2}}}\log^{\frac{r}{2r+p}}(T))\kappa_k
    \\
    +&O_p(\frac{T^{-\frac{2r}{2r+p}}}{\kappa_k^{2-\frac{p}{2r}}}\log^{\frac{2r}{2r+p}}(T))O_p(\frac{T^{-\frac{r}{2r+p}}}{\kappa_k^{\frac{p}{2r}+\frac{1}{2}}}\log^{\frac{r}{2r+p}}(T))O_p(\frac{T^{-\frac{r}{2r+p}}}{\kappa_k^{\frac{p}{2r}+\frac{1}{2}}}\log^{\frac{r}{2r+p}}(T))
    \\
    +&O_p(\kappa_k^{-1})O_p(\frac{T^{-\frac{r}{2r+p}}}{\kappa_k^{\frac{p}{2r}+\frac{1}{2}}}\log^{\frac{r}{2r+p}}(T))
    \\
    +&O_p(\kappa_k^{-1})O_p(\frac{T^{-\frac{r}{2r+p}}}{\kappa_k^{\frac{p}{2r}+\frac{1}{2}}}\log^{\frac{r}{2r+p}}(T))O_p(\frac{T^{-\frac{2r}{2r+p}}}{\kappa_k^{2-\frac{p}{2r}}}\log^{\frac{2r}{2r+p}}(T))O_p(\frac{T^{-\frac{r}{2r+p}}}{\kappa_k^{\frac{p}{2r}+\frac{1}{2}}}\log^{\frac{r}{2r+p}}(T))\kappa_k\Big).
\end{align*}
In order to conclude (i), we notice that by \Cref{assume-snr2} and that $S=O(T^{\frac{r}{2r+p}}\kappa_k^{\frac{p}{2r}+\frac{3}{2}})$, which implies,
\begin{align*}
    \Big|\widehat{\sigma}_{\infty}^2(k)-\breve{\sigma}_{\infty}^2(k)\Big|=&o_p(1).
\end{align*}
Now, we are going to see that $\Big|\breve{\sigma}_{\infty}^2(k)-\widetilde{\sigma}_{\infty}^2(k)\Big| \stackrel{P }{\longrightarrow} 0, \quad T \rightarrow \infty .$ To this end, we will show that the estimator is asymptotically unbiased, and its variance $\rightarrow 0$ as $T\rightarrow\infty.$ First, we notice that, by H\"{o}lder's inequality and Minkowsky's inequality,
\begin{align*}
     \vert\breve{Y}_i\vert=&\vert \kappa_k^{\frac{p}{2r}-1}\Big\langle F_{{h_2},i}-f_i*\mathcal{K}_{h_2}, (f_{\eta_{k}}-f_{\eta_{k+1}})*\mathcal{K}_{h_2}\Big\rangle_{L_2}\vert
     \\
     \le&\kappa_k^{\frac{p}{2r}-1}\vert\vert F_{{h_2},i}-f_i*\mathcal{K}_{h_2}\vert \vert_{L_2}\vert\vert (f_{\eta_{k}}-f_{\eta_{k+1}})*\mathcal{K}_{h_2}\vert\vert_{L_2}
     \\
     \le&
     \kappa_k^{\frac{p}{2r}-1}\kappa_k^{-\frac{p}{2r}}\kappa_k=1.
\end{align*}
Now, we analyze the Bias. We observe that,
\begin{align*}
    \mathbb{E}(\breve{\sigma}_{\infty}^2(k))=\frac{1}{R}\sum_{r=1}^{R} \mathbb{E}\Big(\Big(\frac{1}{\sqrt{S}}\sum_{i\in{\mathcal{S}_r}}\breve{Y}_i\Big)^2\Big)
    =\frac{1}{S}\mathbb{E}\Big(\Big(\sum_{i\in{\mathcal{S}_r}}\breve{Y}_i\Big)^2\Big)
    =\sum_{l=-S+1}^{S+1}\frac{S-l}{S}\mathbb{E}(\breve{Y}_i\breve{Y}_{i+l})    
\end{align*}
and, 
\begin{align*}
   \widetilde{\sigma}_{\infty}^2(k)=\sum_{l=-\infty}^{\infty}\mathbb{E}(\breve{Y}_i\breve{Y}_{i+l}).
\end{align*}
so that, the bias has the following form,
\begin{align*}
    \widetilde{\sigma}_{\infty}^2(k)-\mathbb{E}(\breve{\sigma}_{\infty}^2(k))=2\sum_{l=S}^{\infty}\mathbb{E}(\breve{Y}_i\breve{Y}_{i+l})+2\sum_{l=1}^{S}\frac{l}{S}\mathbb{E}(\breve{Y}_i\breve{Y}_{i+l}).
\end{align*}
Now, we show that each of the above terms vanishes as $T\rightarrow \infty.$
We have that, by condition \eqref{mix-cond} and covariance inequality 
\begin{align*}
    2\sum_{l=S}^{\infty}\mathbb{E}(\breve{Y}_i\breve{Y}_{i+l})\le&8\sum_{l=S}^{\infty}\vert\vert\breve{Y}_i\vert\vert_{L_\infty}^2\alpha_l\le8\sum_{l=S}^{\infty}\alpha_l\rightarrow 0, \ \text{as} \ T\rightarrow\infty
\end{align*}
where $\alpha_l$ is the mixing coefficient.
Then, 
\begin{align*}
    2\sum_{l=1}^{S}\frac{l}{S}\mathbb{E}(\breve{Y}_i\breve{Y}_{i+l})\le 8\sum_{l=1}^{S}\frac{l}{S}\vert\vert\breve{Y}_i\vert\vert_{L_\infty}^2\alpha_l\le \frac{C}{S}\rightarrow 0,
\end{align*}
by condition \eqref{mix-cond}, choice of $S$ and \Cref{assume-snr2}.
Therefore, we conclude that the Bias vanishes as $T\rightarrow\infty.$ To analyze the Variance, we observe that, if $Y_r=\frac{1}{S}\Big(\sum_{i\in{\mathcal{S}_r}}\breve{Y}_i\Big)^2$ 
\begin{align*}
    &Var(\breve{\sigma}_{\infty}^2(k))=\mathbb{E}((\breve{\sigma}_{\infty}^2(k)-\mathbb{E}(\breve{\sigma}_{\infty}^2(k)))^2)
    \\
    =&\frac{1}{R^2}\mathbb{E}\Big(\Big(\sum_{r=1}^{R}Y_r-\mathbb{E}(Y_r)\Big)^2\Big)
    \\
    =&\frac{1}{R}\sum_{l=-R+1}^{R-1}\frac{R-l}{R}cov(Y_r,Y_{l+r})
    \\
    \le&
    \frac{8}{R}\vert\vert Y_r\vert\vert_{L_\infty}^2\sum_{l=0}^{\infty}\widetilde{\alpha}_{l}
    \le
    \frac{8CS}{R}\rightarrow 0, \ \text{as}, \ T\rightarrow\infty.
\end{align*}
where, $\widetilde{\alpha}_l$ are the mixing coefficients of $\{Y_r\}_{r\in{\mathbb{Z}}}$, which is bounded by the mixing coefficient $\alpha_l$.
From here, we conclude the result (ii).
\end{proof}
  
\newpage
\section{Large probability events}\label{sec-large-events}
In this section, we deal with all the large probability events that occurred in the proof of \Cref{theorem:FSBS}. Recall that, for any $(s,e]\subseteq (0,T] $,
	$$ \widetilde f^{  s, e}_{t }    (x ) =  \sqrt { \frac{e-t}{ (e -s  )(t-s )}} \sum_{ l =s +1}^{ t}  f_l (x )  -
   \sqrt { \frac{t-s}{ (e-s  )(e-t)}} \sum_{l =t+1}^{ e}  f_l  (x ),\ x\in\mathcal{X}.$$
\begin{proposition}
\label{prop-1}
For any $x$,
\begin{align*}
\mathbb{P}\Big(\max _{\rho \leqslant k \leq T-\widetilde{r}} \Big| \frac{1}{\sqrt{k}} \sum_{t=\widetilde{r}+1}^{\widetilde{r}+k}\Big(\mathcal{K}_h(x-X_t)-\int \mathcal{K}_h(x-z) d F_t(z)\Big) \Big|
\geq C \sqrt{\frac{\log T}{h^p}}\Big) \leqslant T^{-p-3}.
\end{align*}
\end{proposition}
\begin{proof}
  We have that the random variables $\{Z_t=\mathcal{K}_h\Big(x-X_t\Big)\}_{t=1}^{T}$ satisfies
\begin{equation*}
    \sigma\Big(\mathcal{K}_h\Big(x-X_t\Big)\Big) \subset \sigma\Big(X_t\Big)
\end{equation*}
and,
\begin{equation*}
    \Big|\mathcal{K}_h\Big(x-X_t\Big)\Big| \leqslant \frac{1}{h^p} C_{K}.
\end{equation*}
Moreover, let
\begin{equation*}
V^2=\sup_{t>0}\Big(var(\mathcal{K}_h(x-X_t)+2\sum_{j>t}|cov(Z_t,Z_j)|)\Big).
\end{equation*}
We observe that,
\begin{align*}
    var(\mathcal{K}_h(x-X_t))\le& E((\frac{1}{h^{p}}\mathcal{K}(\frac{x-X_t}{h}))^2)
    \\
    \le& \int \frac{1}{h^{2p}}\mathcal{K}(\frac{x-z}{h})dF_t(z)
\end{align*}
making $\mu=\frac{x-z}{h}$, the last inequality is equal to
\begin{align*}
    \int \frac{1}{h^{2p}}\mathcal{K}(\frac{x-z}{h})dF_t(z)&\le \frac{1}{h^p}\int \mathcal{K}^2(u)dF_t(z)
    \\
    &\le \frac{1}{h^p}C_kC_f.
\end{align*}
Then, by proposition 2.5 on \cite{fan2008nonlinear}, $|cov(Z_1,Z_1+t)|\le C\alpha(t)\frac{1}{h^{2p}}C_K^2$.
On the other hand, 
\begin{align*}
  cov(Z_1,Z_1+t)=&|E(Z_1Z_{t+1})-E(Z_1)^2|
  \\
  \le&\int\int \mathcal{K}_h(x-z_1)\mathcal{K}_h(x-z_2)g_t(z_1,z_2)dz_1dz_2+E(Z_1)^2
  \\
  \le& ||g_t||_{L_\infty}+E(Z_1)^2.
\end{align*}
 Since by assumption, equation \eqref{join-d-cond}, $||g_t||_{L_\infty}<\infty$ and,
 \begin{align*}
     E(Z_1)=E(\mathcal{K}_h(x-X_1))=\int \frac{1}{h^{p}}\mathcal{K}(\frac{x-z}{h})dF_t(z)=\int \mathcal{K}(u)f_t(x-hu)du=O(1),
 \end{align*}
 we obtain that $|cov(Z_1,Z_{1+t})|$. Therefore, $\sum_{t=1}^{\frac{1}{h}-1}|cov(Z_1,Z_{t+1})|\le C\frac{1}{h}$
 and, using the mixing condition bound, inequality \eqref{mix-cond},
 \begin{align*}
 \sum_{t=\frac{1}{h^p}}^{T-1}|cov(Z_1,Z_1+t)|\le& D\sum_{t=\frac{1}{h}}^{\infty}\frac{e^{-2Ct}}{h^{2p}}
 \\
 \le& D\frac{e^{-2C\frac{1}{h^p}}}{h^{2p}}
 \\
 \le& \widetilde{D}\frac{1}{h^{2p}}h^p= \widetilde{D}h^p
 \end{align*}
 where the last inequity is followed by the fact $e^{-x}<\frac{1}{x}$ for $x>-1.$
 In consequence,
 \begin{align*}
     V^2=&\sup_{t>0}\Big(var(\mathcal{K}_h(x-X_t)+2\sum_{j>t}|cov(Z_t,Z_j)|)\Big)
     \\
     =& \widetilde{C}\frac{1}{h^p}+\widetilde{D}\frac{1}{h^p}=\widetilde{\widetilde{C}}\frac{1}{h^p}.
 \end{align*}
Then, by Bernstein inequality for mixing dependence, see \cite{merlevede2009bernstein} for more details, letting 
 \begin{equation*}
     \lambda=C_p\Big(\sqrt{\frac{k\log(T)}{h^p}}+\sqrt{\frac{\log(T)}{h^{2p}}}+\sqrt{\frac{\log(T)\log^2(k)}{h^p}}\Big)
 \end{equation*}
 we get that,
 \begin{align*}
     \mathbb{P}\Big( \Big|  \sum_{t=\widetilde{r}+1}^{\widetilde{r}+k}\Big(\mathcal{K}_h(x-X_t)-\int \mathcal{K}_h(x-z) d F_t(z)\Big) \Big|>\lambda\Big)\le T^{-p-3}.
 \end{align*}
 in consequence,
  \begin{align*}
     \mathbb{P}\Big(\Big| \frac{1}{\sqrt{k}} \sum_{t=\widetilde{r}+1}^{\widetilde{r}+k}\Big(\mathcal{K}_h(x-X_t)-\int \mathcal{K}_h(x-z) d F_t(z)\Big) \Big|>\frac{\lambda}{\sqrt{k}}\Big)\le T^{-p-3}.
 \end{align*}
 Since $kh^p\ge log(T)$ if $k>\rho$, and $\log^2(k)=O(k)$,
 \begin{align*}
     &\frac{\lambda}{\sqrt{k}}=\frac{C_p\Big(\sqrt{\frac{k\log(T)}{h^p}}+\sqrt{\frac{\log(T)}{h^{2p}}}+\sqrt{\frac{\log(T)\log^2(k)}{h^p}}\Big)}{\sqrt{k}}
     \\
     =&C_p\Big(\sqrt{\frac{\log(T)}{h^p}}+\sqrt{\frac{\log(T)}{kh^{2p}}}+\sqrt{\frac{\log(T)\log^2(k)}{kh^p}}\Big)
     \\
     \le&
     C_p\Big(\sqrt{\frac{\log(T)}{h^p}}+\sqrt{\frac{1}{h^{p}}}+\sqrt{\frac{\log(T)}{h^p}}\Big)
     \\
     \le&
     C_1\sqrt{\frac{\log(T)}{h^p}}.
 \end{align*}
 It follows that,
   \begin{align*}
     \mathbb{P}\Big(\Big| \frac{1}{\sqrt{k}} \sum_{t=\widetilde{r}+1}^{\widetilde{r}+k}\Big(\mathcal{K}_h(x-X_t)-\int \mathcal{K}_h(x-z) d F_t(z)\Big) \Big|>C_1\sqrt{\frac{\log(T)}{h^p}}\Big)\le T^{-p-3}.
 \end{align*}
\end{proof}
\begin{proposition}
\label{Concentration-Bound}
Define the events
$$\mathcal{A}_1=\Big\{\max_{t=s+\rho+1}^{e-\rho }  \sup_{x \in \mathbb{R}^p} \bigg|  \widetilde F_{t, h }^{s,e} (x)   -   \widetilde f_{t}^{s,e} (x)\bigg|\ge 2C \sqrt{\frac{\log T}{h^p}}+\frac{2C_1\sqrt{p}}{h^p}+2C_2\sqrt{T}h^{r}\Big\}$$
and,
$$\mathcal{A}_2=\Big\{\max _{\rho \leqslant k \leq T-\widetilde{r}} \sup_{x\in\mathbb{R}^p}\Big| \frac{1}{\sqrt{k}} \sum_{t=\widetilde{r}+1}^{\widetilde{r}+k}\Big(\mathcal{K}_h(x-X_t)-f_t(x)\Big) \Big|
\geq C \sqrt{\frac{\log T}{h^p}}+\frac{C_1\sqrt{p}}{h^p}+C_2\sqrt{T}h^{r}\Big\}.$$
Then
\begin{align}
&\mathbb{P}\Big(\mathcal{A}_1\Big) \leqslant 2R^pT^{-2}\\
&
\mathbb{P}\Big(\mathcal{A}_2\Big) \leqslant R^pT^{-2}
\end{align}
where $R$ is a positive constant. 
\end{proposition}
\begin{proof}
  First, we notice that 
  \begin{align*}
    &\max _{\rho \leqslant k \leqslant T-\widetilde{r}}\sup_{x \in \mathbb{R}^p} \Big|\frac{1}{\sqrt{k}} \sum_{t=\widetilde{r}+1}^{\widetilde{r}+k} \mathcal{K}_h\Big(x-X_t\Big)-f_t(x)\Big|
    \\
    \le&\max _{\rho \leqslant k \leq T-\widetilde{r}} \sup_{x \in \mathbb{R}^p} \Big| \frac{1}{\sqrt{k}} \sum_{t=\widetilde{r}+1}^{\widetilde{r}+k}\Big(\mathcal{K}_h(x-X_t)-\int \mathcal{K}_h(x-z) d F_t(z)\Big) \Big|
    \\
    +&\max _{\rho \leqslant k \leqslant T-\widetilde{r}} \sup_{x \in \mathbb{R}^p} \Big| \frac{1}{\sqrt{k}} \sum_{t=\widehat{r}+1}^{r+k}\Big(\int \mathcal{K}_h(x-z) d F_z(z)-f_t(x)\Big)\Big|=I_1+I_2
\end{align*}
Now we will bound each of the terms $I_1$, $I_{2}$.
For $I_1$, we consider
 $$A=\{x_1,...,x_{(R\sqrt{T}/\sqrt{h^p}h)^p}\}$$ with $\cup_{x_i \in\{x_1,...,x_{(R\sqrt{T}/\sqrt{h^p}h)^p}\}}Rec(x_i,\frac{\sqrt{h^p}h}{\sqrt{T}})\supset D$, where $D$ is the support of $K$ and $Rec(x_i,\frac{\sqrt{h^p}h}{\sqrt{T}})$ are boxes centered at $x_i$ of size $\frac{\sqrt{h^p}h}{\sqrt{T}}$ and $R$ is the size of the boxe containing $D$. Then by \Cref{prop-1}, for any $x_i,$
\begin{align}
\mathbb{P}\Big(\max _{\rho \leqslant k \leq T-\widetilde{r}} \Big| \frac{1}{\sqrt{k}} \sum_{t=\widetilde{r}+1}^{\widetilde{r}+k}\Big(\mathcal{K}_h(x_i-X_t)-\int \mathcal{K}_h(x_i-z) d F_t(z)\Big) \Big|
\geq C \sqrt{\frac{\log T}{h^p}}\Big) \leqslant T^{-p-3},
\end{align}
by an union bound argument,
\begin{align}
\mathbb{P}\Big(\max _{\rho \leqslant k \leq T-\widetilde{r}}\sup_{x\in A}\Big| \frac{1}{\sqrt{k}} \sum_{t=\widetilde{r}+1}^{\widetilde{r}+k}\Big(\mathcal{K}_h(x-X_t)-\int \mathcal{K}_h(x-z) d F_t(z)\Big) \Big|
\geq C \sqrt{\frac{\log T}{h^p}}\Big) \leqslant T^{-p-3}|A|.
\end{align}
Let $I_{1,1}=\{\max _{\rho \leqslant k \leq T-\widetilde{r}}\sup_{x\in A}\Big| \frac{1}{\sqrt{k}} \sum_{t=\widetilde{r}+1}^{\widetilde{r}+k}\Big(\mathcal{K}_h(x-X_t)-\int \mathcal{K}_h(x-z) d F_t(z)\Big) \Big|\}.$
For any $x\in \mathbb{R}^p$, there exist $x_i\in A $ such that 
\begin{align*}
    &\max _{\rho \leqslant k \leq T-\widetilde{r}}\Big| \frac{1}{\sqrt{k}} \sum_{t=\widetilde{r}+1}^{\widetilde{r}+k}\Big(\mathcal{K}_h(x-X_t)-\int \mathcal{K}_h(x-z) d F_t(z)\Big) \Big|
    \\
    \le&\max _{\rho \leqslant k \leq T-\widetilde{r}}\Big| \frac{1}{\sqrt{k}} \sum_{t=\widetilde{r}+1}^{\widetilde{r}+k}\Big(\mathcal{K}_h(x_i-X_t)-\int \mathcal{K}_h(x_i-z) d F_t(z)\Big) \Big|
    \\
    +&\max _{\rho \leqslant k \leq T-\widetilde{r}}\Big| \frac{1}{\sqrt{k}} \sum_{t=\widetilde{r}+1}^{\widetilde{r}+k}\Big(\mathcal{K}_h(x-X_t)-\mathcal{K}_h(x_i-X_t) \Big) \Big|
    \\
    +&\max _{\rho \leqslant k \leq T-\widetilde{r}}\Big| \frac{1}{\sqrt{k}} \sum_{t=\widetilde{r}+1}^{\widetilde{r}+k}\Big(\int \mathcal{K}_h(x_i-z) d F_t(z)-\int \mathcal{K}_h(x-z) d F_t(z)\Big) \Big|
    \\
    \le&\max _{\rho \leqslant k \leq T-\widetilde{r}}\sup_{x\in A}\Big| \frac{1}{\sqrt{k}} \sum_{t=\widetilde{r}+1}^{\widetilde{r}+k}\Big(\mathcal{K}_h(x_i-X_t)-\int \mathcal{K}_h(x_i-z) d F_t(z)\Big) \Big|
    \\
    +&\max _{\rho \leqslant k \leq T-\widetilde{r}}\Big| \frac{1}{\sqrt{k}} \sum_{t=\widetilde{r}+1}^{\widetilde{r}+k}\Big(\mathcal{K}_h(x-X_t)-\mathcal{K}_h(x_i-X_t) \Big) \Big|
    \\
    +&\max _{\rho \leqslant k \leq T-\widetilde{r}}\Big| \frac{1}{\sqrt{k}} \sum_{t=\widetilde{r}+1}^{\widetilde{r}+k}\Big(\int \mathcal{K}_h(x_i-z) d F_t(z)-\int \mathcal{K}_h(x-z) d F_t(z)\Big) \Big|
    \\
    =& I_{1,1}+I_{1,2}+I_{1,3}
\end{align*}
The term $I_{1,2}$ is bounded as followed. 
\begin{align*}
    &\max _{\rho \leqslant k \leq T-\widetilde{r}}\Big| \frac{1}{\sqrt{k}} \sum_{t=\widetilde{r}+1}^{\widetilde{r}+k}\Big(\mathcal{K}_h(x-X_t)-\mathcal{K}_h(x_i-X_t) \Big) \Big|
    \\
    \le&
    \max _{\rho \leqslant k \leq T-\widetilde{r}} \frac{1}{\sqrt{k}} \sum_{t=\widetilde{r}+1}^{\widetilde{r}+k}\Big|\mathcal{K}_h(x-X_t)-\mathcal{K}_h(x_i-X_t)\Big| 
    \\
    \le&
    \max _{\rho \leqslant k \leq T-\widetilde{r}} \frac{1}{\sqrt{k}} \sum_{t=\widetilde{r}+1}^{\widetilde{r}+k}\frac{\vert x-x_i\vert}{h^{p+1}} 
    \\
    \le&\max _{\rho \leqslant k \leq T-\widetilde{r}} \frac{1}{\sqrt{k}} \sum_{t=\widetilde{r}+1}^{\widetilde{r}+k}\frac{\sqrt{h^{p}}h\sqrt{p}}{\sqrt{T}h^{p+1}} 
    \le \frac{\sqrt{p}}{\sqrt{h^p}}
\end{align*}
For the term $I_{1,3}$, since the random variables $\{\mathcal{K}_h(x-X_t)\}_{t=1}^{T}$ have bounded expected value for any $x\in \mathbb{R}^p$
\begin{align*}
    &\max _{\rho \leqslant k \leq T-\widetilde{r}}\Big| \frac{1}{\sqrt{k}} \sum_{t=\widetilde{r}+1}^{\widetilde{r}+k}\Big(\int \mathcal{K}_h(x_i-z) d F_t(z)-\int \mathcal{K}_h(x-z) d F_t(z)\Big) \Big|
    \\
    \le& \max _{\rho \leqslant k \leq T-\widetilde{r}}\Big| \frac{1}{\sqrt{k}} \sum_{t=\widetilde{r}+1}^{\widetilde{r}+k}\Big(\frac{\sqrt{h^p}h2C\sqrt{p}}{h^{p+1}\sqrt{T}}\Big) \Big|
    \le \frac{2C\sqrt{p}}{\sqrt{h^p}}.
\end{align*}
Thus, 
\begin{align*}
    &\max _{\rho \leqslant k \leq T-\widetilde{r}} \sup_{x \in \mathbb{R}^p} \Big| \frac{1}{\sqrt{k}} \sum_{t=\widetilde{r}+1}^{\widetilde{r}+k}\Big(\mathcal{K}_h(x-X_t)-\int \mathcal{K}_h(x-z) d F_t(z)\Big) \Big|
    \\
    \le& I_{1,1}+C_2\frac{\sqrt{p}}{\sqrt{h^{p}}}
\end{align*}
From here, 
\begin{align}
\label{event-needed}
    \mathbb{P}\Big(I_1>C_1 \sqrt{\frac{\log T}{h^p}}+\frac{C_2\sqrt{p}}{\sqrt{h^p}}\Big)
    \le&
    \mathbb{P}\Big(I_{1,1}+I_{1,2}+I_{1,3}>C_1 \sqrt{\frac{\log T}{h^p}}+\frac{C_2\sqrt{p}}{\sqrt{h^p}}\Big)
    \\
    \le&\mathbb{P}\Big(I_{1,1}\Big)
    \le T^{-p-3}|A|
    =T^{-p-2}(R\sqrt{T}/\sqrt{h^p}h)^p
    \\
    \le& T^{-p-3}(R\sqrt{T}\sqrt{T}T^{\frac{1}{p}})^{p}
    =R^pT^{-2}
\end{align}
Finally, we analyze the term $I_2$. By the adaptive assumption, the following is satisfied,
\begin{align*}
    &\max _{\rho \leqslant k \leqslant T-\widetilde{r}} \sup_{x \in \mathbb{R}^p} \Big| \frac{1}{\sqrt{k}} \sum_{t=\widehat{r}+1}^{r+k}\Big(\int \mathcal{K}_h(x-z) d F_z(z)-f_t(x)\Big)\Big|
    \\
    \le& \max _{\rho \leqslant k \leqslant T-\widetilde{r}}  \frac{1}{\sqrt{k}} \sum_{t=\widehat{r}+1}^{r+k}\sup_{x \in \mathbb{R}^p}\Big| \int \mathcal{K}_h(x-z) d F_z(z)-f_t(x)\Big| 
    \\
    \le&\max _{\rho \leqslant k \leqslant T-\widetilde{r}}  \frac{1}{\sqrt{k}}\sum_{t=\widehat{r}+1}^{r+k} C_2h^r
    \\
    \le& C_2\sqrt{T}h^r
\end{align*}
We conclude the bound for event $\mathcal{A}_2.$
We conclude the bound for event $\mathcal{A}_2.$ Next, to derive the bound for event $\mathcal{A}_1,$
by definition of $\widetilde F_{t, h}^{s,e}$ and $\widetilde f_{t}^{s,e}$, we have that
\begin{align*}
    \bigg | \widetilde F_{t, h}^{s,e } (x) -\widetilde f_{t}^{s,e} (x)  \bigg|  
&\le
\bigg |\sqrt{\frac{e-t}{(e-s)(t-s)}}\sum_{l=s+1}^{t}(F_{l,h}(x)-f_{l,h}(x))\bigg |
\\
&+
\bigg |\sqrt{\frac{t-s}{(e-s)(e-t)}}\sum_{l=t+1}^{e}(F_{l,h}(x)-f_{l,h}(x))\bigg |.
\end{align*}
Then, we observe that, 
\begin{align*}
   \sqrt{\frac{e-t}{(e-s)(t-s)}}\le\sqrt{\frac{1}{t-s}} \ \text{if} \ s\le t, \ \text{and} \  \sqrt{\frac{t-s}{(e-s)(e-t)}}\le \sqrt{\frac{1}{e-t}} \ \text{if} \ t\le e.
\end{align*}
Therefore,
\begin{align*}
X=\max_{t=s+\rho+1}^{e-\rho }\bigg | \widetilde F_{t, h}^{(s,e } (x) -\widetilde f_{t}^{s,e} (x)  \bigg|  
\le 
& \max_{t=s+\rho+1}^{e-\rho }  \bigg|   \sqrt { \frac{1}{  t-s }} \sum_{ l =s+1}^{ t}
\bigg (  F_{ l , h} (x)  -   \{ f_{ l,h }(x) \bigg )  \bigg| 
  \\
  +
  & 
 \max_{t=s+\rho+1}^{e-\rho } \bigg|   \sqrt { \frac{1}{  e- t }} \sum_{ l =t+1}^{ e }
\bigg (  F_{ l , h} (x)  -  f_{ l,h }(x) \bigg )  \bigg|
=X_1+X_2.
\end{align*} 
Finally, letting $\lambda=2C_1 \sqrt{\frac{\log T}{h^p}}+\frac{2C_2\sqrt{p}}{\sqrt{h^p}}+2C_2\sqrt{T}h^{r},$ we get that
\begin{align*}
\mathbb P(X\ge \lambda)\le& 
\mathbb P(X_1+X_2\ge \frac{\lambda}{2}+\frac{\lambda}{2})
\\
\le&
 \mathbb P(X_1\ge \frac{\lambda}{2})+\mathbb P(X_2\ge\frac{\lambda}{2})
\\
\le&   2R^pT ^{-2},
\end{align*}   
where the last inequality follows from above. This concludes the bound for $\mathcal{A}_1.$
\end{proof}
\begin{remark}
\label{remark-1}
On the events
$(\mathcal{A}_1)^c$
and,
$(\mathcal{A}_2)^c$,
by \Cref{assume: model assumption}, we have that
\begin{align*}
    \max_{t=s+\rho+1}^{e-\rho }  \vert\vert  \widetilde F_{t, h }^{s,e} (x)   -   \widetilde f_{t}^{s,e} (x)\vert\vert_{L_2}\le&
    \max_{t=s+\rho+1}^{e-\rho } \widetilde{C}_{\mathcal{X}} \sup_{x \in \mathbb{R}^p} \bigg|  \widetilde F_{t, h }^{s,e} (x)   -   \widetilde f_{t}^{s,e} (x)\bigg|
    \\
    \le&2\widetilde{C}_{\mathcal{X}}C \sqrt{\frac{\log T}{h^p}}+\frac{2\widetilde{C}_{\mathcal{X}}C_1\sqrt{p}}{h^p}+2\widetilde{C}_{\mathcal{X}}C_2\sqrt{T}h^{r}
\end{align*}
where $\widetilde{C}_{\mathcal{X}}$ is the volume of the set $\mathcal{X}.$
Moreover, using inequality \eqref{event-needed}, we have that
\begin{align*}
    &\max _{\rho \leqslant k \leq T-\widetilde{r}}  \Big|\Big| \frac{1}{\sqrt{k}} \sum_{t=\widetilde{r}+1}^{\widetilde{r}+k}\Big(\mathcal{K}_h(\cdot-X_t)-\int \mathcal{K}_h(\cdot-z) d F_t(z)\Big) \Big|\Big|_{L_2}
    \\
    \le&C_{\mathcal{X}}\max _{\rho \leqslant k \leq T-\widetilde{r}} \sup_{x \in \mathbb{R}^p} \Big| \frac{1}{\sqrt{k}} \sum_{t=\widetilde{r}+1}^{\widetilde{r}+k}\Big(\mathcal{K}_h(x-X_t)-\int \mathcal{K}_h(x-z) d F_t(z)\Big) \Big|
    \\
    =&O_p\Big(\sqrt{\frac{\log T}{h^p}}\Big).
\end{align*}
\end{remark}

\newpage
\section{$\alpha$-mixing condition} \label{sec-mic-properties}
A process $(X_t,t\in{\mathbb{Z}})$ is said to be $\alpha$-mixing if 
$$\alpha_k=\sup_{t\in{\mathbb{Z}}}\alpha(\sigma(X_s,s\le t),\sigma(X_s,s\ge t+k))\longrightarrow_{k\rightarrow \infty} 0.$$
The strong mixing, or $\alpha$-mixing coefficient between two $\sigma$-fields $\mathcal{A}$ and $\mathcal{B}$ is defined as
$$
\alpha(\mathcal{A}, \mathcal{B})=\sup _{A \in \mathcal{A}, B \in \mathcal{B}}|\mathbb{P}(A \cap B)-\mathbb{P}(A) \mathbb{P}(B)| .
$$
Suppose $X$ and $Y$ are two random variables. Then for positive numbers $p^{-1}+q^{-1}+r^{-1}=1$, it holds that
$$
|\operatorname{Cov}(X, Y)| \leq 4\|X\|_{L_p}\|Y\|_{L_q}\{\alpha(\sigma(X), \sigma(Y))\}^{1 / r} .
$$
Let $\Big\{Z_t\Big\}_{t=-\infty}^{\infty}$ be a stationary time series vectors. Denote the alpha mixing coefficients of $k$ to be
$$
\alpha(k)=\alpha\Big(\sigma\Big\{\ldots, Z_{t-1}, Z_t\Big\}, \sigma\Big\{Z_{t+k}, Z_{t+k+1}, \ldots\Big\}\Big) .
$$
Note that the definition is independent of $t$.
\subsection{ Maximal Inequality}
The unstationary version of the following lemma is in Lemma B.5. of \cite{kirch2006resampling}.
\begin{lemma}
\label{lem1-mix}
Suppose $\Big\{y_i\Big\}_{i=1}^{\infty}$ is a stationary alpha-mixing time series with mixing coefficient $\alpha(k)$ and that $\mathbb{E}\Big(y_i\Big)=0$. Suppose that there exists $\delta, \Delta>0$ such that
$$
\mathbb{E}\Big(\Big|y_i\Big|^{2+\delta+\Delta}\Big) \leq D_1
$$
and
$$
\sum_{k=0}^{\infty}(k+1)^{\delta / 2} \alpha(k)^{\Delta /(2+\delta+\Delta)} \leq D_2 .
$$
Then
$$
\mathbb{E}\Big(\max _{k=1, \ldots, n}\Big|\sum_{i=1}^k y_i\Big|^{2+\delta}\Big) \leq D n^{(2+\delta) / 2},
$$
where $D$ only depends on $\delta$ and the joint distribution of $\Big\{y_i\Big\}_{i=1}^{\infty}$.
\end{lemma}
\begin{proof}
This is Lemma B.8. of \cite{kirch2006resampling}.
\end{proof} 

\begin{lemma}
 Suppose that there exists $\delta, \Delta>0$ such that
$$
\mathbb{E}\Big(\Big|y_i\Big|^{2+\delta+\Delta}\Big) \leq D_1
$$
and
$$
\sum_{k=0}^{\infty}(k+1)^{\delta / 2} \alpha(k)^{\Delta /(2+\delta+\Delta)} \leq D_2 .
$$
Then it holds that for any $d>0,0<\nu<1$ and $x>0$,
$$
\mathbb{P}\Big(\max _{k \in[\nu d, d]} \frac{\Big|\sum_{i=1}^k y_i\Big|}{\sqrt{k}} \geq x\Big) \leq C x^{-2-\delta},
$$
where $C$ is some constant.
\end{lemma}
\begin{proof}
 Let
$$
S_d^*=\max _{k=1, \ldots, d}\Big|\sum_{i=1}^k y_i\Big| .
$$
Then \Cref{lem1-mix} implies that
$$
\Big\|S_d^*\Big\|_{L_{2+\delta}} \leq C_1 d^{1 / 2}
$$
Therefore it holds that
$$
\mathbb{P}\Big(\Big|\frac{S_d^*}{\sqrt{d}}\Big| \geq x\Big)=\mathbb{P}\Big(\Big|\frac{S_d^*}{\sqrt{d}}\Big|^{2+\delta} \geq x^{2+\delta}\Big) \leq C_1 x^{-2-\delta} .
$$
Observe that
$$
\frac{\Big|S_d^*\Big|}{\sqrt{d}}=\max _{k=1, \ldots, d} \frac{\Big|\sum_{i=1}^k y_i\Big|}{\sqrt{d}} \geq \max _{k \in[\nu d, d]} \frac{\Big|\sum_{i=1}^k y_i\Big|}{\sqrt{d}} \geq \max _{k \in[\nu d, d]} \frac{\Big|\sum_{i=1}^k y_i\Big|}{\sqrt{k / \nu}}
$$
Therefore
$$
\mathbb{P}\Big(\max _{k \in[\nu d, d]} \frac{\Big|\sum_{i=1}^k y_i\Big|}{\sqrt{k}} \geq x / \sqrt{\nu}\Big) \leq \mathbb{P}\Big(\Big|\frac{S_d^*}{\sqrt{d}}\Big| \geq x\Big) \leq C_1 x^{-2-\delta},
$$
which gives
$$
\mathbb{P}\Big(\max _{k \in[\nu d, d]} \frac{\Big|\sum_{i=1}^k y_i\Big|}{\sqrt{k}} \geq x\Big) \leq C_2 x^{-2-\delta} .
$$
\end{proof}
\begin{lemma}
\label{lemma3}
 Let $\nu>0$ be given. Under the same assumptions as in \Cref{lem1-mix}, for any $0<a<1$ it holds that
$$
\mathbb{P}\Big(\Big|\sum_{i=1}^r y_i\Big| \leq \frac{C}{a} \sqrt{r}\{\log (r \nu)+1\} \text { for all } r \geq 1 / \nu\Big) \geq 1-a^2,
$$
where $C$ is some absolute constant.
\end{lemma}
\begin{proof}
 Let $s \in \mathbb{Z}^{+}$and $\mathcal{T}_s=\Big[2^s / \nu, 2^{s+1} / \nu\Big]$. By Lemma 3, for all $x \geq 1$,
$$
\mathbb{P}\Big(\sup _{r \in \mathcal{T}_s} \frac{\Big|\sum_{i=1}^r y_i\Big|}{\sqrt{r}} \geq x\Big) \leq C_1 x^{-2-\delta} \leq C_1 x^{-2} .
$$
Therefore by a union bound, for any $0<a<1$,
$$
\mathbb{P}\Big(\exists s \in \mathbb{Z}^{+}: \sup _{r \in \mathcal{T}_s} \frac{\Big|\sum_{i=1}^r y_i\Big|}{\sqrt{r}} \geq \frac{\sqrt{C_1}}{a}(s+1)\Big) \leq \sum_{s=0}^{\infty} \frac{a^2}{(s+1)^2}=a^2 \pi^2 / 6 .
$$
For any $r \in\Big[2^s / \nu, 2^{s+1} / \nu\Big], s \leq \log (r \nu) / \log (2)$, and therefore
$$
\mathbb{P}\Big(\exists s \in \mathbb{Z}^{+}: \sup _{r \in \mathcal{T}_s} \frac{\Big|\sum_{i=1}^r y_i\Big|}{\sqrt{r}} \geq \frac{\sqrt{C_1}}{a}\Big\{\frac{\log (r \nu)}{\log (2)}+1\Big\}\Big) \leq a^2 \pi^2 / 6 .
$$
Equation (2) directly gives
$$
\mathbb{P}\Big(\sup _{r \in \mathcal{T}_s} \frac{\Big|\sum_{i=1}^r y_i\Big|}{\sqrt{r}} \geq \frac{C}{a}\{\log (r \nu)+1\}\Big) \leq a^2 .
$$
\end{proof}

\subsection{Central Limit theorem} 
Below is the central limit theorem for $\alpha$-mixing random variable. We refer to \cite{doukhan1994mixing} for more details.
\begin{lemma}
\label{CTL}
   Let $\Big\{Z_t\Big\}$ be a centred $\alpha$-mixing stationary time series. Suppose for the mixing coefficients and moments, for some $\delta>0$ it holds
$$
\sum_{k=1}^{\infty} \alpha_k^{\delta /(2+\delta)}<\infty, \quad \mathbb{E}\Big[\Big|Z_1\Big|^{2+\delta}<\infty\Big] .
$$
Denote $S_n=\sum_{t=1}^n Z_t$ and $\sigma_n^2=\mathbb{E}\Big[\Big|S_n\Big|^2\Big]$. Then
$$
\frac{S_{\lfloor n t\rfloor}}{\sigma_n} \rightarrow W(t),
$$
where convergence is in Skorohod topology and $W(t)$ is the standard Brownian motion on $[0,1]$. 
\end{lemma}
\newpage

\section{ Additional Technical Results}\label{sec-additional-res}
\begin{lemma} \label{lemma:properties of seeded}
Let $ \mathcal J$ be defined as in 
\Cref{definition:seeded} and suppose  \Cref{assume: model assumption} {\bf e} holds. 
Denote
$$ \zeta _k=  \frac{9}{10} \min \{ \eta_{k+1}-\eta_k, \eta_k -\eta_{k-1} \}\  k\in\{1,...,K\}. $$
Then for each change point $\eta_k $ there exists a seeded interval $ \mathcal I_k  =(s_k,e_k ] $ such that 
\\
{\bf a.} 
$\mathcal I_k$ contains exactly one change point $\eta_k $; 
\\
{\bf b.}  
$\min\{ \eta_k-s_k , e_k -\eta_k  \}\ge \frac{1}{16} \zeta_k  $; and
\\
{\bf c.} 
$\max\{ \eta_k-s_k , e_k -\eta_k  \}\le \zeta_k  $;
\end{lemma} 
\begin{proof}
These are the desired  properties of seeded intervals by construction.  The proof is the same as theorem 3 of  \citeauthor{kovacs2020seeded} (\citeyear{kovacs2020seeded}) and is provided here for completeness.
\\
\\
Since $\zeta_k = \Theta(T) $, by construction of seeded intervals, one can find a seeded interval 
$(s_k, e_k] = (c_k-r_k , c_k +r_k] $ 
such that  
$ (c_k-r_k , c_k +r_k] \subseteq ( \eta_k-\zeta _k, \eta_k+\zeta_k]$, $r_ k\ge \frac{\zeta_k}{4} $ 
and 
$ |c_k-\eta_k|\le \frac{5r_ k }{8} $. 
So $ (c_k-r_k , c_k +r_k] $ contains only one change point $ \eta_k$.
In addition,
$$ e_k - \eta_k  =   c_k+r_k-\eta_k   \ge r_k -   |c_k -\eta_k|    \ge \frac{3r_k }{8} \ge \frac{3 \zeta _k}{32},$$
and similarly 
$\eta_k -s_k \ge  \frac{3 \zeta _k}{32}  $, so {\bf b} holds. 
Finally, since $ (c_k-r_k , c_k +r_k] \subseteq ( \eta_k-\zeta _k, \eta_k+\zeta_k]$, 
it holds that 
$ c_k+r_k \le \eta_k +\zeta_k $ 
and so 
$$ e_k -\eta_k = c_k+r_k -\eta_k \le \zeta_k .$$
\end{proof}
 
\begin{lemma}\label{lemma:KDE}
Let $\{ X_{i}  \}_{i = 1}^{ T}$ be random grid points sampled from a common density function $f_t: \mathbb{R}^p \to \mathbb R$, satisfying \Cref{assume: model assumption}-{\bf{a}} and -{\bf{b}}. Under Assumption~\eqref{kernel-as}, the density estimator of the sampling distribution $\mu$,
    \[
        \widehat{f}_t(x)=\frac{1}{T}\sum_{t=1}^{T}\mathcal{K}_{h}(x-X_{i}), \quad x \in \mathbb{R}^p,
    \]
satisfies,
\begin{align}
\label{Eq-7}
\vert\vert \widehat{f}_T-f_t\vert\vert_{L_\infty}=O_p\Big( \Big(\frac{\log(T)}{T}\Big)^{\frac{2r}{2r+p}}\Big).
\end{align}
\end{lemma}
 The verification of these bounds can be found in many places in the literature. See for example \cite{yu1993density} and \cite{Tsybakov:1315296}.

 \begin{remark}
 \label{remark2}
    Even more, by \Cref{assume: model assumption},
\begin{align}
\label{L2-bound}
\vert\vert \widehat{f}_T-f_t\vert\vert_{L_2}\le C_{\mathcal{X}}\vert\vert \widehat{f}_T-f_t\vert\vert_{L_\infty}=O\Big( \Big(\frac{\log(T)}{T}\Big)^{\frac{2r}{2r+p}}\Big)
\end{align}
with high probability. Therefore, given that
\begin{equation}
\label{kappa-b}
    \kappa=\frac{\vert\vert \sqrt{\frac{\eta_{k+1}-\eta_k}{(\eta_{k+1}-\eta_{k-1})(\eta_k-\eta_{k-1})}}\sum_{i=\eta_{k-1}+1}^{\eta_k}f_{i}-\sqrt{\frac{(\eta_k-\eta_{k-1})}{(\eta_{k+1}-\eta_{k-1})(\eta_{k+1}-\eta}_k)}\sum_{i=\eta_k+1}^{\eta_{k+1}}f_{i}\vert\vert_{L_2}}{\sqrt{\frac{(\eta_k-\eta_{k-1})(\eta_{k+1}-\eta_k)}{\eta_{k+1}-\eta_{k-1}}}}
\end{equation}
and \eqref{kappa-hat}, by triangle inequality, \eqref{L2-bound} and the fact that $\Delta=\Theta(T)$,
\begin{align*}
    \vert \kappa-\widehat{\kappa}\vert=O_p\Big( \Big(\frac{\log(T)}{T}\Big)^{\frac{2r}{2r+p}}\Big).
\end{align*}
From here, and \Cref{kernel-as}, if $h_1=O(\kappa^{\frac{1}{r}})$ and $h_2=O(\widehat{\kappa}^{\frac{1}{r}})$, we conclude that $$\vert\vert F_{t,h_1}-F_{t,h_2}\vert\vert_{L_2}^2=O\Big(\frac{\vert\kappa-\widehat{\kappa}\vert}{\kappa^{\frac{p}{r}+1}}\Big).$$
In fact,
\begin{align*}
    &\vert\vert F_{t,h_1}-F_{t,h_2}\vert\vert_{L_2}^2
    \\
    =&\int_{\mathbb{R}^{p}}(\frac{1}{h_1^p}\mathcal{K}(\frac{x-X_t}{h_1})-\frac{1}{h_2^p}\mathcal{K}(\frac{x-X_t}{h_2}))^2dx
    \\
    =&\int_{\mathbb{R}^{p}}\Big(\frac{1}{h_1^p}\mathcal{K}(\frac{x-X_t}{h_1})\Big)^2-2\frac{1}{h_1^p}\mathcal{K}(\frac{x-X_t}{h_1})\frac{1}{h_2^p}\mathcal{K}(\frac{x-X_t}{h_2})
    +\Big(\frac{1}{h_2^p}\mathcal{K}(\frac{x-X_t}{h_2})\Big)^2dx.
\end{align*}
Now, we analyze the two following terms,
\begin{align*}
    I_1=\int_{\mathbb{R}^{p}}\Big(\frac{1}{h_1^p}\mathcal{K}(\frac{x-X_t}{h_1})\Big)^2-\frac{1}{h_1^p}\mathcal{K}(\frac{x-X_t}{h_1})\frac{1}{h_2^p}\mathcal{K}(\frac{x-X_t}{h_2})dx
\end{align*}
and
\begin{align*}
    I_2=\int_{\mathbb{R}^{p}}\Big(\frac{1}{h_2^p}\mathcal{K}(\frac{x-X_t}{h_2})\Big)^2-\frac{1}{h_1^p}\mathcal{K}(\frac{x-X_t}{h_1})\frac{1}{h_2^p}\mathcal{K}(\frac{x-X_t}{h_2})dx.
\end{align*}
For $I_1$, letting $u=\frac{x-X_t}{h_1}$, we have that
\begin{align*}
    \int_{\mathbb{R}^{p}}\Big(\frac{1}{h_1^p}\mathcal{K}(\frac{x-X_t}{h_1})\Big)^2dx=&\int_{\mathbb{R}^{p}}\frac{1}{h_1^{p}}\Big(\mathcal{K}(u)\Big)^2du
\end{align*}
and, letting $v=\frac{x-X_t}{h_2}$, we have that
\begin{align*}
    \int_{\mathbb{R}^{p}}\frac{1}{h_1^p}\mathcal{K}(\frac{x-X_t}{h_1})\frac{1}{h_2^p}\mathcal{K}(\frac{x-X_t}{h_2})dx=&\int_{\mathbb{R}^{p}}\frac{1}{h_1^p}\mathcal{K}(v\frac{h_2}{h_1})\mathcal{K}(v)dv.
\end{align*}
Therefore, by \Cref{kernel-as} and the Mean Value Theorem,
\begin{align*}
    I_1=\frac{1}{h_1^{p}}\int_{\mathbb{R}^{p}}\mathcal{K}(v)\Big(\mathcal{K}(v)-\mathcal{K}(v\frac{h_2}{h_1})\Big)dv\le&C\frac{1}{h_1^{p}}\Big|1-\frac{h_2}{h_1}\Big|\int_{\mathbb{R}^{p}}\mathcal{K}(v)\vert\vert v\vert\vert dv\\
    \le&C_1\frac{\vert h_1-h_2\vert}{h_1^{p+1}}
    \\
    =&O\Big(\frac{\vert \kappa-\widehat{\kappa}\vert}{\kappa^{\frac{p+1}{r}}}\kappa^{\frac{1}{r}-1}\Big)=O\Big(\frac{\vert \kappa-\widehat{\kappa}\vert}{\kappa^{\frac{p}{r}+1}}\Big).
\end{align*}
Similarly, we have,
\begin{align*}
    I_2=&\int_{\mathbb{R}^{p}}\Big(\frac{1}{h_2^p}\mathcal{K}(\frac{x-X_t}{h_2})\Big)^2-\frac{1}{h_1^p}\mathcal{K}(\frac{x-X_t}{h_1})\frac{1}{h_2^p}\mathcal{K}(\frac{x-X_t}{h_2})dx\\
    =&\frac{1}{h_2^{p}}\int_{\mathbb{R}^{p}}\mathcal{K}(v)\Big(\mathcal{K}(v)-\mathcal{K}(v\frac{h_1}{h_2})\Big)dv\le C\frac{1}{h_2^{p}}\Big|1-\frac{h_1}{h_2}\Big|\int_{\mathbb{R}^{p}}\mathcal{K}(v)\vert\vert v\vert\vert dv=O\Big(\frac{\vert \kappa-\widehat{\kappa}\vert}{\kappa^{\frac{p}{r}+1}}\Big).
\end{align*}

\end{remark}

\subsection{Multivariate change point detection lemmas}
We present some technical results corresponding to the generalization of the univariate CUSUM to the Multivariate case. For more details, we refer the interested readers to \cite{padilla2021optimal} and \cite{wang2020univariate}.

Let  $\{X_t\}_{t=1}^{T}\subset \mathbb{R}^p$ a process with unknown densities $\{f_t\}_{t=1}^{T}$. 

\begin{assumption}
\label{asum-1}
    We assume there exist  $\{\eta_k\}_{k=1}^{K}\subset \{2,...,T\}$ with $1=\eta_0<\eta_1<...<\eta_k\le T<\eta_{K+1}=T+1,$ such that 
\begin{equation}
 \label{mod1}
    f_t\neq f_{t+1} \ \text{if and only if} \ \ t\in{\{\eta_1,...,\eta_K\}},
\end{equation}
Assume
\begin{align*}
& 
\min_{k = 1, \ldots, K+1} ( \eta_k-\eta_{k-1} ) \ge  \Delta > 0,
\\
& 
0< \vert\vert f  _{\eta_{k+1}} -  f _{\eta_{k}} \vert\vert_{L_\infty}  = \kappa_{k}   \text{ for all }  k = 1, \ldots, K.
\end{align*}
\end{assumption}
In the rest of this section, we use the notation 
\[
\widetilde{f}^{(s, e]}_{t}(x) = \sqrt{\frac{e - t}{(e-s)(t-s)}} \sum_{j = s + 1}^t f_{j}(x) - \sqrt{\frac{t - s}{(e-s)(e - t)}} \sum_{j = t + 1}^e f_{j}(x),
\]
for all $0 \leq s < t < e \leq T$ and $x \in \mathbb{R}^p$.
\begin{lemma}\label{lemma:cusum boundary bound}
If $[s, e]$ contain two and only two change points $\eta_r$ and $\eta_{r+1}$, then
\[
\sup_{s\leq t \leq e} \vert\vert\widetilde{f}^{s, e}_t\vert\vert_{L_2} \leq \sqrt{e - \eta_{r+1}} \vert\vert f_{r+1}-f_{r}\vert\vert_{L_2} + \sqrt{\eta_r - s} \vert\vert f_{r}-f_{r-1}\vert\vert_{L_2}.
\]
\end{lemma}
\begin{proof}
This is Lemma 15 in \citeauthor{wang2020univariate} (\citeyear{wang2020univariate}). Consider the sequence $\Big\{g_t\Big\}_{t=s+1}^e$ be such that
$$
g_t=\Big\{\begin{array}{lll}
f_{\eta_k}, & \text { if } \  s+1 \leq t<\eta_k, \\
f_t, & \text { if } \quad \eta_k \leq t \leq e .
\end{array}\Big.
$$
For any $t \geq \eta_k$,
$$
\begin{aligned}
& \widetilde{f}_t^{s, e}-\widetilde{g}_t^{s, e} \\
= & \sqrt{\frac{e-t}{(e-s)(t-s)}}\Big(\sum_{i=s+1}^t f_i-\sum_{i=s+1}^{\eta_k} f_{\eta_k}-\sum_{i=\eta_k+1}^t f_i\Big) \\
- & \sqrt{\frac{t-s}{(e-s)(e-t)}}\Big(\sum_{i=t+1}^e f_i-\sum_{i=t+1}^e f_i\Big) \\
= & \sqrt{\frac{e-t}{(e-s)(t-s)}}\Big(\eta_k-s\Big)\Big(f_{\eta_k}-f_{\eta_{k-1}}\Big) .
\end{aligned}
$$
So for $t \geq \eta_k,\vert\vert\widetilde{f}_t^{s, e}-\widetilde{g}_t^{s, e}\vert\vert_{L_2} \leq \sqrt{\eta_k-s} \kappa_k$. Since $\sup _{s \leq t \leq e}\vert\vert\widetilde{f}_t^{s, e}\vert\vert_{L_2}=\max \Big\{\vert\vert\widetilde{f}_{\eta_k}^{s, e}\vert\vert_{L_2},\vert\vert\widetilde{f}_{\eta_{k+1}}^{s, e}\vert\vert_{L_2}\Big\}$, and that
$$
\begin{aligned}
\max \Big\{\vert\vert\widetilde{f}_{\eta_k}^{s, e}\vert\vert_{L_2},\vert\vert\widetilde{f}_{\eta_{k+1}}^{s, e}\vert\vert_{L_2}\Big\} & \leq \sup _{s \leq t \leq e}\vert\vert\widetilde{g}_t^{s, e}\vert\vert_{L_2}+\sqrt{\eta_k-s} \kappa_k \\
& \leq \sqrt{e-\eta_{k+1}} \kappa_{k+1}+\sqrt{\eta_r-s} \kappa_k
\end{aligned}
$$
where the last inequality follows form the fact that $g_t$ has only one change point in $[s, e]$.
\end{proof}  
\begin{lemma}
\label{lem-2-uc}
Suppose $e-s \leq C_R \Delta$, where $C_R>0$ is an absolute constant, and that
$$
\eta_{k-1} \leq s \leq \eta_k \leq \ldots \leq \eta_{k+q} \leq e \leq \eta_{k+q+1}, \quad q \geq 0
$$
Denote
$$
\kappa_{\max }^{s, e}=\max \Big\{\sup _{x \in \mathbb{R}^p}\Big|f_{\eta_p}(x)-f_{\eta_{p-1}}(x)\Big|: k \leq p \leq k+q\Big\} .
$$
Then for any $k-1 \leq p \leq k+q$, it holds that
$$
\sup _{x \in \mathbb{R}^p}\Big|\frac{1}{e-s} \sum_{i=s+1}^e f_{i}(x)-f_{\eta_p}(x)\Big| \leq C_R \kappa_{\max }^{s, e} .
$$
\end{lemma}
\begin{proof}
This is Lemma 18 in \citeauthor{wang2020univariate} (\citeyear{wang2020univariate}). Since $e-s \leq C_R \Delta$, the interval $[s, e]$ contains at most $C_R+1$ change points. Observe that
$$
\begin{aligned}
& \Big|\Big|\frac{1}{e-s} \sum_{i=s}^e f_i-f_{\eta_p}\Big|\Big|_{L_\infty} \\
= & \frac{1}{e-s}\Big|\Big|\sum_{i=s}^{\eta_k}\Big(f_{\eta_{k-1}}-f_{\eta_p}\Big)+\sum_{i=\eta_k+1}^{\eta_{k+1}}\Big(f_{\eta_k}-f_{\eta_p}\Big)+\ldots+\sum_{i=\eta_{k+q}+1}^e\Big(f_{\eta_{k+q}}-f_{\eta_p}\Big)\Big|\Big|_{L_\infty}
\end{aligned}
$$
$$
\begin{aligned}
& \leq \frac{1}{e-s} \sum_{i=s}^{\eta_k}|p-k| \kappa_{\max }^{s, e}+\sum_{i=\eta_k+1}^{\eta_{k+1}}|p-k-1| \kappa_{\max }^{s, e}+\ldots+\sum_{i=\eta_{k+q}+1}^e|p-k-q-1| \kappa_{\max }^{s, e} \\
& \leq \frac{1}{e-s} \sum_{i=s}^e\Big(C_R+1\Big) \kappa_{\max }^{s, e},
\end{aligned}
$$
where $\Big|p_1-p_2\Big| \leq C_R+1$ for any $\eta_{p_1}, \eta_{p_2} \in[s, e]$ is used in the last inequality.
\end{proof}  
\begin{lemma}
\label{lemma20}
    Let $(s, e) \subset(0, n)$ contains two or more change points such that
$$
\eta_{k-1} \leq s \leq \eta_k \leq \ldots \leq \eta_{k+q} \leq e \leq \eta_{k+q+1}, \quad q \geq 1
$$
If $\eta_k-s \leq c_1 \Delta$, for $c_1>0$, then
$$
\Big|\Big|\widetilde{f}_{\eta_k}^{s, e}\Big|\Big|_{L_{\infty}} \leq \sqrt{c_1}\Big|\Big|\widetilde{f}_{\eta_{k+1}}^{s, e}\Big|\Big|_{L_{\infty}}+2 \kappa_k \sqrt{\eta_k-s}
$$
\end{lemma}
\begin{proof}
This is Lemma 20 in \cite{wang2020univariate}.
 Consider the sequence $\Big\{g_t\Big\}_{t=s+1}^e$ be such that
$$
g_t= \begin{cases}f_{\eta_{r+1}}, & s+1 \leq t \leq \eta_k, \\ f_t, & \eta_k+1 \leq t \leq e\end{cases}
$$
For any $t \geq \eta_r$, it holds that
$$
\vert\vert\widetilde{f}_{\eta_k}^{s, e}-\widetilde{g}_{\eta_k}^{s, e}\vert\vert_{L_\infty}=\Big|\Big|\sqrt{\frac{(e-s)-t}{(e-s)(t-s)}}\Big(\eta_k-s\Big)\Big(f_{\eta_{k+1}}-f_{\eta_k}\Big) \Big|\Big|_{L_\infty}\leq \sqrt{\eta_k-s} \kappa_k .
$$
Thus,
$$
\begin{aligned}
\vert\vert\widetilde{f}_{\eta_k}^{s, e}\vert\vert_{L_\infty} & \leq\vert\vert\widetilde{g}_{\eta_k}^{s, e}\vert\vert_{L_\infty}+\sqrt{\eta_k-s} \kappa_k \leq \sqrt{\frac{\Big(\eta_k-s\Big)\Big(e-\eta_{k+1}\Big)}{\Big(\eta_{k+1}-s\Big)\Big(e-\eta_k\Big)}}\vert\vert\widetilde{g}_{\eta_{k+1}}^{s, e}\vert\vert_{L_\infty}+\sqrt{\eta_k-s} \kappa_k \\
& \leq \sqrt{\frac{c_1 \Delta}{\Delta}}\vert\vert\widetilde{g}_{\eta_{k+1}}^{s, e}\vert\vert_{L_\infty}+\sqrt{\eta_k-s} \kappa_k \leq \sqrt{c_1}\vert\vert\widetilde{f}_{\eta_{k+1}}^{s, e}\vert\vert_{L_\infty}+2 \sqrt{\eta_k-s} \kappa_k,
\end{aligned}
$$
where the first inequality follows from the observation that the first change point of $g_t$ in $(s, e)$ is at $\eta_{k+1}$.
\end{proof}

\begin{lemma}
\label{max-pop}
Under \Cref{asum-1}, for any interval $(s, e) \subset(0, T)$ satisfying
$$
\eta_{k-1} \leq s \leq \eta_k \leq \ldots \leq \eta_{k+q} \leq e \leq \eta_{k+q+1}, \quad q \geq 0 .
$$
Let
$$
b \in \underset{t=s+1, \ldots, e }{\arg \max } \sup _{x \in \mathbb{R}^p}\Big|\widetilde{f}_{t}^{(s, e]}(x)\Big| .
$$
Then $b \in\Big\{\eta_1, \ldots, \eta_K\Big\}$.
For any fixed $z \in \mathbb{R}^p$, if $\widetilde{f}_{t}^{(s, e]}(z)>0$ for some $t \in(s, e)$, then $\widetilde{f}_{t}^{(s, e]}(z)$ is either strictly monotonic or decreases and then increases within each of the interval $\Big(s, \eta_k\Big),\Big(\eta_k, \eta_{k+1}\Big), \ldots,\Big(\eta_{k+q}, e\Big)$.
\end{lemma}
\begin{proof}
 We prove this by contradiction. Assume that $b \notin\{\eta_1, \ldots, \eta_K\}$. Let $z_1 \in \underset{x \in \mathbb{R}^p}{\arg \max } \Big|\bar{f}_{b}^{s, e}(x)\Big|$. Due to the definition of $b$, we have
$$
b \in\underset{t=s+1, \ldots, e}{\arg \max }\Big|\widetilde{f}_{t}^{(s, e]}\Big(z_1\Big)\Big| .
$$
It is easy to see that the collection of change points $\{f_{t}(z_1)\}_{t=s+1}^{e}$ is a subset of the change points of $\{f\}_{t=s+1}^{e}.$
Then, from Lemma $2.2$ in \cite{venkatraman1992consistency} that
$$
\widetilde{f}_{b}^{(s, e]}\Big(z_1\Big)<\max _{j \in\{k, \ldots, k+q\}} \widetilde{f}_{\eta_j}^{(s, e]}\Big(z_1\Big) \leq \max _{t=s+1, \ldots, e} \sup _{x \in \mathbb{R}^p}\Big|\widetilde{f}_{t}^{(s, e]}(x)\Big|
$$
which is a contradiction.
\end{proof}
Recall that in Algorithm 1, when searching for change points in the interval $(s, e)$, we actually restrict to values $t \in\Big(s+\rho, e-\rho\Big)$. We now show that for intervals satisfying condition $S E$ from Lemma 1, taking the maximum of the CUSUM statistic over $\Big(s+\rho, e-\rho\Big)$ is equivalent to searching on $(s, e)$, when there are change points in $\Big(s+\rho, e-\rho\Big)$.

\begin{lemma}
\label{in-up-bound}
Let $z_0 \in \mathbb{R}^p,(s, e) \subset(0, T)$. Suppose that there exists a true change point $\eta_k \in(s, e)$ such that
\begin{equation}
\label{lem10-1}
    \min \Big\{\eta_k-s, e-\eta_k\Big\} \geq c_1 \Delta,
\end{equation}
and
\begin{equation}
\label{lem10-2}
\Big|\widetilde{f}_{\eta_k}^{(s, e]}\Big(z_0\Big)\Big| \geq\Big(c_1 / 2\Big) \frac{\kappa \Delta}{\sqrt{e-s}},
\end{equation}
where $c_1>0$ is a sufficiently small constant. In addition, assume that
\begin{equation}
\label{lem10-3}
\max _{t=s+1, \ldots, e}\Big|\widetilde{f}_{t}^{(s, e]}\Big(z_0\Big)\Big|-\Big|\widetilde{f}_{\eta_k}^{(s, e]}\Big(z_0\Big)\Big| \leq c_2 \Delta^4(e-s)^{-7 / 2} \kappa,
\end{equation}
where $c_2>0$ is a sufficiently small constant.
Then for any $d \in(s, e)$ satisfying
\begin{equation}
\label{lem10-4}
\Big|d-\eta_k\Big| \leq c_1 \Delta / 32,
\end{equation}
it holds that
\begin{equation}
\Big|\widetilde{f}_{\eta_k}^{(s, e]}\Big(z_0\Big)\Big|-\Big|\widetilde{f}_{d}^{(s, e]}\Big(z_0\Big)\Big|>c\Big|d-\eta_k\Big| \Delta\Big|\widetilde{f}_{\eta_k}^{(s, e]}\Big(z_0\Big)\Big|(e-s)^{-2},
\end{equation}
where $c>0$ is a sufficiently small constant, depending on all the other absolute constants.
\end{lemma}
\begin{proof}
Without loss of generality, we assume that $d \geq \eta_k$ and $\widetilde{f}_{\eta_k}\Big(z_0\Big) \geq 0$. Following the arguments in Lemma $2.6$ in \cite{venkatraman1992consistency}, it suffices to consider two cases: (i) $\eta_{k+1}>e$ and (ii) $\eta_{k+1} \leq e$
{\bf{Case (i)}}. Note that
$$
\widetilde{f}_{\eta_k}^{(s, e]}\Big(z_0\Big)=\sqrt{\frac{\Big(e-\eta_k\Big)\Big(\eta_k-s\Big)}{e-s}}\Big\{f_{\eta_k}\Big(z_0\Big)-f_{\eta_{k+1}}\Big(z_0\Big)\Big\}
$$
and
$$
\widetilde{f}_{d}^{(s, e]}\Big(z_0\Big)=\Big(\eta_k-s\Big) \sqrt{\frac{e-d}{(e-s)(d-s)}}\Big\{f_{\eta_k}\Big(z_0\Big)-f_{\eta_{k+1}}\Big(z_0\Big)\Big\} .
$$
Therefore, it follows from \eqref{lem10-1} that
\begin{equation}
\label{lem10-10}
\widetilde{f}_{\eta_k}^{(s, e]}\Big(z_0\Big)-\widetilde{f}_{d}^{(s, e]}\Big(z_0\Big)=\Big(1-\sqrt{\frac{(e-d)\Big(\eta_k-s\Big)}{(d-s)\Big(e-\eta_k\Big)}}\Big) \widetilde{f}_{\eta_k}^{(s, e]}\Big(z_0\Big) \geq c \Delta\Big|d-\eta_k\Big|(e-s)^{-2} \widetilde{f}_{\eta_k}^{(s, e]}\Big(z_0\Big) .
\end{equation}
The inequality follows from the following arguments. Let $u=\eta_k-s, v=e-\eta_k$ and $w=d-\eta_k$. Then
\begin{align*}
& 1-\sqrt{\frac{(e-d)\Big(\eta_k-s\Big)}{(d-s)\Big(e-\eta_k\Big)}}-c \Delta\Big|d-\eta_k\Big|(e-s)^2 \\
=& 1-\sqrt{\frac{(v-w) u}{(u+w) v}-c \frac{\Delta w}{(u+v)^2}} \\
=& \frac{w(u+v)}{\sqrt{(u+w) v}(\sqrt{(v-w) u}+\sqrt{(u+w) v})}-c \frac{\Delta w}{(u+v)^2} .
\end{align*}
The numerator of the above equals
\begin{align*}
& w(u+v)^3-c \Delta w(u+w) v-c \Delta w \sqrt{u v(u+w)(v-w)} \\
\geq & 2 c_1 \Delta w\Big\{(u+v)^2-\frac{c(u+w) v}{2 c_1}-\frac{c \sqrt{u v(u+w)(v-w)}}{2 c_1}\Big\} \\
\geq & 2 c_1 \Delta w\Big\{\Big(1-c /\Big(2 c_1\Big)\Big)(u+v)^2-2^{-1 / 2} c / c_1 u v\Big\}>0
\end{align*}
as long as
$$
c<\frac{\sqrt{2} c_1}{4+1 /\Big(\sqrt{2} c_1\Big)} .
$$
{\bf{Case (ii)}}. Let $g=c_1 \Delta / 16$. We can write
$$
\widetilde{f}_{\eta_k}^{(s, e]}\Big(z_0\Big)=a \sqrt{\frac{e-s}{\Big(\eta_k-s\Big)\Big(e-\eta_k\Big)}}, \quad \widetilde{f}_{\eta_k+g}^{(s, e]}\Big(z_0\Big)=(a+g \theta) \sqrt{\frac{e-s}{\Big(e-\eta_k-g\Big)\Big(\eta_k+g-s\Big)}},
$$
where
$$
a=\sum_{j=s+1}^{\eta_k}\Big\{f_{j}\Big(z_0\Big)-\frac{1}{e-s} \sum_{j=s+1}^e f_{j}\Big(z_0\Big)\Big\}
$$
$$
\theta=\frac{a \sqrt{\Big(\eta_k+g-s\Big)\Big(e-\eta_k-g\Big)}}{g}\Big\{\frac{1}{\sqrt{\Big(\eta_k-s\Big)\Big(e-\eta_k\Big)}}-\frac{1}{\Big(\eta_k+g-s\Big)\Big(e-\eta_k-g\Big)}+\frac{b}{a \sqrt{e-s}}\Big\},
$$
and $b=\widetilde{f}_{\eta_k+g}^{(s, e]}\Big(z_0\Big)-\widetilde{f}_{\eta_k}^{(s, e]}\Big(z_0\Big)$.
To ease notation, let $d-\eta_k=l \leq g / 2, N_1=\eta_k-s$ and $N_2=e-\eta_k-g$. We have
\begin{equation}
\label{lem10-5}
E_l=\widetilde{f}_{\eta_k}^{(s, e]}\Big(z_0\Big)-\widetilde{f}_{d}^{(s, e]}\Big(z_0\Big)=E_{1 l}\Big(1+E_{2 l}\Big)+E_{3 l},
\end{equation}
where
\begin{equation*}
E_{1 l}=\frac{a l(g-l) \sqrt{e-s}}{\sqrt{N_1\Big(N_2+g\Big)} \sqrt{\Big(N_1+l\Big)\Big(g+N_2-l\Big)}\Big(\sqrt{\Big(N_1+l\Big)\Big(g+N_2-l\Big)}+\sqrt{N_1\Big(g+N_2\Big)}\Big)},
\end{equation*}
\begin{equation*}
E_{2 l}=\frac{\Big(N_2-N_1\Big)\Big(N_2-N_1-l\Big)}{\Big(\sqrt{\Big(N_1+l\Big)\Big(g+N_2-l\Big)}+\sqrt{\Big(N_1+g\Big) N_2}\Big)\Big(\sqrt{N_1\Big(g+N_2\Big)}+\sqrt{\Big(N_1+g\Big) N_2}\Big)},
\end{equation*}
and
\begin{equation*}
E_{3 l}=-\frac{b l}{g} \sqrt{\frac{\Big(N_1+g\Big) N_2}{\Big(N_1+l\Big)\Big(g+N_2-l\Big)}} .
\end{equation*}
Next, we notice that $g-l \geq c_1 \Delta / 32$. It holds that
\begin{equation}
\label{lem10-6}
E_{1 l} \geq c_{1 l}\Big|d-\eta_k\Big| \Delta \widetilde{f}_{\eta_k}^{(s, e]}\Big(z_0\Big)(e-s)^{-2},
\end{equation}
where $c_{1 l}>0$ is a sufficiently small constant depending on $c_1$. As for $E_{2 l}$, due to \eqref{lem10-4}, we have
\begin{equation}
\label{lem10-7}
E_{2 l} \geq-1 / 2 \text {. }
\end{equation}
As for $E_{3 l}$, we have
\begin{align}
\label{lem10-8}
E_{3 l} & \geq-c_{3 l, 1} b\Big|d-\eta_k\Big|(e-s) \Delta^{-2} \geq-c_{3 l, 2} b\Big|d-\eta_k\Big| \Delta^{-3}(e-s)^{3 / 2} \widetilde{f}_{\eta_k}^{(s, e]}\Big(z_0\Big) \kappa^{-1} \\
& \geq-c_{1 l} / 2\Big|d-\eta_k\Big| \Delta \widetilde{f}_{\eta_k}^{(s, e]}\Big(z_0\Big)(e-s)^{-2},
\end{align}
where the second inequality follows from \eqref{lem10-2} and the third inequality follows from \eqref{lem10-3}, $c_{3 l, 1}, c_{3 l, 2}>$
0 are sufficiently small constants, depending on all the other absolute constants.
Combining \eqref{lem10-5}, \eqref{lem10-6}, \eqref{lem10-7} and \eqref{lem10-8}, we have
\begin{equation}
\label{lem10-9}
    \widetilde{f}_{\eta_k}^{(s, e]}\Big(z_0\Big)-\widetilde{f}_{d}^{(s, e]}\Big(z_0\Big) \geq c\Big|d-\eta_k\Big| \Delta \widetilde{f}_{\eta_k}^{(s, e]}\Big(z_0\Big)(e-s)^{-2},
\end{equation}
where $c>0$ is a sufficiently small constant.
In view of \eqref{lem10-10} and \eqref{lem10-9}, the proof is complete.
\end{proof}
Consider the following events
\begin{align*}
& 
\mathcal A ( (s,e],  \rho, \gamma   )  =
 \bigg\{ \max_{t=s+\rho+1,\ldots,e-\rho }\sup _{z \in \mathbb{R}^p} |\widetilde F_{t,h}^{s,e }(z)  - \widetilde f ^{s,e}_t(z)  | \le \gamma     \bigg\} ;
\\
& 
\mathcal B (r, \rho, \gamma    ) = \bigg\{\max_{N=\rho,\ldots, T-r   }\sup _{z \in \mathbb{R}^p} \bigg| \frac{1}{\sqrt  N} \sum_{t=r+1}^{r+ N} (F_{t,h} -f _t )\bigg| \le  \gamma    \bigg\}
\\
&\hspace{2cm}
\bigcup \bigg\{\max_{ N=\rho,\ldots ,r    } \bigg| \frac{1}{\sqrt N}\sum_{t=r-N+1}^r \sup _{z \in \mathbb{R}^p}(F_{t,h}(z) -f _t(z) )\bigg| \le   \gamma\bigg\}.
\end{align*}
\begin{lemma}\label{final-bound}
Suppose  \Cref{asum-1} holds.  Let  $[s ,e ]$ be an subinterval of $[1, T]$  and contain at least one change point $\eta_r$  
with    
$\min\{ \eta_{r} -s  , e  -\eta_{r} \}\ge cT $ 
for some constant $ c>0$. 
Let 
$\kse= \max\{\kappa_p: \min\{ \eta_p -s  , e  -\eta_p \} \ge cT \}$.
Let
$$b \in \arg \max_{t=s+\rho,\ldots,   e-\rho }\vert\vert\widetilde F  _{t,h}^{s,e}  \vert\vert_{L_2} .$$
For some $c_1>0$, $\lambda>0$ and $\delta>0$, suppose that   the following events  hold
\begin{align}
&
\mathcal A ( (s,e], \rho,  \gamma    ),  \label{eq:wbs noise 1}
\\
&    
\mathcal B (s, \rho,\gamma  ) \cup   \mathcal B (e ,\rho,    \gamma     )  \cup\bigcup_{\eta      \in    \{ \eta_k\}_{k=1}^K}       \mathcal B (\eta,\rho,    \gamma  )    \label{eq:wbs noise 2}
\end{align}
and that 
\begin{equation}
\label{con-1}
\max_{t=s+\rho,\ldots,   e-\rho }\vert\vert\widetilde F _{t,h}^{s,e}  \vert\vert_{L_2} = \vert\vert\widetilde F   _{b,h}^{s,e}  \vert\vert_{L_2}  \ge c_1 \kse \sqrt{T } 
\end{equation}
If there exists a sufficiently small $c_2 > 0$ such that
\begin{equation}\label{con-2}
\gamma   \le c_2\kse\sqrt T   \quad \text{and that}\quad  \rho \le  c_2T ,
\end{equation}
then there exists a change point $\eta_{k} \in (s, e)$  such that 
\[\min \{e-\eta_k,\eta_k-s\}  > c_3 T    \quad \text{and} \quad 
|\eta_{k} -b |\le  C_3\max\{ \gamma     ^2\kappa_k^{-2} , \rho\},
\]
where $c_3$ is some sufficiently small constant independent of $T$.  
\end{lemma}

\begin{proof}
  Let $z_1 \in \arg \max _{z \in \mathbb{R}^p}\Big|\widetilde{f}_{b}^{(s, e]}(z)\Big|$. Without loss of generality, assume that $\widetilde{f}_{b}^{(s, e]}\Big(z_1\Big)>0$ and that $\widetilde{f}_{b}^{(s, e]}\Big(z_1\Big)$ as a function of $t$ is locally decreasing at $b$. Observe that there has to be a change point $\eta_k \in(s, b)$, or otherwise $\widetilde{f}_{b}^{(s, e]}\Big(z_1\Big)>0$ implies that $\widetilde{f}_{t}^{(s, e]}\Big(z_1\Big)$ is decreasing, as a consequence of \Cref{max-pop}.
Thus, there exists a change point $\eta_k \in(s, b)$ satisfying that
\begin{align}
\label{lem11-1}
\sup _{z \in \mathbb{R}^p}\Big|\widetilde{f}_{\eta_k}^{(s, e]}(z)\Big| & \geq\Big|\widetilde{f}_{\eta_k}^{(s, e]}\Big(z_1\Big)\Big|>\Big|\widetilde{f}_{b}^{(s, e]}\Big(z_1\Big)\Big| \geq \sup _{z \in \mathbb{R}^p}\Big|\widetilde{F}_b^{(s, e]}(z)\Big|-\gamma\geq c \kappa_k \sqrt{\Delta}
\end{align}
where the second inequality follows from \Cref{max-pop}, the third because of the good event $\mathcal{A}$, and fourth inequalities by \eqref{con-1} and \Cref{assume: model assumption}, and $c>0$ is an absolute constant.
Observe that $(s, e)$ has to contain at least one change point or otherwise $\sup _{z \in \mathbb{R}}\Big|\widetilde{f}_{\eta_k}^{(s, e]}(z)\Big|=0$ which contradicts \eqref{lem11-1}.

{\bf{Step 1}}. In this step, we are to show that
\begin{equation}
\label{lem11-2}
    \min \Big\{\eta_k-s, e-\eta_k\Big\} \geq \min \Big\{1, c_1^2\Big\} \Delta / 16
\end{equation}
Suppose that $\eta_k$ is the only change point in $(s, e)$. Then \eqref{lem11-2} must hold or otherwise it follows from \eqref{eq:one change point 1d cusum} that
$$
\sup _{z \in \mathbb{R}^p}\Big|\widetilde{f}_{\eta_k}^{s, e}(z)\Big| \leq \kappa_k \frac{c_1 \sqrt{\Delta}}{4},
$$
which contradicts \eqref{lem11-1}.

Suppose $(s, e)$ contains at least two change points. Then arguing by contradiction, if $\eta_k-s<$ $\min \Big\{1, c_1^2\Big\} \Delta / 16$, it must be the cast that $\eta_k$ is the left most change point in $(s, e)$. Therefore
\begin{align}
\sup _{z \in \mathbb{R}^p}\Big|\widetilde{f}_{\eta_k}^{s, e}(z)\Big| & \leq c_1 / 4 \sup _{z \in \mathbb{R}^p}\Big|\widetilde{f}_{\eta_{k+1}}^{s, e}(z)\Big|+2 \kappa_k \sqrt{\eta_k-s} \\
&<c_1 / 4 \max _{s+\rho<t<e-\rho} \sup _{z \in \mathbb{R}^p}\Big|\widetilde{f}_{t}^{s, e}(z)\Big|+2\sqrt{\Delta} \kappa_k \\
& \leq c_1 / 4 \max _{s+\rho<t<e-\rho} \sup _{z \in \mathbb{R}^p}\Big|\widetilde{F}_t^{s, e}(z)\Big|+c_1 / 4 \gamma+2\sqrt{\Delta} \kappa_k  \\
& \leq \sup _{z \in \mathbb{R}^p}\Big|\widetilde{F}_b^{s, e}(z)\Big|-\gamma
\end{align}
where the first inequality follows from \Cref{lemma20}, the second follows from the assumption of $\eta_k-s$, the third from the definition of the event $\mathcal{A}$ and the last from \eqref{con-1} and \Cref{assume: model assumption}. The last display contradicts \eqref{lem11-1}, thus \eqref{lem11-2} must hold.
\\
{\bf{Step 2}}. Let
\begin{equation*}
    z_0 \in \underset{z \in \mathbb{R}^p}{\arg \max }\Big|\widetilde{f}_{\eta_k}^{s, e}(z)\Big| .
\end{equation*}
It follows from \Cref{in-up-bound} that there exits $d \in\Big(\eta_k, \eta_k+c_1 \Delta / 32\Big)$ such that
\begin{equation}
\label{lem11-3}
    \widetilde{f}_{\eta_k}^{s, e}\Big(z_0\Big)-\widetilde{f}_{d}^{s, e}\Big(z_0\Big) \geq 2 \gamma.
\end{equation}
We claim that $b \in\Big(\eta_k, d\Big) \subset\Big(\eta_k, \eta_k+c_1 \Delta / 16\Big)$. By contradiction, suppose that $b \geq d$. Then
\begin{equation}
\label{lem11-4}
   \widetilde{f}_{b}^{s, e}\Big(z_0\Big) \leq \widetilde{f}_{d}^{s, e}\Big(z_0\Big) \leq \max _{s<t<e} \sup _{z \in \mathbb{R}^p}\Big|\widetilde{f}_{t}^{s, e}(z)\Big|-2 \gamma \leq \sup _{z \in \mathbb{R}^p}\Big|\widetilde{F}_b^{s, e}(z)\Big|-\gamma, 
\end{equation}
where the first inequality follows from \Cref{max-pop}, the second follows from \eqref{lem11-3} and the third follows from the definition of the event $\mathcal{A}$. Note that \eqref{lem11-4} is a contradiction to the bound in \eqref{lem11-1}, therefore we have $b \in\Big(\eta_k, \eta_k+c_1 \Delta / 32\Big)$.

{\bf{Step 3}}. Let
\begin{equation*}
j^* \in \underset{j=1, \ldots, T}{\arg \max }\Big|\widetilde{F}_b^{s, e}(X(j))\Big|, \quad f^{s, e}=\Big(f_{s+1}\Big(X\Big(j^*\Big)\Big), \ldots, f_{e}\Big(X\Big(j^*\Big)\Big)\Big)^{\top} \in \mathbb{R}^{(e-s)}
\end{equation*}
and
\begin{equation*}
F^{s, e}=\Big(\frac{1}{h^p} k\Big(\frac{X\Big(j^*\Big)-X(s)}{h}\Big), \ldots, \frac{1}{h^p} k\Big(\frac{X\Big(j^*\Big)-X(e)}{h}\Big)\Big) \in \mathbb{R}^{(e-s)} .
\end{equation*}
By the definition of $b$, it holds that
\begin{equation*}
\Big\|F^{s, e}-\mathcal{P}_b^{s, e}\Big(F^{s, e}\Big)\Big\|^2 \leq\Big\|F^{s, e}-\mathcal{P}_{\eta_k}^{s, e}\Big(F^{s, e}\Big)\Big\|^2 \leq\Big\|F^{s, e}-\mathcal{P}_{\eta_k}^{s, e}\Big(f^{s, e}\Big)\Big\|^2
\end{equation*}
where the operator $\mathcal{P}^{s, e}(\cdot)$ is defined in Lemma 21 in \cite{wang2020univariate}. For the sake of contradiction, throughout the rest of this argument suppose that, for some sufficiently large constant $C_3>0$ to be specified,
\begin{equation}
\label{lem11-5}
\eta_k+C_3 \lambda_{\mathcal{A}}^2 \kappa_k^{-2}<b .
\end{equation}
We will show that this leads to the bound
\begin{equation}
\label{lem11-6}
\Big\|F^{s, e}-\mathcal{P}_b^{s, e}\Big(F^{s, e}\Big)\Big\|^2>\Big\|F^{s, e}-\mathcal{P}_{\eta_k}^{s, e}\Big(f^{s, e}\Big)\Big\|^2,
\end{equation}
which is a contradiction. If we can show that
\begin{equation}
\label{lem11-7}
2\Big\langle F^{s, e}-f^{s, e}, \mathcal{P}_b^{s, e}\Big(F^{s, e}\Big)-\mathcal{P}_{\eta_k}^{s, e}\Big(f^{s, e}\Big)\Big\rangle<\Big\|f^{s, e}-\mathcal{P}_b^{s, e}\Big(f^{s, e}\Big)\Big\|^2-\Big\|f^{s, e}-\mathcal{P}_{\eta_k}^{s, e}\Big(f^{s, e}\Big)\Big\|^2,
\end{equation}
then \eqref{lem11-6} holds.
To derive \eqref{lem11-7} from \eqref{lem11-5}, we first note that $\min \Big\{e-\eta_k, \eta_k-s\Big\} \geq \min \Big\{1, c_1^2\Big\} \Delta / 16$ and that $\Big|b-\eta_k\Big| \leq c_1 \Delta / 32$ implies that
\begin{equation}
\min \{e-b, b-s\} \geq \min \Big\{1, c_1^2\Big\} \Delta / 16-c_1 \Delta / 32 \geq \min \Big\{1, c_1^2\Big\} \Delta / 32
\end{equation}
As for the right-hand side of \eqref{lem11-7}, we have
\begin{align}
\label{lem11-8}
&\Big\|f^{s, e}-\mathcal{P}_b^{s, e}\Big(f^{s, e}\Big)\Big\|^2-\Big\|f^{s, e}-\mathcal{P}_{\eta_k}^{s, e}\Big(f^{s, e}\Big)\Big\|^2=\Big(\widetilde{f}_{\eta_k}^{s, e}\Big(X\Big(j^*\Big)\Big)\Big)^2-\Big(\widetilde{f}_{b}^{s, e}\Big(X\Big(j^*\Big)\Big)\Big)^2 \\
\geq &\Big(\widetilde{f}_{\eta_{k}}^{s, e}\Big(X\Big(j^*\Big)\Big)-\widetilde{f}_{b}^{s, e}\Big(X\Big(j^*\Big)\Big)\Big)\Big|\widetilde{f}_{\eta_k}^{s, e}\Big(X\Big(j^*\Big)\Big)\Big|
\end{align}
On the event $\mathcal{A} \cap \mathcal{B}$, we are to use \Cref{in-up-bound}. Note that \eqref{lem10-2} holds due to the fact that here we have
\begin{equation}
\label{lem11-9}
    \Big|\widetilde{f}_{\eta_k}^{s, e}\Big(X\Big(j^*\Big)\Big)\Big| \geq\Big|\widetilde{f}_{b}^{s, e}\Big(X\Big(j^*\Big)\Big)\Big| \geq\Big|\widetilde{F}_b^{s, e}\Big(X\Big(j^*\Big)\Big)\Big|-\gamma \geq c_1 \kappa_k \sqrt{\Delta}-\gamma \geq\Big(c_1\Big) / 2 \kappa_k \sqrt{\Delta},
\end{equation}
where the first inequality follows from the fact that $\eta_k$ is a true change point, the second inequality holds due to the event $\mathcal{A}$, the third inequality follows from \eqref{con-1}, and the final inequality follows from \eqref{con-2}. Towards this end, it follows from \Cref{in-up-bound} that
\begin{equation}
\label{lem11-10}
\Big.\Big.\Big.\mid \widetilde{f}_{\eta_k}^{s, e}\Big(X\Big(j^*\Big)\Big)\Big|-| \widetilde{f}_{b}^{s, e}\Big(X\Big(j^*\Big)\Big)\Big||>c| b-\eta_k|\Delta| \widetilde{f}_{\eta_k}^{s, e}\Big(X\Big(j^*\Big)\Big)\Big) \mid(e-s)^{-2} .
\end{equation}
Combining \eqref{lem11-8}, \eqref{lem11-9} and \eqref{lem11-10}, we have
\begin{equation}
\label{lem11-14}
\Big\|f^{s, e}-\mathcal{P}_b^{s, e}\Big(f^{s, e}\Big)\Big\|^2-\Big\|f^{s, e}-\mathcal{P}_{\eta_k}^{s, e}\Big(f^{s, e}\Big)\Big\|^2 \geq \frac{c c_1^2}{4} \Delta^2 \kappa_k \mathcal{A}^2(e-s)^{-2}\Big|b-\eta_k\Big| .
\end{equation}
The left-hand side of \eqref{lem11-7} can be decomposed as follows.
\begin{align}
\label{lem11-15}
& 2\Big\langle F^{s, e}-f^{s, e}, \mathcal{P}_b^{s, e}\Big(F^{s, e}\Big)-\mathcal{P}_{\eta_k}^{s, e}\Big(f^{s, e}\Big)\Big\rangle \\
=& 2\Big\langle F^{s, e}-f^{s, e}, \mathcal{P}_b^{s, e}\Big(F^{s, e}\Big)-\mathcal{P}_b^{s, e}\Big(f^{s, e}\Big)\Big\rangle+2\Big\langle Y^{s, e}-f^{s, e}, \mathcal{P}_b^{s, e}\Big(f^{s, e}\Big)-\mathcal{P}_{\eta_k}^{s, e}\Big(f^{s, e}\Big)\Big\rangle \\
=&(I)+2\Big(\sum_{i=1}^{\eta_k-s}+\sum_{i=\eta_k-s+1}^{b-s}+\sum_{i=b-s+1}^{e-s}\Big)\Big(F^{s, e}-f^{s, e}\Big)_i\Big(\mathcal{P}_b^{s, e}\Big(f^{s, e}\Big)-\mathcal{P}_{\eta_k}^{s, e}\Big(f^{s, e}\Big)\Big)_i \\
=&(I)+(I I .1)+(I I .2)+(I I .3) .
\end{align}
As for the term (I), we have
\begin{equation}
    \label{lem11-16}
    (I) \leq 2 \gamma^2 \text {. }
\end{equation}
As for the term (II.1), we have
\begin{equation*}
(I I .1)=2 \sqrt{\eta_k-s}\Big\{\frac{1}{\sqrt{\eta_k-s}} \sum_{i=1}^{\eta_k-s}\Big(F^{s, e}-f^{s, e}\Big)_i\Big\}\Big\{\frac{1}{b-s} \sum_{i=1}^{b-s}\Big(f^{s, e}\Big)_i-\frac{1}{\eta_k-s} \sum_{i=1}^{\eta_k-s}\Big(f^{s, e}\Big)_i\Big\} .
\end{equation*}
In addition, it holds that
\begin{align*}
&\quad\Big|\frac{1}{b-s} \sum_{i=1}^{b-s}\Big(f^{s, e}\Big)_i-\frac{1}{\eta_k-s} \sum_{i=1}^{\eta_k-s}\Big(f^{s, e}\Big)_i\Big|=\frac{b-\eta_k}{b-s}\Big|-\frac{1}{\eta_k-s} \sum_{i=1}^{\eta_k-s} f_{i}\Big(X\Big(j^*\Big)\Big)+f_{\eta_{k+1}}\Big(X\Big(j^*\Big)\Big)\Big| \\
&\leq \frac{b-\eta_k}{b-s}\Big(C_R+1\Big) \kappa_{s_0, e_0}^{\max },
\end{align*}
where the inequality is followed by \Cref{lem-2-uc}. Combining with the good events,
\begin{align}
\label{lem11-11}
(I I .1) & \leq 2 \sqrt{\eta_k-s} \frac{b-\eta_k}{b-s}\Big(C_R+1\Big) \kappa_{s_0, e_0}^{\max } \gamma \\
& \leq 2 \frac{4}{\min \Big\{1, c_1^2\Big\}} \Delta^{-1 / 2} \gamma\Big|b-\eta_k\Big|\Big(C_R+1\Big) \kappa_{s_0, e_0}^{\max }
\end{align}
As for the term (II.2), it holds that
\begin{equation}
\label{lem11-12}
(I I .2) \leq 2 \sqrt{\Big|b-\eta_k\Big|} \gamma\Big(2 C_R+3\Big) \kappa_{s_0, e_0}^{\max }
\end{equation}
As for the term (II.3), it holds that
\begin{equation}
\label{lem11-13}
(I I .3) \leq 2 \frac{4}{\min \Big\{1, c_1^2\Big\}} \Delta^{-1 / 2} \gamma\Big|b-\eta_k\Big|\Big(C_R+1\Big) \kappa_{s_0, e_0}^{\max }
\end{equation}
Therefore, combining \eqref{lem11-11}, \eqref{lem11-12}, \eqref{lem11-13}, \eqref{lem11-14}, \eqref{lem11-15} and \eqref{lem11-16}, we have that \eqref{lem11-7} holds if
\begin{equation}
\Delta^2 \kappa_k^2(e-s)^{-2}\Big|b-\eta_k\Big| \gtrsim \max \Big\{\gamma^2, \Delta^{-1 / 2} \gamma\Big|b-\eta_k\Big| \kappa_k, \sqrt{\Big|b-\eta_k\Big|} \gamma \kappa_k\Big\}
\end{equation}
The second inequality holds due to \Cref{assume-snr}, the third inequality holds due to \eqref{lem11-4} and the first inequality is a consequence of the third inequality and \Cref{assume-snr}.
\end{proof}


 \end{document}